\documentclass[a4paper,10pt]{article}
\usepackage{graphicx}
\usepackage{rotating}

\usepackage[latin1]{inputenc}

\title{On a variant of random homogenization theory: convergence of the
  residual process and approximation of the homogenized coefficients}

\author{Fr\'ed\'eric Legoll and Florian Thomines
\\
{\footnotesize Laboratoire Navier, \'Ecole Nationale des Ponts et
  Chauss\'ees, Universit\'e Paris-Est, 6 et 8 avenue
  Blaise Pascal,}   
\\ 
{\footnotesize 77455 Marne-La-Vall\'ee Cedex 2, France}\\
{\footnotesize and
} 
\\
{\footnotesize INRIA Rocquencourt, MICMAC team-project, Domaine de
  Voluceau,
} \\ {\footnotesize 
B.P. 105, 78153 Le Chesnay Cedex, France}
\\
{\footnotesize \tt legoll@lami.enpc.fr, florian.thomines@enpc.fr} 
}
\date{\today}

\usepackage{amsfonts}
\usepackage{amsmath,amssymb,amsthm}

\usepackage[english]{babel}

\newcommand{\RR}{\mathbb{R}}
\newcommand{\NN}{\mathbb{N}}
\newcommand{\ZZ}{\mathbb{Z}}
\newcommand{\PP}{\mathbb{P}}
\newcommand{\EE}{\mathbb{E}}
\newcommand{\Id}{\text{Id}}
\newcommand{\Var}{\mathbb{V}\mbox{ar}}
\newcommand{\Cov}{\mathbb{C}\mbox{ov}}
\newcommand{\esssup}[1]{\mathop{\operatorname{Ess Sup}}_{#1}}
\newcommand{\essinf}[1]{\mathop{\operatorname{Ess Inf}}_{#1}}
\newcommand{\dps}{\displaystyle}
\newcommand{\eps}{\varepsilon}

\def\longrightharpoonup{\relbar\joinrel\rightharpoonup}

\newtheorem{theo}{Theorem}
\newtheorem{remark}[theo]{Remark}
\newtheorem{lemme}[theo]{Lemma}

\begin{document}
\maketitle

\begin{abstract}
We consider the variant of stochastic homogenization theory introduced
in~\cite{Blanc2006,Blanc2007}. The equation under consideration is a
standard linear elliptic equation in divergence form, where the highly
oscillatory coefficient is the composition of a periodic matrix with a
stochastic diffeomorphism. The homogenized limit of this problem has
been identified in~\cite{Blanc2006}.

We first establish, in the one-dimensional case, a convergence result
(with an explicit rate) on the 
residual process, defined as the difference between the solution to the
highly oscillatory problem and the solution to the homogenized problem. 

We next return to the multidimensional situation. As often in random
homogenization, the homogenized matrix is defined from a so-called
corrector function, which is the solution to a problem set on the
entire space. We describe and prove the almost sure convergence of an
approximation strategy based on truncated versions of the corrector
problem. 
\end{abstract}


\section{Introduction}







Homogenization theory for linear second-order elliptic equations with
highly oscillatory coefficients is a well developed topic. In the
periodic case, the homogenized problem is known, and convergence rates
of the oscillatory solution (denoted $u^\eps$) towards the homogenized
solution $u_\star$ have been
obtained. 

The situation is less clear in the random (say stationary ergodic)
setting. The convergence of $u^\eps(\cdot,\omega)$ to some deterministic
$u_\star$ is a classical result. However, rates of convergence are much more
difficult to obtain. A central difficulty in stochastic homogenization
is that the corrector problem, that needs to be solved to next compute
the homogenized matrix, is set on the entire space (in contrast with the
periodic case, where it is set on the periodic cell). This induces many
theoretical and practical difficulties. 

In what follows, we are interested in the problem
\begin{equation}
\label{eq:edp_intro}
-\mbox{div}\left[A\left(\frac{x}{\eps},\omega\right)
   \nabla u^\eps(x,\omega) \right] =f(x) \mbox{ in $\mathcal{D}$}, 
\qquad
u^\eps(\cdot,\omega) =0 \mbox{ on $\partial \mathcal{D}$},
\end{equation}
where the random matrix $A$ satisfies standard coercivity and
boundedness properties 
(and some structure assumptions that we detail below),
$\mathcal{D}$ is an open bounded set of $\RR^d$ and $f \in
L^2(\mathcal{D})$.

The analysis of the residual, that we define as the difference
between the oscillatory solution $u^\eps$ and the homogenized solution
$u_\star$, 
was first taken up in~\cite{bourgeat_residu}, and next complemented
in~\cite{bal_residu}. Both studies consider the equation
$\dps -\frac{d}{dx}\left[
  a\left(\frac{x}{\eps},\omega\right)\frac{d u^\eps}{dx}\right]=f(x)$ in
the one-dimensional setting,
where $\dps a\left(x,\omega\right)$ is a random
stationary process. The behavior, when $\eps \to 0$, of the residual
$u^\eps(x,\omega) - u_\star(x)$ turns out to depend on the asymptotic
behavior of the correlation function of the conductivity coefficient 
$\dps \eta(x) := \Cov(a(0,\cdot),
a(x,\cdot))$. In~\cite{bourgeat_residu}, the case of small correlation
lengths is studied, which amounts to assuming that $\dps \eta(x)
\sim_{x\to \infty} x^{-\alpha}$ with $\alpha >1$. The correlation
function is thus integrable. In that case, when $\eps \to 0$,
the random process $\dps \frac{u^\eps(x,\omega) -
  u_\star(x)}{\sqrt{\eps}}$ converges in distribution to a Gaussian
random process. The case of long
correlation lengths, namely when $\dps \eta(x) \sim_{x\to \infty}
x^{-\alpha}$ for some $0 < \alpha < 1$, is studied
in~\cite{bal_residu}, where it is shown that the random process
$\dps \frac{u^\eps(x,\omega) - u_\star(x)}{\eps^{\alpha/2}}$ converges in
distribution to a Gaussian random process (defined using a fractional
Brownian motion).
This result shows that the rate of convergence of $u^\eps$ to
$u_\star$ can be as slow as $\eps^{\alpha/2}$ for any $\alpha > 0$, without
any further assumptions on the stationary process $a$. 

Of course, in both works, the one-dimensional setting allows 
to get some analytical expression for the residual. In turn, the
analysis of the asymptotic behavior of the residual performed
in~\cite{bourgeat_residu,bal_residu} relies on this analytical
expression. 
In higher dimensions, the case of the equation $\dps -\Delta u^\eps
+ q\left(\frac{x}{\eps},\omega\right) u^\eps = f(x)$ has been studied
in~\cite{bal_correcteur_multiD}.

\medskip

Our first aim here is to study a similar question for a variant of the
classical stochastic homogenization theory.  
We consider in the sequel the following problem, which has been
introduced in~\cite{Blanc2006} and further studied in~\cite{Blanc2007}:
\begin{equation}
\label{PB:BLL_intro}
-\mbox{div}\left[A_{\rm
     per}\left(\phi^{-1}\left(\frac{x}{\eps},\omega\right)\right)
   \nabla u^\eps(x,\omega) \right] =f(x) \mbox{ in $\mathcal{D}$}, 
\qquad
u^\eps(\cdot,\omega) =0 \mbox{ on $\partial \mathcal{D}$},
\end{equation}
where $\phi$ is almost surely a diffeomorphism from $\RR^d$
to $\RR^d$ with some stationary properties and $A_{\rm per}$ is a
$\ZZ^d$-periodic matrix, that satisfies the classical coercivity and
boundedness properties (see precise assumptions in
Section~\ref{sec:bll_dD} below).
This model is appropriate
to represent a periodic, ideal material, that is randomly deformed
(think of fibers in a composite material that are placed at a random
position, rather than on a perfect, periodic lattice). 
In~\cite{Blanc2006}, it is shown 
that the solution $u^\eps(\cdot,\omega)$ to the above
problem 
converges
as $\eps$ goes to $0$ to $u_\star$, solution to some homogenized
problem (see Section~\ref{sec:homog_bll} below). 
In the sequel, we aim at obtaining the rate of convergence of $u^\eps$
to $u_\star$, in the one dimensional setting. We make below an assumption
on the random diffeomorphism which implies that our setting is close to
the one studied in~\cite{bourgeat_residu} (rather than that studied
in~\cite{bal_residu}). Under this assumption, we show that the random
process $\dps \frac{u^\eps(x,\omega) -
  u_\star(x)}{\sqrt{\eps}}$ converges in distribution to a Gaussian
random process that we completely characterize (see
Section~\ref{sec:main1}, Theorem~\ref{theo4}).

\bigskip

We next turn to a question of different a nature. 
As pointed out above, the
homogenized matrix $A^\star$ associated to~\eqref{PB:BLL_intro} depends
on the solution of the corrector problem, 
which is set on the {\em entire} space. Computing an approximation of
$A^\star$ is thus, in practice, a challenging question. A standard
strategy is to consider the corrector problem on a large, but bounded
domain $Q_N$, supplemented with (say periodic) boundary conditions. An
approximation of the {\em exact} corrector is thus computed, from which 
an approximate homogenized matrix $A^\star_N(\omega)$ is inferred. As
a by-product of working on a bounded domain, the approximate
homogenized matrix is random. 
In the classical random homogenization setting (that
is~\eqref{eq:edp_intro} where $A$ is a stationary matrix), the
convergence (and its rate) of 
$A^\star_N(\omega)$ to $A^\star$ has been studied
in~\cite{bourgeat_ihp}, using some previous approximation
results~\cite{yurinski}. It is shown there that 
$A_N^\star(\omega)$ almost surely converges to $A^\star$, and that 
$\EE \left[ \left|
    A_N^\star - A^\star \right|^2 \right]$ converges to 0 as
$N^{-\alpha}$, for some $\alpha > 0$ which implicitly depends on the
mixing properties of the random coefficient $A$ of the
equation~\eqref{eq:edp_intro}. It is expected 
that, depending on the properties of that random coefficient, $\alpha$
can be arbitrary small. 

In this work, we consider the above
variant~\eqref{PB:BLL_intro} of the classical random homogenization setting. 
We describe a strategy (originally introduced in~\cite{cras_ronan}) to
approximate $A^\star$ which is based, as in
the classical setting, on solving the corrector problems on bounded
domains $Q_N$. We prove here the convergence of this approach (see
Section~\ref{sec:pres_approx}, Theorem~\ref{conv:AN}). 

\medskip

Our article is articulated as follows. In Section~\ref{sec:homog_bll},
we present in details the variant of the classical random homogenization
introduced in~\cite{Blanc2006,Blanc2007}. We next present in
Section~\ref{sec:main} our two main results, first on the residual
process in dimension one (see Section~\ref{sec:main1} and
Theorem~\ref{theo4}), second on a practical approximation of the
homogenized matrix in dimension $d \geq 2$ (see Section~\ref{sec:pres_approx} and
Theorem~\ref{conv:AN}). The subsequent two sections are devoted to the
proof of Theorem~\ref{theo4}. The actual proof is performed in
Section~\ref{sec:asymp}, and needs some technical results which are
proved in Section~\ref{sec:tech}. Our final section,
Section~\ref{sec:proofAN}, collects the proof of Theorem~\ref{conv:AN}.

\section{A variant of the classical random homogenization}
\label{sec:homog_bll}

To begin with, we introduce the basic
setting of stochastic homogenization we will employ. We refer
to~\cite{engquist-souganidis} for a general, numerically oriented 
presentation, and to~\cite{blp,cd,jikov} for classical textbooks. 
We also refer to~\cite{Blanc2006,Blanc2007} for a 
presentation of our particular setting. 
Throughout this article, $(\Omega,\mathcal{F},\PP)$ is a
probability space and we denote by $\dps \EE(X) =
\int_\Omega X(\omega) d\PP(\omega)$ the expectation value of any
random variable $X\in L^1(\Omega, d\PP)$. For any fixed 
$d\in \NN^\star$ (the ambient physical dimension), we assume that
the group $(\ZZ^d,+)$ acts on 
$\Omega$. We denote by $(\tau_k)_{k\in \ZZ^d}$ this action, and
assume that it preserves the measure $\PP$, that is, for all
$k\in \ZZ^d$ and all $B \in {\cal F}$, $\PP(\tau_k B)
= \PP(B)$. We assume that the action $\tau$ is {\em ergodic}, that is,
if $B \in {\cal F}$ is such that $\tau_k B = B$ for any $k \in
\ZZ^d$, then $\PP(B) = 0$ or 1. 
In addition, we define the following notion of (discrete) stationarity
(see~\cite{Blanc2006,Blanc2007}): any  
$F \in L^1_{\rm loc}\left(\RR^d, L^1(\Omega)\right)$ is said to be
{\em stationary} if
\begin{equation}
\label{eq:stationnarite-disc}
\forall k \in \ZZ^d, \quad
F(x+k, \omega) = F(x,\tau_k\omega)
\quad
\text{almost everywhere and almost surely.}
\end{equation}
In this setting, the ergodic theorem~\cite{krengel,shiryaev,tempelman} can be
stated as follows: 
{\it Let $F\in L^\infty\left(\RR^d, L^1(\Omega)\right)$ be a stationary random
variable in the above sense. For $k = (k_1,k_2, \dots, k_d) \in \ZZ^d$,
we set $\dps |k|_\infty = \sup_{1\leq i \leq d} |k_i|$. Then
$$
\frac{1}{(2N+1)^d} \sum_{|k|_\infty\leq N} F(x,\tau_k\omega)
\mathop{\longrightarrow}_{N\rightarrow \infty}
\EE\left(F(x,\cdot)\right) \quad \mbox{in }L^\infty(\RR^d),
\mbox{ almost surely}.
$$
This implies that (denoting by $Q=(0,1)^d$ the unit cube in $\RR^d$)
$$
F\left(\frac x \eps ,\omega \right)
\mathop{\longrightharpoonup}_{\eps \rightarrow 0}^*
\EE\left(\int_Q F(x,\cdot)dx\right) \quad \mbox{in }L^\infty(\RR^d),
\mbox{ almost surely}.
$$
}

\subsection{Mathematical setting and homogenization result}
\label{sec:bll_dD}

As pointed out in the introduction, we consider in this article the following problem,
which has been introduced in~\cite{Blanc2006} and further studied
in~\cite{Blanc2007}: 
\begin{equation}
\label{PB:BLL}
-\mbox{div}\left[A_{\rm
     per}\left(\phi^{-1}\left(\frac{x}{\eps},\omega\right)\right)
   \nabla u^\eps(x,\omega) \right] =f(x) \mbox{ in $\mathcal{D}$}, 
\qquad
u^\eps(\cdot,\omega) =0 \mbox{ on $\partial \mathcal{D}$},
\end{equation}
where $\mathcal{D}$ is a bounded open set of $\RR^d$, $f \in
L^2(\mathcal{D})$, $\phi$ is almost surely a diffeomorphism from $\RR^d$
to $\RR^d$, and $A_{\rm per}$ is a $\ZZ^d$-periodic matrix, that
satisfies the classical coercivity and boundedness properties: there
exists $a^+ \geq a_- > 0$ such that 
\begin{equation}
\label{eq:hyp_a}
\forall \xi \in \RR^d, \ \ a_-|\xi|^2 \leq A_{\rm
  per}(x) \xi \cdot \xi \ \ \mbox{almost everywhere on $\RR^d$, and}
\ \ a^+ = \| A_{\rm per} \|_{L^\infty(\RR^d)} < \infty.
\end{equation}
In addition, we assume that the map $\phi(\cdot,\omega)$ satisfies 
\begin{eqnarray}
 \essinf{\omega \in \Omega, \, x \in \RR^d}\left(\det(\nabla \phi(x, \omega))\right)=\nu >0,\label{phi:cond1}\\
 \esssup{\omega \in \Omega, \, x \in \RR^d} \left| \nabla \phi(x, \omega) \right| = M < +\infty,\label{phi:cond2}\\
 \nabla \phi \mbox{ is stationary in the sense of~\eqref{eq:stationnarite-disc}.}\label{phi:cond3}
\end{eqnarray}
Assumptions~\eqref{phi:cond1} and~\eqref{phi:cond2} mean that $\phi$ is
a well-behaved diffeomorphism, uniformly in $\omega$. Note that
$A_{\rm per} \circ \phi^{-1}$ is in general not
stationary. The above setting is thus not a particular case of the
classical stationary setting. 

In~\cite{Blanc2006}, it is shown 
that, under the above conditions, $u^\eps(\cdot,\omega)$ converges
to $u_\star$ almost surely (strongly in 
$L^2({\cal D})$ and weakly in $H^1({\cal D})$) when $\eps$ goes to $0$,
where $u_\star$ 
is the solution to the homogenized problem 
\begin{equation}
\label{PB:homo}
 -\mbox{div}\left[A^\star \nabla u_\star(x) \right] = f(x) 
\mbox{ in $\mathcal{D}$}, 
\qquad
u_\star =0 \mbox{ on $\partial \mathcal{D}$}.
\end{equation}
In~\eqref{PB:homo}, the homogenized matrix coefficient $A^\star$ is
equal to 
\begin{equation}
\label{def:astar}
\forall 1 \leq i,j \leq d, \quad
A^\star_{ij} = \mbox{det}\left(\EE\left(\int_Q \nabla \phi(y,\cdot) dy
  \right)\right)^{-1}
\EE\left(\int_{\phi(Q,\cdot)} e_i^T 
A_{\rm per}\left(\phi^{-1}\left(y,\cdot\right)\right) 
\left( e_j + \nabla w_{e_j}(y,\cdot) \right) dy \right),
\end{equation}
where $Q=(0,1)^d$ and where, for all $p\in\RR^d$, $w_p$ solves the
following corrector problem:
\begin{equation}
\label{PB:correc}
 \left\{ 
\begin{array}{l}
\dps
 -\mbox{div}\left[A_{\rm per}\left(\phi^{-1}(y,\omega)\right)
   \left(p+\nabla w_p(y,\omega)\right) \right] = 0 \mbox{ in } \RR^d, 
\\
\dps w_p(y,\omega) =
 \widetilde{w}_p(\phi^{-1}(y,\omega),\omega), 
\quad \nabla \widetilde{w}_p \mbox{ is stationary in the sense
  of~\eqref{eq:stationnarite-disc},}
\\ 
\dps \EE\left(\int_{\phi(Q,\cdot)} \nabla w_p(y,\cdot) dy \right)=0.
\end{array}
\right.
\end{equation}

\subsection{The one-dimensional case}

Our first main result, presented in Section~\ref{sec:main1}, is a
convergence result in the one-dimensional case. In that setting, it is
possible to write some explicit formulas.
Choosing $\mathcal{D}=(0,1)$, the
problems~\eqref{PB:BLL} and~\eqref{PB:homo} respectively read  
\begin{equation}
\label{PB:stoch}
 -\frac{d}{dx}\left[ a_{\rm
     per}\left(\phi^{-1}\left(\frac{x}{\eps},\omega\right)\right)
   \frac{d u^\eps}{dx}(x,\omega) \right] =f(x) \mbox{ in $(0,1)$}, 
\quad
u^\eps(0,\omega)=0, \quad u^\eps(1,\omega)=0,
\end{equation}
and
\begin{equation}
\label{PB:homo-1D}
 -\frac{d}{dx}\left(a^\star \frac{d u_\star}{dx}(x) \right) = f(x)
 \mbox{ in $(0,1)$}, 
\quad
u_\star(0)=0, \quad u_\star(1)=0.
\end{equation}
The corrector problem~\eqref{PB:correc}, that reads
\begin{equation}
\label{PB:correc-1D}
\left\{ 
\begin{array}{l}
\dps
 -\frac{d}{dy} \left[a_{\rm per}\left(\phi^{-1}(y,\omega)\right) 
\left(1+ \frac{dw}{dy}(y,\omega)\right) \right] = 0 \mbox{ in } \RR, 
\\ \noalign{\vskip 5pt}
\dps w(y,\omega) =
 \widetilde{w}(\phi^{-1}(y,\omega),\omega), \quad 
\frac{d\widetilde{w}}{dy}
\mbox{ is stationary in the sense of~\eqref{eq:stationnarite-disc},}\\ 
\dps \EE\left(\int_{\phi(Q,\cdot)} \frac{dw}{dy}(y,\cdot) dy \right)=0,
\end{array}
 \right.
\end{equation}
can be analytically solved. Its solution $w$ satisfies
\begin{equation}
\label{eq:corr}
1+\frac{dw}{dy}(y,\omega) = \frac{a^\star}{a_{\rm per}(\phi^{-1}(y,\omega))},
\end{equation}
where the homogenized coefficient $a^\star$ is given by
\begin{equation}
\label{def:astar-1d}
(a^\star)^{-1} = 
\frac{1}{\EE\left(\int_0^1 \phi'(y,\cdot)dy\right)}
\EE \left( \int_0^1 \frac{\phi'(y,\cdot)}{a_{\rm per}(y)} \, dy\right).
\end{equation}
As pointed out in~\cite{Blanc2006}, we observe on~\eqref{eq:corr}
that, in the one-dimensional case, the gradient of the corrector has the
same structure as the highly oscillatory coefficient
in~\eqref{PB:stoch}: it is equal to a periodic function composed with
$\phi^{-1}$. This is not the case in dimensions $d \geq 2$, as shown
in~\cite{Blanc2006}. 

\section{Main results}
\label{sec:main}

In this article, we show the following two main results,
Theorems~\ref{theo4} and~\ref{conv:AN}.

\subsection{Residual process in dimension one}
\label{sec:main1}

Our first aim is to characterize how the residual process
$u^\eps(x,\omega)-u_\star(x)$ converges to zero, where $u^\eps$
solves~\eqref{PB:stoch} and $u_\star$ solves~\eqref{PB:homo-1D}. 
To this aim, we make the following assumptions. Let us introduce the
$1$-periodic function 
\begin{equation}
\psi(x) = \frac{1}{a_{\rm per}(x)}-\frac{1}{a^\star}
\label{def:psi}
\end{equation}
and the random variables
\begin{equation}
\label{eq:def_Y}
Y_k(\omega) = \int_k^{k+1} \psi(t) \phi'(t,\omega) dt.
\end{equation} 
As $\psi$ is periodic and $\phi'$ is stationary, the random variables $Y_k$
are identically distributed. Due to~\eqref{def:astar-1d}, we have
$$
\EE(Y_0) = \EE\left(\int_0^1 \psi(t) \phi'(t,\cdot) dt\right) =0.
$$
We furthermore assume that the random variables $Y_k$ are
independent, and hence that
\begin{equation}
\label{eq:assump_Y_iid}
\text{the variables $Y_k$ are i.i.d.}
\end{equation}
Likewise, we consider the random variables
\begin{equation}
\label{eq:def_D}
D_k(\omega) = \int_k^{k+1} \phi'(t,\omega) dt,
\end{equation} 
which are identically distributed, and make the assumption that
\begin{equation}
\label{eq:assump_D_iid}
\text{the variables $D_k$ are i.i.d.}
\end{equation}

\begin{remark}
Suppose that the derivative of the random diffeomorphism $\phi$ reads
$$
\phi'(y,\omega) = 1 + \sum\limits_{k \in \ZZ} X_k(\omega) \ 
G_{\rm per}(y) \ \mathbf{1}_{[k,k+1)}(y), 
$$
where $X_k(\omega)$ are independent and identically distributed random
variables and $G_{\rm per}$ is a $1$-periodic bounded function, such
that, for some $0 < m < 1$,
$$
| X_0(\omega) | \leq m \ \text{almost surely and} \quad
\| G_{\rm per} \|_{L^\infty(\RR)} \leq m.
$$
Then, the conditions~\eqref{phi:cond1}, \eqref{phi:cond2} and~\eqref{phi:cond3}
are satisfied with $\nu = 1-m^2>0$ and $M=1+m^2$. By construction, the
assumptions~\eqref{eq:assump_Y_iid} and~\eqref{eq:assump_D_iid} are also
fullfilled.
\end{remark}

The first main result of this article is the following theorem, the
proof of which is postponed until Section~\ref{sec:pth4}.
\begin{theo}
\label{theo4}
Assume that $a_{\rm per}$ and $\phi$
satisfy~\eqref{eq:hyp_a},~\eqref{phi:cond1},~\eqref{phi:cond2}
and~\eqref{phi:cond3}. Assume furthermore the independence
conditions~\eqref{eq:assump_Y_iid} and~\eqref{eq:assump_D_iid}. 
We consider $u^\eps$ solution to~\eqref{PB:stoch}
and $u_\star$ solution to~\eqref{PB:homo-1D}. Then the
residual process converges in distribution to a Gaussian process,
$$
\frac{u^\eps(x,\omega)-u_\star(x)}{\sqrt{\eps}} 
\overset{\mathcal{L}}{\underset{\eps\to 0}{\longrightarrow}} 
G_0(x,\omega),
$$
where
\begin{equation}
\label{eq:def_G0}
G_0(x,\omega) = 
\frac{\sqrt{\Var(Y_0)}}{\sqrt{\EE\left(\int_0^1\phi'\right)}} 
\int_0^1 K_0(x,t) \, dW_t,
\end{equation}
where $W_t$ denotes the classical Brownian motion and $K_0(x,t)$ is
given by
\begin{equation}
K_0(x,t)=\left(\mathbf{1}_{[0,x]}(t)-x \right)\left(\int_0^1
  F(s) ds - F(t) \right) 
\quad \text{with} \quad F(t) = \int_0^t f(s) \, ds.
\label{def:K0}
\end{equation}
\end{theo}

\begin{remark}
It might be possible to weaken assumptions~\eqref{eq:assump_Y_iid}
and~\eqref{eq:assump_D_iid}, and to only assume that the identically
distributed variables $Y_k$ are such that
$\dps \sum_{k\in \ZZ} \left| \Cov(Y_0,Y_k) \right| <  +\infty$,
and likewise for $D_k$. We have however not pursued in that direction. 
\end{remark}

\subsection{Approximation of the homogenized matrix}
\label{sec:pres_approx}

In this section, we return to the multidimensional setting. To compute the
homogenized matrix $A^\star$ defined by~\eqref{def:astar}, we first need
to solve the corrector problem~\eqref{PB:correc}, which is set on the
entire space. In practice, approximations are therefore in order. 

In the sequel, we describe a strategy introduced
in~\cite{cras_ronan}, and that mimicks the approach proposed 
in~\cite{bourgeat_ihp} to approximate standard corrector problems in
classical random homogenization. In this article, we analyze
this approach and prove its convergence (see Theorem~\ref{conv:AN}
below). This is our second main result. We refer
to~\cite[Section~3.2]{singapour} for some illustrative numerical tests.

\medskip

\noindent
\textbf{Convention:} Following~\cite[Lemme~2.1]{Blanc2006}, we adopt the
convention that $\dps \left[ \nabla \phi \right]_{ij} = 
\frac{\partial \phi_i}{\partial x_j}$ for any $1 \leq i,j \leq
d$. Hence, for any scalar-valued 
function $\psi$, the gradient of $\widetilde{\psi} = \psi \circ \phi$ is
given by $\nabla \widetilde{\psi}(z) = 
(\nabla \phi(z))^T \nabla \psi (\phi(z))$. This convention implies that
$\left[ \nabla \phi(\phi^{-1}) \right] \, \nabla (\phi^{-1}) = \Id$.

\paragraph{Presentation of the approximation} 

The weak formulation of the corrector problem~\eqref{PB:correc} reads as
follows (see~\cite{Blanc2006}): for all $\widetilde{\psi}$ stationary in
the sense of~\eqref{eq:stationnarite-disc}, we have 
$$
\EE\left(\int_{\phi(Q,\cdot)} (\nabla \psi(y,\omega))^T 
A_{\rm per}\left(\phi^{-1}(y,\omega)\right)
\left(p+\nabla w_p(y,\omega)\right) dy\right) = 0,
$$
where $\psi=\widetilde{\psi}\circ\phi^{-1}$. The above expression can be
rewritten, after a change of variables, as
$$
\EE\left[\int_Q \det (\nabla \phi) \left( \nabla \widetilde{\psi} \right)^T 
(\nabla \phi)^{-1} A_{\rm per} 
\left(p + (\nabla \phi)^{-T} \nabla \widetilde{w}_p \right)\right] = 0.
$$
Since $\widetilde{\psi}$, $\nabla \phi$, $A_{\rm per}$ and $\nabla
\widetilde{w}_p$ are stationary in the sense
of~\eqref{eq:stationnarite-disc}, the ergodic theorem yields
$$
\lim\limits_{N\to\infty} \frac{1}{|Q_N|} \int_{Q_N}
\det (\nabla \phi) \left( \nabla \widetilde{\psi} \right)^T 
(\nabla \phi)^{-1} A_{\rm per} 
\left(p + (\nabla \phi)^{-T} \nabla \widetilde{w}_p\right) = 0
\quad \text{a.s.}
$$
where $Q_N = N Q$.
For a fixed $N$, we now define the approximate corrector
$\widetilde{w}_p^N$ as the $Q_N$-periodic function satisfying:
\begin{equation}
\label{PB:corrN-per}
\text{for all $\widetilde{\psi}$ $Q_N$-periodic,}
\quad
\int_{Q_N} \det (\nabla \phi) \left( \nabla \widetilde{\psi} \right)^T 
(\nabla \phi)^{-1} A_{\rm per} 
\left(p + (\nabla \phi)^{-T} \nabla \widetilde{w}^N_p\right)=0.
\end{equation}
Note that $\widetilde{w}_p^N$ is uniquely defined up to an additive constant. 

In turn, recall that $A^\star$ is defined by~\eqref{def:astar}. After a
change of variables, we infer from that equation that,
for any $1 \leq i,j \leq d$, we have
$$
A^\star_{ij} 
=
\mbox{det} 
\left(\EE\left(\int_Q \nabla \phi(y,\cdot) dy \right)\right)^{-1}
\EE\left(\int_Q \mbox{det}(\nabla \phi(y,\cdot)) \,
e_i^T A_{\rm per}\left(y\right) 
\left( e_j + (\nabla \phi)^{-T} \nabla
  \widetilde{w}_{e_j}(y,\cdot)\right) 
dy \right).
$$
The ergodic theorem yields
$$
A^\star_{ij}= \lim\limits_{N\to \infty} \left\{ 
\det \left(\frac{1}{|Q_N|}
\int_{Q_N} \nabla \phi(\cdot,\omega) \right)^{-1} 
\frac{1}{|Q_N|} \int_{Q_N} \det (\nabla \phi) e_i^T A_{\rm per} 
\left(e_j + (\nabla \phi)^{-T} \nabla \widetilde{w}_{e_j}\right) \right\}
\quad \text{a.s.}
$$
It is thus natural to approximate $A^\star$ by the matrix $A^\star_N(\omega)$
defined by
\begin{equation}
\label{def:AstarN}
A^\star_N(\omega) = \det \left(\frac{1}{|Q_N|}
\int_{Q_N} \nabla \phi(\cdot,\omega) \right)^{-1} 
B^\star_N(\omega),
\end{equation}
where the matrix $B^\star_N(\omega)$ is defined by, 
for any $1 \leq i,j \leq d$,
\begin{eqnarray}
\nonumber
\left[ B^\star_N \right]_{ij}(\omega)
&=& 
\frac{1}{|Q_N|} \int_{Q_N} \det (\nabla \phi) \, e_i^T A_{\rm per} 
\left(e_j + (\nabla \phi)^{-T} \nabla \widetilde{w}^N_{e_j}\right) 
\\
\label{eq:-1}
&=& \frac{1}{|Q_N|} \int_{\phi(Q_N,\omega)} e_i^T A_{\rm per}
\left(\phi^{-1}(y,\omega)\right) \left(e_j + \nabla w^N_{e_j}(y,\omega)
\right) \, dy,
\end{eqnarray}
where, for any $p \in \RR^d$, $\widetilde{w}^N_p$ is defined
by~\eqref{PB:corrN-per} and where
$$
w_p^N(y,\omega) = \widetilde{w}_p^N(\phi^{-1}(y,\omega),\omega).
$$ 
Note that, as is standard in stochastic homogenization, the
approximation $A^\star_N(\omega)$ is a {\em 
  random} matrix, even though the exact homogenized matrix $A^\star$ is
deterministic. This is a by-product of working on the {\em truncated}
domain $Q_N$ rather than $\RR^d$. 

\paragraph{Convergence of the approach} 

We prove in Section~\ref{sec:proofAN} below the following convergence result:
\begin{theo}
\label{conv:AN}
Let $\phi$ be a diffeomorphism
satisfying~\eqref{phi:cond1},~\eqref{phi:cond2} and~\eqref{phi:cond3},
and $A_{\rm per}$ be a periodic matrix that satisfies the ellipticity
condition~\eqref{eq:hyp_a}. 
Then the random matrix $A^\star_N(\omega)$
defined by~\eqref{def:AstarN} converges almost surely to the
deterministic homogenized matrix $A^\star$ defined by~\eqref{def:astar} when
$N \to \infty$. 
\end{theo}

\section{Asymptotic behavior of the residual}
\label{sec:asymp}

The aim of this Section and of the next one is to prove our first main
result, Theorem~\ref{theo4}.
Using the one dimensional setting, we first establish a
``representation'' formula for the residual (see Section~\ref{sec:rep},
Theorem~\ref{theo2}). Using this formula, we are next in position to
study the asymptotic behavior of the residual when $\eps \to 0$ (see
Section~\ref{sec:pth4}). Section~\ref{sec:tech} collects the proofs of
some technical results used in Sections~\ref{sec:rep}
and~\ref{sec:pth4}. 

\subsection{Representation formulas}
\label{sec:rep}

The following technical result will be very useful in the sequel. Its
proof is postponed until Section~\ref{sec:tech1}.

\begin{lemme} 
\label{lemme1}
Assume that $a_{\rm per}$ and $\phi$
satisfy~\eqref{eq:hyp_a},~\eqref{phi:cond1},~\eqref{phi:cond2}
and~\eqref{phi:cond3}.
Assume furthermore the independence
conditions~\eqref{eq:assump_Y_iid} and~\eqref{eq:assump_D_iid}.
For any $0 \leq \alpha
\leq \beta \leq 1$, and any ${\cal A} \in
L^\infty(\alpha,\beta)$ with ${\cal A}' \in L^2(\alpha,\beta)$, define
the random variable
\begin{equation}
\label{eq:def_Z_bar}
\overline{Z}_\eps(\alpha,\beta,\omega) =
\frac{1}{\sqrt{\eps}} 
\int_\alpha^\beta
{\cal A}(t)
\psi\left(\phi^{-1}\left(\frac{t}{\eps},\omega\right)\right) dt, 
\end{equation}
where the function $\psi$ is defined by~\eqref{def:psi}. For any $p \in
\NN^\star$, there exists a
deterministic constant $C_p$ independent of ${\cal A}$, $\eps$, $\alpha$
and $\beta$, such that  
\begin{equation*}
\forall \eps>0, \quad
\EE \left[ \overline{Z}_\eps(\alpha,\beta,\cdot)^{2p} \right] 
\leq
C_p \left[ (\beta-\alpha)^p + \eps^{(p-1)/2} \right]
\left[ \| {\cal A} \|^{2p}_{L^\infty(\alpha,\beta)} + 
(\beta-\alpha)^p \, \| {\cal A}' \|^{2p}_{L^2(\alpha,\beta)} \right].
\end{equation*}
\end{lemme}

The above result heuristically implies that the quantity $\dps
\int_\alpha^\beta
{\cal A}(t)
\psi\left(\phi^{-1}\left(\frac{t}{\eps},\omega\right)\right) dt
$ is of the order of $\sqrt{\eps}$.

We will show below a {\em convergence} result for the random variables 
$\overline{Z}_\eps(\alpha,\beta,\omega)$ (see Lemma~\ref{lemme2}
below). The boundedness result stated in the above lemma is however
sufficient for now. Using it, we indeed prove the following theorem,
which is a key ingredient to prove Theorem~\ref{theo4}. 

\begin{theo}
\label{theo2}
Assume that $a_{\rm per}$ and $\phi$
satisfy~\eqref{eq:hyp_a},~\eqref{phi:cond1},~\eqref{phi:cond2}
and~\eqref{phi:cond3}.
Assume furthermore the independence
conditions~\eqref{eq:assump_Y_iid} and~\eqref{eq:assump_D_iid}.
Let $u^\eps$ be the solution to~\eqref{PB:stoch} and $u_\star$ be the
solution to~\eqref{PB:homo-1D}. Then 
\begin{equation}
u^\eps(x,\omega) - u_\star(x) = 
\int_0^1 K_0(x,t) \,
\psi\left(\phi^{-1}\left(\frac{t}{\eps},\omega\right)\right) dt 
+ r_\eps(x,\omega),
\label{eq:theo2}
\end{equation}
where $K_0$ is defined by~\eqref{def:K0}, $\psi$ is defined
by~\eqref{def:psi}, and there exists a deterministic constant $C$
independent of $\eps$ such that, for any $\eps>0$, 
\begin{equation}
\label{eq:theo2_bound}
\sup\limits_{x\in [0,1]}
\EE \left| r_\eps(x,\cdot) \right| \leq C \eps
\quad \text{and} \quad
\EE \left[ \| r_\eps \|_{L^2(0,1)}^2 \right]
\leq
C \eps^2.
\end{equation}
In addition, for any $p \in \NN^\star$, there exists a deterministic
constant $C_p$ independent of $\eps$ such that 
\begin{equation}
\label{tight_reps}
\forall \eps \in (0,1), \quad
\forall (x,y) \in (0,1)^2, \quad 
\EE \left[ \left| r_\eps(x,\cdot) - r_\eps(y,\cdot) \right|^{2p} \right]
\leq
C_p \eps^{2p} \sqrt{(x-y)^{2p} + \eps^{p-1/2}}
\end{equation}
and
\begin{equation}
\label{tight_reps2}
\forall \eps \in (0,1), \quad
\forall (x,y) \in (0,1)^2, \quad 
\EE \left[ \left| r_\eps(x,\cdot) - r_\eps(y,\cdot) \right|^{2p} \right]
\leq
C_p \eps^p \, (x-y)^{2p}.
\end{equation}
\end{theo}
In view of Lemma~\ref{lemme1}, the first term of the right-hand side
of~\eqref{eq:theo2} is of the order of 
$\sqrt{\eps}$. The term $r_\eps$, which is of the order of $\eps$ in
view of~\eqref{eq:theo2_bound}, is hence a higher-order term. The
bounds~\eqref{tight_reps} and~\eqref{tight_reps2} will be useful below
to show that some random process is tight (see Section~\ref{sec:pth4},
Theorem~\ref{theo:cv_process}). 

\medskip

Using the same arguments, we show the following result: 
\begin{theo}
\label{theo3}
Assume that $a_{\rm per}$ and $\phi$
satisfy~\eqref{eq:hyp_a},~\eqref{phi:cond1},~\eqref{phi:cond2}
and~\eqref{phi:cond3}.
Assume furthermore the independence
conditions~\eqref{eq:assump_Y_iid} and~\eqref{eq:assump_D_iid}.
Let $u^\eps$ be the solution to~\eqref{PB:stoch}, $u_\star$ be the
solution to~\eqref{PB:homo-1D}, and $w$ be the corrector, which solves~\eqref{PB:correc-1D}. Then 
\begin{multline}
\label{eq:theo3}
\frac{d}{dx}\left(u^\eps(x,\omega)-u_\star(x) - 
\eps w\left(\frac{x}{\eps},\omega\right)
\frac{d u_\star}{dx}(x)\right) = 
a^{-1}_{\rm per}\left(\phi^{-1}\left(\frac{x}{\eps},\omega\right)\right)
\int_0^1 K_1(t) \, 
\psi\left(\phi^{-1}\left(\frac{t}{\eps},\omega\right)\right)dt
\\
+
f(x) \int_0^x 
\psi\left(\phi^{-1}\left(\frac{t}{\eps},\omega\right)\right)dt
+
\overline{r}_\eps(x,\omega),
\end{multline}
where $\psi$ is defined by~\eqref{def:psi}, $K_1$ is given by
\begin{equation}
K_1(t)= a^\star \left(F(t)- \int_0^1 F(s) ds \right)
\quad \text{with} \quad F(t) = \int_0^t f(s) \, ds,
\label{def:K1}
\end{equation}
and there exists a deterministic constant $C$ independent of
$\eps$ such that, for all $\eps>0$, 
\begin{equation}
\label{eq:theo3_bound}
\EE \left[ \sup_{x\in [0,1]} \left| \overline{r}_\eps(x,\cdot) \right| 
\right] \leq C \eps
\quad \text{and} \quad
\EE \left[ \sup_{x\in [0,1]} \left| \overline{r}_\eps(x,\cdot) \right|^2 
\right] \leq C \eps^2.
\end{equation}
\end{theo}
Again, in view of Lemma~\ref{lemme1}, the two first terms
of the right-hand side of~\eqref{eq:theo3} are of the order of
$\sqrt{\eps}$. The term 
$\overline{r}_\eps$, which is of the order of $\eps$, 
is hence a higher-order term.

\begin{remark}
It is easy to deduce from~\eqref{eq:theo3}, using Lemma~\ref{lemme1}
and~\eqref{eq:theo3_bound}, 
that there exists a deterministic constant $C$ independent of $\eps$
such that 
\begin{equation}
\label{eq:H1_conv1}
\EE \left[ \left\| 
\frac{d}{dx}\left(u^\eps(x,\omega)-u_\star(x) - 
\eps w\left(\frac{x}{\eps},\omega\right)
\frac{d u_\star}{dx}(x)\right)
\right\|^2_{L^2(0,1)} \right] \leq C \eps.
\end{equation}
Likewise, we deduce from~\eqref{eq:theo2}, using Lemma~\ref{lemme1}
and~\eqref{eq:theo2_bound}, 
that 
\begin{equation}
\label{eq:H1_conv2}
\EE \left[ \left\| 
u^\eps(x,\omega)-u_\star(x) 
\right\|^2_{L^2(0,1)} \right] \leq C \eps.
\end{equation} 
Using the expression~\eqref{eq:corr_1D} below, we infer
from~\eqref{eq:H1_conv1} and~\eqref{eq:H1_conv2} that
\begin{equation}
\label{eq:H1_conv3}
\EE \left[ \left\| 
u^\eps(x,\omega)-u_\star(x) - 
\eps w\left(\frac{x}{\eps},\omega\right)
\frac{d u_\star}{dx}(x)
\right\|^2_{H^1(0,1)} \right] \leq C \eps.
\end{equation}
We recover (in the one-dimensional situation) a classical result of
homogenization: the corrector $w$ 
allows to obtain a convergence result in the $H^1$ strong norm. We refer
to~\cite[Theorem 3]{Papanicolaou-Varadhan} for a corresponding result in
classical random homogenization (in the multidimensional setting).
\end{remark}

The proof of Theorems~\ref{theo2} and~\ref{theo3} are
direct consequences of Lemma~\ref{lemme1} and of the analytical
expression of $u^\eps$ and $u_\star$. 

\begin{proof}[Proof of Theorem~\ref{theo2}]
Introduce $\dps F(x)=\int_0^x f(t)dt$. The solution to~\eqref{PB:stoch} reads
\begin{equation}
u^\eps(x,\omega)=c^\eps(\omega) \int_0^x\frac{1}{a_{\rm
     per}\left(\phi^{-1}\left(\frac{t}{\eps},\omega\right)\right)}dt - \int_0^x\frac{F(t)}{a_{\rm per}\left(\phi^{-1}\left(\frac{t}{\eps},\omega\right)\right)}dt,
\label{eq:ueps}
\end{equation}
where 
\begin{equation}
  c^\eps(\omega)=\frac{ \dps \int_0^1\frac{F(t)}{a_{\rm
        per}\left(\phi^{-1}\left(\frac{t}{\eps},\omega\right)\right)}dt}{ \dps \int_0^1\frac{1}{a_{\rm per}\left(\phi^{-1}\left(\frac{t}{\eps},\omega\right)\right)}dt}.
\label{eq:ceps}
\end{equation}
Likewise, the solution $u_\star$ of the homogenized problem~\eqref{PB:homo} is
\begin{eqnarray}
u_\star(x)=c^\star \frac{x}{a^\star} - \int_0^x \frac{F(t)}{a^\star} dt,
\label{eq:ustar}
\end{eqnarray}
where $a^\star$ is given by~\eqref{def:astar-1d} and
\begin{eqnarray}
 c^\star = \int_0^1 F(t) dt.\label{eq:cstar}
\end{eqnarray}


\noindent
\textbf{Step 1: Representation formula} 

We compute the residual process using~\eqref{eq:ustar} and~\eqref{eq:ueps}:
\begin{eqnarray}
u^\eps(x,\omega) - u_\star(x) &=&
c^\eps(\omega) \int_0^x\frac{1}{a_{\rm
     per}\left(\phi^{-1}\left(\frac{t}{\eps},\omega\right)\right)}dt
- 
c^\star \frac{x}{a^\star}
- \int_0^x F(t) \psi \left(
  \phi^{-1}\left(\frac{t}{\eps},\omega\right)\right) \, dt
\nonumber
\\
&=&
c^\eps(\omega) \int_0^x \psi 
\left( \phi^{-1}\left(\frac{t}{\eps},\omega\right)\right)dt
+ (c^\eps(\omega) - c^\star) \frac{x}{a^\star}
- \int_0^x F(t) \psi \left(
\phi^{-1}\left(\frac{t}{\eps},\omega\right)\right) \, dt
\nonumber
\\
&=&
(c^\eps(\omega) - c^\star) \int_0^x
\psi\left(\phi^{-1}\left(\frac{t}{\eps},\omega \right)\right) dt
+ (c^\eps(\omega)-c^\star) \frac{x}{a^\star}
\nonumber
\\
& & +
\int_0^x \left(c^\star - F(t)\right)
\psi\left(\phi^{-1}\left(\frac{t}{\eps},\omega \right)\right) dt,
\label{eq:ueps-ustar}
\end{eqnarray}
where $\psi$ is defined by~\eqref{def:psi}. We also infer
from~\eqref{eq:ceps} that
\begin{eqnarray}
c^\eps(\omega) - c^\star &=& 
\left( \int_0^1 \frac{1}{a_{\rm 
      per}\left(\phi^{-1}\left(\frac{t}{\eps},\omega\right)\right)}dt \right)^{-1} 
\int_0^1 \left(F(t) - c^\star \right) \frac{1}{a_{\rm
    per}\left(\phi^{-1}\left(\frac{t}{\eps},\omega\right)\right)}dt
\nonumber
\\
&=&
\left( \int_0^1 \frac{1}{a_{\rm
      per}\left(\phi^{-1}\left(\frac{t}{\eps},\omega\right)\right)}dt \right)^{-1} 
\int_0^1 \left(F(t) - c^\star \right) \psi
\left(\phi^{-1}\left(\frac{t}{\eps},\omega\right)\right)dt
\label{eq:ceps-cstar-pre} 
\end{eqnarray}
where we have used that, in view of~\eqref{eq:cstar}, we have $\dps
\int_0^1 \left(F(t) - c^\star \right) \frac{1}{a^\star}dt=0$. Observe
now that 
$$
\left( \int_0^1 \frac{1}{a_{\rm
      per}\left(\phi^{-1}\left(\frac{t}{\eps},\omega\right)\right)}dt \right)^{-1} 
=
a^\star - \frac{a^\star}{\dps \int_0^1 a^{-1}_{\rm per}\left(\phi^{-1}\left(\frac{t}{\eps},\omega\right)\right) dt} \int_0^1 \psi\left(\phi^{-1}\left(\frac{t}{\eps},\omega\right)\right) dt.
$$
We then deduce from~\eqref{eq:ceps-cstar-pre} that
\begin{equation}
c^\eps(\omega) - c^\star =
a^\star \int_0^1 \left(F(t)- c^\star\right)
\psi\left(\phi^{-1}\left(\frac{t}{\eps},\omega\right)\right) dt -
\rho^\eps(\omega), 
\label{eq:ceps-cstar}
\end{equation}
where
\begin{equation}
\rho^\eps(\omega) = \left[ \frac{a^\star}{\dps \int_0^1 a^{-1}_{\rm
      per}\left(\phi^{-1}\left(\frac{t}{\eps},\omega\right)\right) dt} \int_0^1 \psi\left(\phi^{-1}\left(\frac{t}{\eps},\omega\right)\right) dt \right]
\int_0^1\left(F(t)- c^\star\right)
\psi\left(\phi^{-1}\left(\frac{t}{\eps},\omega\right)\right) dt. 
\label{eq:rho}
\end{equation}
Collecting~\eqref{eq:ueps-ustar} and~\eqref{eq:ceps-cstar}, we write
\begin{eqnarray*}
u^\eps(x,\omega) - u_\star(x) &=&
(c^\eps(\omega) - c^\star) \int_0^x
\psi\left(\phi^{-1}\left(\frac{t}{\eps},\omega \right)\right) dt
+
\int_0^x \left(c^\star - F(t)\right)
\psi\left(\phi^{-1}\left(\frac{t}{\eps},\omega \right)\right) dt
\\
& & + \ x \int_0^1 \left(F(t)-c^\star\right)
\psi\left(\phi^{-1}\left(\frac{t}{\eps},\omega\right)\right) dt
- \frac{x}{a^\star} \rho^\eps(\omega)
\\
&=& r_\eps(x,\omega) + \int_0^x \left(c^\star - F(t)\right)
\psi\left(\phi^{-1}\left(\frac{t}{\eps},\omega \right)\right) dt
+ \ x \int_0^1 \left(F(t)- c^\star\right)
\psi\left(\phi^{-1}\left(\frac{t}{\eps},\omega\right)\right) dt
\\
&=& r_\eps(x,\omega) + \int_0^1 K_0(x,t) 
\psi\left(\phi^{-1}\left(\frac{t}{\eps},\omega \right)\right) dt
\end{eqnarray*}
with
$$
K_0(x,t) = \left( \mathbf{1}_{[0,x]}(t) - x \right) 
\left(c^\star - F(t)\right)
$$
and
\begin{equation}
\label{eq:def_r_eps}
r_\eps(x,\omega) = - \frac{x}{a^\star} \rho^\eps(\omega) 
+ (c^\eps(\omega)-c^\star)\int_0^x\psi\left(\phi^{-1}\left(\frac{t}{\eps},\omega \right)\right) dt.
\end{equation}
In view of~\eqref{eq:cstar}, we recover the
expression~\eqref{def:K0} of $K_0$. We thus have written the residual in
the form~\eqref{eq:theo2}. 

\bigskip

\noindent
\textbf{Step 2: Proof of the bound~\eqref{eq:theo2_bound}}

We first bound $\rho^\eps(\omega)$.
We infer from~\eqref{eq:rho} that
\begin{equation}
\label{eq:bound_rho}
\left|\rho^\eps(\omega)\right|
\leq a^+ a^\star 
\left| 
\int_0^1 \psi \left(\phi^{-1}\left(\frac{t}{\eps},\omega\right)\right) dt
\right| 
\ \
\left| 
\int_0^1 \left(F(t)- c^\star\right) 
\psi\left(\phi^{-1}\left(\frac{t}{\eps},\omega\right)\right) dt \right|.
\end{equation}
Using the Cauchy Schwartz inequality, we deduce that
$$
\EE(\left|\rho^\eps\right|) \leq \eps a^+ a^\star 
\sqrt{\EE\left( \left| \frac{1}{\sqrt{\eps}}
\int_0^1 \psi \left(\phi^{-1}\left(\frac{t}{\eps},\cdot\right)\right) dt
\right|^2\right)} 
\ 
\sqrt{\EE\left( \left| \frac{1}{\sqrt{\eps}}
\int_0^1 \left(F(t)- c^\star\right) 
\psi\left(\phi^{-1}\left(\frac{t}{\eps},\cdot\right)\right) dt 
\right|^2\right)}.
$$
Using Lemma~\ref{lemme1} with $p=1$, $\alpha=0$, $\beta=1$, ${\cal A}(t) = 1$
and ${\cal A}(t) = F(t)- c^\star$, 
we obtain that there exists a constant $C$ independent of $\eps$ such that 
\begin{equation}
\EE(|\rho^\eps|) \leq C \eps.
\label{eq:majrho}
\end{equation}
We also deduce from~\eqref{eq:bound_rho} that, for any $p \in \NN^\star$,
$$
\EE \left(\left|\rho^\eps\right|^{2p}\right) \leq (a^+ a^\star)^{2p} \eps^{2p}
\sqrt{\EE\left( \left| \frac{1}{\sqrt{\eps}}
\int_0^1 \psi \left(\phi^{-1}\left(\frac{t}{\eps},\cdot\right)\right) dt
\right|^{4p}\right)} 
\ 
\sqrt{\EE\left( \left| \frac{1}{\sqrt{\eps}}
\int_0^1 \left(F(t)- c^\star\right) 
\psi\left(\phi^{-1}\left(\frac{t}{\eps},\cdot\right)\right) dt 
\right|^{4p}\right)}.
$$
Using again Lemma~\ref{lemme1}, 
we obtain that there exists a constant $C_p$ independent of $\eps$ such that 
\begin{equation}
\EE \left(|\rho^\eps|^{2p}\right) \leq C_p \eps^{2p}.
\label{eq:majrho2}
\end{equation}
Using the obtained bounds on $\rho^\eps$, we now estimate $r_\eps$. We infer from~\eqref{eq:ceps-cstar}, using~\eqref{eq:majrho2} and
Lemma~\ref{lemme1}, that, for any $p \in \NN^\star$, 
\begin{eqnarray}
\EE \left( \left| c^\eps - c^\star \right|^{2p} \right) 
&\leq& (a^\star)^{2p} \eps^p
C_p \EE \left( \left| \frac{1}{\sqrt{\eps}} \int_0^1 \left(F(t)- c^\star\right)
\psi\left(\phi^{-1}\left(\frac{t}{\eps},\cdot\right)\right) dt 
\right|^{2p} \right)
+
C_p \EE \left( \left| \rho^\eps \right|^{2p} \right)
\nonumber
\\
&\leq&
(a^\star)^{2p} C_p \eps^p 
+
C_p \eps^{2p}
\nonumber
\\
&\leq&
C_p \eps^p 
\label{eq:ceps-cstar_maj}
\end{eqnarray}
for a constant $C_p$ independent of $\eps$.
In view of~\eqref{eq:def_r_eps}, we thus obtain, using~\eqref{eq:majrho}
and~\eqref{eq:ceps-cstar_maj}, that
\begin{eqnarray*}
\EE(|r_\eps(x,\cdot)|) 
&\leq& 
\frac{x}{a^\star} \EE(|\rho^\eps|) 
+
\EE\left(\left|(c^\eps-c^\star)\int_0^x\psi\left(\phi^{-1}\left(\frac{t}{\eps},\cdot
      \right)\right) dt\right|\right)
\\
&\leq& 
C \eps
+ \sqrt{\eps}
\sqrt{ \EE\left(\left|(c^\eps-c^\star) \right|^2 \right) }
\sqrt{ \EE\left(\left| \frac{1}{\sqrt{\eps}}
\int_0^x\psi\left(\phi^{-1}\left(\frac{t}{\eps},\cdot
      \right)\right) dt\right|^2\right) }
\\
&\leq& 
C \eps
\end{eqnarray*}
for a constant $C$ independent from $\eps$ and $x \in (0,1)$. This
concludes the proof of the first assertion in~\eqref{eq:theo2_bound}. 

\medskip

Similarly, we have
$$
\left( r_\eps(x,\omega) \right)^2 
\leq
\frac{2}{(a^\star)^2} \rho^\eps(\omega)^2
+
2 (c^\eps(\omega)-c^\star)^2 \left| 
\int_0^x\psi\left(\phi^{-1}\left(\frac{t}{\eps},\omega
  \right)\right) dt \right|^2,
$$
thus
\begin{eqnarray*}
\EE \left[ \| r_\eps \|_{L^2(0,1)}^2 \right]
& \leq &
\frac{2}{(a^\star)^2} \EE \left[ |\rho^\eps|^2 \right]
+
2 \int_0^1  \EE \left[ (c^\eps-c^\star)^2 \left| 
\int_0^x\psi\left(\phi^{-1}\left(\frac{t}{\eps},\cdot
  \right)\right) dt \right|^2 \right]dx
\\
& \leq &
C \eps^2
+
2 \int_0^1  \sqrt{
\EE \left[ (c^\eps-c^\star)^4 \right]
\EE \left[ \left| 
\int_0^x \psi\left(\phi^{-1}\left(\frac{t}{\eps},\cdot
  \right)\right) dt \right|^4 \right] }dx.
\end{eqnarray*}
Using~\eqref{eq:ceps-cstar_maj} and Lemma~\ref{lemme1} with $p=2$, we
deduce that 
$$
\EE \left[ \| r_\eps \|_{L^2(0,1)}^2 \right]
\leq
C \eps^2
+
2 \int_0^1  \sqrt{C \eps^4}dx
\leq
C \eps^2
$$
for a constant $C$ independent from $\eps$. 
This concludes the proof of the second assertion
in~\eqref{eq:theo2_bound}.  

\bigskip

\noindent
\textbf{Step 3: Proof of the bounds~\eqref{tight_reps} and~\eqref{tight_reps2}}

We first prove~\eqref{tight_reps}. In view of~\eqref{eq:def_r_eps}, we
have
$$
r_\eps(x,\omega) - r_\eps(y,\omega)
=
\frac{y-x}{a^\star} \rho^\eps(\omega) 
+ (c^\eps(\omega)-c^\star)
\int_y^x
\psi\left(\phi^{-1}\left(\frac{t}{\eps},\omega \right)\right) dt,
$$
thus
\begin{equation}
\label{eq:diff_r}
\left| r_\eps(x,\omega) - r_\eps(y,\omega) \right|^{2p}
\leq
C_p \left( \frac{y-x}{a^\star} \right)^{2p} (\rho^\eps(\omega))^{2p}
+ C_p (c^\eps(\omega)-c^\star)^{2p}
\left| \int_y^x
\psi\left(\phi^{-1}\left(\frac{t}{\eps},\omega \right)\right) dt
\right|^{2p},
\end{equation}
and
\begin{eqnarray*}
\EE \left[ \left| r_\eps(x,\cdot) - r_\eps(y,\cdot) \right|^{2p} \right]
& \leq &
C_p \left( \frac{y-x}{a^\star} \right)^{2p} 
\EE \left[ (\rho^\eps)^{2p} \right]
+ C_p \sqrt{\EE [(c^\eps-c^\star)^{4p}]}
\sqrt{\EE \left[ \left| \int_y^x
\psi\left(\phi^{-1}\left(\frac{t}{\eps},\cdot \right)\right) dt
\right|^{4p} \right] } 
\\
& \leq &
C_p \eps^{2p} (y-x)^{2p} 
+ C_p \sqrt{\eps^{2p}}
\sqrt{\eps^{2p} \left( (x-y)^{2p} + \eps^{(2p-1)/2} \right)} 
\\
& \leq &
C_p \eps^{2p} \left[ (y-x)^{2p} + \sqrt{(x-y)^{2p} + \eps^{p-1/2}} \right].
\end{eqnarray*}
Since $|y-x| \leq 1$, we have 
$\dps (y-x)^{2p} \leq |y-x|^p \leq \sqrt{(x-y)^{2p} + \eps^{p-1/2}}$,
and thus 
$$
\EE \left[ \left| r_\eps(x,\cdot) - r_\eps(y,\cdot) \right|^{2p} \right]
\leq
C_p \eps^{2p} \sqrt{(x-y)^{2p} +  \eps^{p-1/2}}.
$$
This concludes the proof of~\eqref{tight_reps}. We finally prove~\eqref{tight_reps2}. We infer from~\eqref{eq:diff_r} that
$$
\left| r_\eps(x,\omega) - r_\eps(y,\omega) \right|^{2p}
\leq
C_p \left( \frac{y-x}{a^\star} \right)^{2p} (\rho^\eps(\omega))^{2p}
+ C_p (c^\eps(\omega)-c^\star)^{2p} \ (x-y)^{2p} \ 
\| \psi \|_{L^\infty(\RR)}^{2p}, 
$$
hence, using~\eqref{eq:majrho2} and~\eqref{eq:ceps-cstar_maj},
$$
\EE \left[ \left| r_\eps(x,\cdot) - r_\eps(y,\cdot) \right|^{2p} \right]
\leq
C_p (x-y)^{2p} \left[ \eps^{2p} + \eps^p \right].
$$
This concludes the proof of~\eqref{tight_reps2} and thus that of
Theorem~\ref{theo2}.
\end{proof}

\bigskip

\begin{proof}[Proof of Theorem~\ref{theo3}]
Recall that the solution to the corrector problem~\eqref{PB:correc-1D}
satisfies~\eqref{eq:corr}. We thus have, using~\eqref{eq:ueps}
and~\eqref{eq:ustar},  
\begin{eqnarray*}
\frac{d}{dx}\left(u^\eps(x,\omega)-u_\star(x) - \eps
  w\left(\frac{x}{\eps},\omega\right)\frac{d
    u_\star}{dx}(x)\right) 
&=& 
\frac{d u^\eps}{dx}(x,\omega)-\frac{d u_\star}{dx}(x) \left( 1 +
  w'\left( \frac{x}{\eps},\omega \right) \right) 
- \eps \frac{d^2 u_\star}{dx^2}(x)
w\left(\frac{x}{\eps},\omega\right)
\nonumber 
\\
&=&
\frac{c^\eps(\omega) - c^\star}{
a_{\rm per}\left(\phi^{-1}\left(\frac{x}{\eps},\omega\right)\right)}
- \eps \frac{d^2 u_\star}{dx^2}(x)
w\left(\frac{x}{\eps},\omega\right).
\end{eqnarray*}
Using~\eqref{eq:ceps-cstar}, we deduce that
\begin{eqnarray}
&&
\frac{d}{dx}\left(u^\eps(x,\omega)-u_\star(x) - \eps
  w\left(\frac{x}{\eps},\omega\right)\frac{d
    u_\star}{dx}(x)\right) 
\nonumber
\\
&=& 
\frac{a^\star}{
a_{\rm per}\left(\phi^{-1}\left(\frac{x}{\eps},\omega\right)\right)}
\int_0^1 \left(F(t)- c^\star\right)
\psi\left(\phi^{-1}\left(\frac{t}{\eps},\omega\right)\right) dt 
- \eps \frac{d^2 u_\star}{dx^2}(x) 
w\left(\frac{x}{\eps},\omega\right)
- \frac{\rho^\eps(\omega)}{a_{\rm
    per}\left(\phi^{-1}\left(\frac{x}{\eps},\omega\right)\right)}
\nonumber
\\
&=&
a^{-1}_{\rm per}\left(\phi^{-1}\left(\frac{x}{\eps},\omega\right)\right)
\int_0^1 K_1(t) 
\psi\left(\phi^{-1}\left(\frac{t}{\eps},\omega\right)\right) dt 
+
\eps \frac{f(x)}{a^\star}
w\left(\frac{x}{\eps},\omega\right) 
+ \overline{r}^\eps(x,\omega),
\label{eq:eq_pre}
\end{eqnarray}
with $K_1$ defined by~\eqref{def:K1} and
\begin{equation}
\label{eq:eq}
\overline{r}^\eps(x,\omega)
= 
- \frac{\rho^\eps(\omega)}{a_{\rm
    per}\left(\phi^{-1}\left(\frac{x}{\eps},\omega\right)\right)}.
\end{equation}
Observe now that, in view of~\eqref{eq:corr} and~\eqref{def:psi}, we have
$$
w(y,\omega) = a^\star \int_0^y 
\psi\left(\phi^{-1}\left(t,\omega\right)\right) dt,
$$ 
where we have chosen the integration constant in $w$ such that
$w(0,\omega) = 0$ almost surely. Thus 
\begin{equation}
\label{eq:corr_1D}
w\left(\frac{x}{\eps},\omega\right) 
= 
a^\star \int_0^{x/\eps} 
\psi\left(\phi^{-1}\left(t,\omega\right)\right) dt
=
\frac{a^\star}{\eps} \int_0^x 
\psi\left(\phi^{-1}\left(\frac{t}{\eps},\omega\right)\right) dt.
\end{equation}
Collecting this equation with~\eqref{eq:eq_pre} yields~\eqref{eq:theo3}.  
The bound~\eqref{eq:theo3_bound} follows
from~\eqref{eq:eq},~\eqref{eq:majrho} and~\eqref{eq:majrho2}. 
This concludes the proof of Theorem~\ref{theo3}.
\end{proof}

\subsection{Proof of Theorem~\ref{theo4}}
\label{sec:pth4}

In this section, we prove that the random process 
$\dps \frac{u^\eps(x,\omega) - u_\star(x)}{\sqrt{\eps}}$ converges in
distribution to a Gaussian random process that we characterize.  
Using \eqref{eq:theo2}, we see that
\begin{equation}
\label{eq:reste}
\frac{u^\eps(x,\omega) - u_\star(x)}{\sqrt{\eps}} 
= 
G_\eps(x,\omega) + R_\eps(x,\omega),
\end{equation}
where
\begin{eqnarray}
G_\eps(x,\omega) &=& 
\frac{1}{\sqrt{\eps}}\int_0^1 K_0(x,t) \,
\psi\left(\phi^{-1}\left(\frac{t}{\eps},\omega\right)\right) dt,
\label{def:Geps}
\\ 
R_\eps(x,\omega) &=& 
\frac{1}{\sqrt{\eps}} r_\eps(x,\omega).
\label{def:Reps}
\end{eqnarray}
In view of~\eqref{eq:theo2_bound}, we have
$$
\sup\limits_{x\in [0,1]}
\EE \left| R_\eps(x,\cdot) \right| \leq C \sqrt{\eps}
$$
for a constant $C$ independent of $\eps$. As a consequence, 
\begin{equation}
\label{eq:borne_reste}
\forall x \in (0,1), \quad \text{$R_\eps(x,\cdot)$ converges to 0 in
  probability.} 
\end{equation}
We are thus left with studying the behaviour of $G_\eps(x,\omega)$ as
$\eps \to 0$.  

To prove that the random process $G_\eps(x,\omega)$ converges in
distribution, we will use the following result:

\begin{theo}[\cite{Billingsley1968}, page 54]
\label{theo:cv_process}
Suppose that $\left(G_\eps\right)_{\eps \in (0,1)}$ and $G_0$
are random processes with values in the space of continuous functions
$C^0(0,1)$ with $G_\eps(0,\omega)=G_0(0,\omega)=0$ almost
surely. Assume that
\begin{enumerate}
\item[(i)] for any $k \in \NN^\star$ and any 
$0 \leq x_1 \leq \dots \leq x_k \leq 1$, the random variable
$(G_\eps(x_1,\omega),\dots,G_\eps(x_k,\omega)) \in \RR^k$
converges in distribution to the random variable 
$(G_0(x_1,\omega),\dots,G_0(x_k,\omega))$ as $\eps \rightarrow 0$.
\item[(ii)] $\left(G_\eps\right)_{\eps \in (0,1)}$ is a tight
  sequence of random processes in $C^0(0,1)$. A sufficient condition for
  the tightness of $\left(G_\eps\right)_{\eps\in (0,1)}$
  is the Kolmogorov criterion: there exist $\delta>0$, $\beta>0$ and $C>0$
  such that
\begin{equation}
\label{eq:kol}
\forall \eps \in (0,1), \quad
\forall (x,y) \in (0,1)^2, \quad 
\EE\left[ |G_\eps(x,\cdot) - G_\eps(y,\cdot)|^\beta 
\right] \leq C |x-y|^{1+\delta}.
\end{equation}
\end{enumerate}
Then the process $G_\eps$ converges in distribution to the process $G_0$
as $\eps$ goes to $0$. 
\end{theo}
%
%

For any $x \in (0,1)$, the random variable $G_\eps(x,\omega)$ is of the form of
the random variable $\overline{Z}_\eps(\alpha,\beta,\omega)$
defined in~\eqref{eq:def_Z_bar}, with $\alpha=0$, $\beta=1$ and ${\cal
  A}(t) = K_0(x,t)$.  
In Lemma~\ref{lemme1}, we have shown that the random variable 
$\overline{Z}_\eps(\alpha,\beta,\omega)$ is bounded in the $L^{2p}$
norm. We now show that this random variable converges in law to a
Gaussian random variable. This will be a key ingredient to prove the first
condition of Theorem~\ref{theo:cv_process}. 

\begin{lemme}
\label{lemme2}
Assume that $a_{\rm per}$ and $\phi$
satisfy~\eqref{eq:hyp_a},~\eqref{phi:cond1},~\eqref{phi:cond2}
and~\eqref{phi:cond3}.
Assume furthermore the independence
conditions~\eqref{eq:assump_Y_iid} and~\eqref{eq:assump_D_iid}.
For any $0 \leq \alpha \leq \beta \leq 1$, consider a function ${\cal
  A}$, piecewise continuous over $(\alpha,\beta)$, with a finite number
of discontinuities located at points $\{t_k\}_{1 \leq k \leq m}$, and
such that ${\cal A}' \in L^1(t_k,t_{k+1})$ for any $1 \leq k \leq m-1$. 
Consider the random variable
\begin{equation}
\label{eq:def_Z_bar2}
\overline{Z}_\eps(\alpha,\beta,\omega)=
\frac{1}{\sqrt{\eps}} 
\int_\alpha^\beta {\cal A}(t) \psi\left(\phi^{-1}\left(\frac{t}{\eps},\omega\right)\right)dt  
 \end{equation}
where the function $\psi$ is defined by~\eqref{def:psi}. 
Then $\overline{Z}_\eps(\alpha,\beta,\omega)$ converges in
distribution to a Gaussian random variable
$\overline{Z}_0(\alpha,\beta,\omega)$, of mean zero and variance 
$\dps
\overline{\sigma}(\alpha,\beta) = 
\frac{\Var(Y_0)}{\EE(\int_0^1\phi')} \
\| {\cal A} \|^2_{L^2(\alpha,\beta)}
$,
with $\dps \Var(Y_0)=\EE \left[ \left(\int_0^1 \psi \phi'\right)^2 \right]$. We write
\begin{equation}
\label{eq:def_Z0_bar}
\overline{Z}_0(\alpha,\beta,\omega)
=\frac{\sqrt{\Var(Y_0)}}{\sqrt{\EE(\int_0^1\phi')}} 
\int_\alpha^\beta {\cal A}(t) dW_t,
\end{equation}
where $W_t$ denote the classical Brownian motion.
\end{lemme}

The proof of Lemma~\ref{lemme2} is postponed until
Section~\ref{sec:tech2}. 

To prove the second condition of Theorem~\ref{theo:cv_process}, we will
show that $G_\eps(x,\omega)$ satisfies~\eqref{eq:kol}. Observe that
\begin{equation}
\label{eq:decompo_Geps}
G_\eps(x,\omega)
=
\frac{1}{\sqrt{\eps}} \int_0^x
{\cal A}(t)
\psi\left(\phi^{-1}\left(\frac{t}{\eps},\omega\right)\right) dt
-
\frac{x}{\sqrt{\eps}} \int_0^1
{\cal A}(t)
\psi\left(\phi^{-1}\left(\frac{t}{\eps},\omega\right)\right) dt
\end{equation}
with $\dps {\cal A}(t) = \int_0^1 F(s) \, ds - F(t)$, where 
$\dps F(t) = \int_0^t f(s) \, ds$.  
To prove that $G_\eps(x,\omega)$ satisfies~\eqref{eq:kol}, we
will use the following result, the proof of which is postponed until
Section~\ref{sec:tech3}.

\begin{lemme}
\label{lemme3}
Assume that $a_{\rm per}$ and $\phi$
satisfy~\eqref{eq:hyp_a},~\eqref{phi:cond1},~\eqref{phi:cond2}
and~\eqref{phi:cond3}.
Assume furthermore the independence
conditions~\eqref{eq:assump_Y_iid} and~\eqref{eq:assump_D_iid}.
Consider two functions ${\cal A}_1$ and ${\cal A}_2$ with ${\cal A}_j \in
L^\infty(0,1)$ and ${\cal A}'_j \in L^2(0,1)$, $j=1,2$.
For any $x \in (0,1)$, consider the random variable
\begin{equation}
\label{def:Heps}
H_\eps(x,\omega) = 
\frac{1}{\sqrt{\eps}} \int_0^x
{\cal A}_1(t)
\psi\left(\phi^{-1}\left(\frac{t}{\eps},\omega\right)\right) dt
+
\frac{x}{\sqrt{\eps}} \int_0^1
{\cal A}_2(t)
\psi\left(\phi^{-1}\left(\frac{t}{\eps},\omega\right)\right) dt
\end{equation}
where the function $\psi$ is defined by~\eqref{def:psi}. 

For any $p \in \NN^\star$, there exists a deterministic constant $C_p$
independent of $\eps$, $x$ and $y$ such that
\begin{equation}
\label{eq:kol2}
\forall \eps \in (0,1), \quad
\forall (x,y) \in (0,1)^2, \quad 
\EE\left[ \left| H_\eps(x,\cdot) - H_\eps(y,\cdot)
  \right|^{2p} \right] \leq C_p \left( |x-y|^p + \eps^{(p-1)/2} \right).
\end{equation}
In addition, there exists a deterministic constant $C$ independent of
$\eps$, $x$ and $y$ such that, for any $x$ and $y$ with $|x-y| \leq
\eps$, 
\begin{equation}
\label{eq:kol3}
\left| H_\eps(x,\omega) - H_\eps(y,\omega) \right| 
\leq 
C \sqrt{|x-y|} \quad \text{a.s.}
\end{equation}
\end{lemme}

\medskip

We are now in position to prove Theorem~\ref{theo4}.

\begin{proof}[Proof of Theorem~\ref{theo4}]
We have seen (see~\eqref{eq:reste}) that
\begin{equation}
\label{eq:decompo_residu}
\frac{u^\eps(x,\omega) - u_\star(x)}{\sqrt{\eps}} 
= G_\eps(x,\omega) + R_\eps(x,\omega),
\end{equation}
where $G_\eps(x,\omega)$ and $R_\eps(x,\omega)$ are
defined by~\eqref{def:Geps} and~\eqref{def:Reps}, respectively. 

Let us study the process $G_\eps(x,\omega)$, which reads, we recall,
$$
G_\eps(x,\omega) 
= \frac{1}{\sqrt{\eps}} \int_0^1
K_0(x,t) 
\psi\left(\phi^{-1}\left(\frac{t}{\eps},\omega\right)\right) dt.
$$

As $K_0(0,t)=0$ for any $t$, we have that $G_\eps(0,\omega)=0$ for any
$\eps$, almost surely. We first show that this process satisfies the
first condition of Theorem~\ref{theo:cv_process}.  
For each set of points $0 \leq x_1 \leq \dots \leq x_k \leq 1$ and each
${\cal X}=(\xi_1,\dots,\xi_k) \in \RR^k$, we consider the random
variable 
$$
z_\eps(\omega)=\sum\limits_{j=1}^k \xi_j G_\eps(x_j,\omega).
$$
Observing that
$$
G_\eps(x_j,\omega) = \frac{1}{\sqrt{\eps}} \int_0^1 K_0(x_j,t) \, 
\psi\left(\phi^{-1}\left(\frac{t}{\eps},\omega\right)\right) dt,
$$
we can write $z_\eps$ as
$$
z_\eps(\omega) = \frac{1}{\sqrt{\eps}} \int_0^1 {\cal A}_{{\cal X}}(t) \psi\left(\phi^{-1}\left(\frac{t}{\eps},\omega\right)\right) dt
$$
where
$$
{\cal A}_{{\cal X}}(t) 
=
\sum\limits_{j=1}^k \xi_j K_0(x_j,t)
= 
\left(\int_0^1 F(s) ds - F(t)\right) \sum\limits_{j=1}^k \xi_j 
(\mathbf{1}_{[0,x_j]}(t)-x_j),
$$
with $\dps F(t) = \int_0^t f(s) \, ds$.
By assumption, $f \in L^2(0,1)$, thus ${\cal A}_{{\cal X}}$ is piecewise
continuous with a finite number of discontinuities located at
$\{x_j\}_{1\leq j \leq k}$. In addition, we see that, over each
$(x_i,x_{i+1}), \, 1 \leq i \leq k-1$,
$$
{\cal A}_{\cal X}'(t) =\left( \sum\limits_{j=1}^k \xi_j x_j -  \sum\limits_{j > i}^k \xi_j \right)f(t)
$$
is in  $L^2(x_i,x_{i+1}) \subset L^1(x_i,x_{i+1})$. Thus, using
Lemma~\ref{lemme2}, we obtain that $z_\eps(\omega)$ converges in law to 
$$
z_0(\omega) 
=
\sum\limits_{j=1}^k \xi_j G_0(x_j,\omega)
$$
where $G_0$ is defined by~\eqref{eq:def_G0}. This implies that
$$
\lim_{\eps \to 0} \EE\left[ \exp \left(i \sum_{j=1}^k \xi_j
    G_\eps(x_j,\cdot) \right) \right]
=
\lim_{\eps \to 0} \EE(\exp(i z_\eps)) 
= 
\EE(\exp(i z_0))
=
\EE\left[ \exp \left(i \sum_{j=1}^k \xi_j
    G_0(x_j,\cdot) \right) \right].
$$
Hence, for any $k \in \NN^\star$ and any 
$0 \leq x_1 \leq \dots \leq x_k \leq 1$, 
\begin{equation}
\label{eq:cv_Geps}
(G_\eps(x_1,\omega),\dots,G_\eps(x_k,\omega))
\text{ converges in distribution to }
(G_0(x_1,\omega),\dots,G_0(x_k,\omega)) \text{ as $\eps \to 0$}.
\end{equation}
Collecting~\eqref{eq:decompo_residu},~\eqref{eq:cv_Geps}
and~\eqref{eq:borne_reste}, we obtain that  
\begin{equation}
\label{cond1}
\begin{array}{c}
\text{
the residual process 
$\dps \frac{u^\eps(x,\omega) - u_\star(x)}{\sqrt{\eps}}$
satisfies Condition (i) of Theorem~\ref{theo:cv_process}}
\\
\text{with the limit
process $G_0(x,\omega)$ defined by~\eqref{eq:def_G0}.
}
\end{array}
\end{equation}

We now prove the Kolmogorov criterion, first on the random process
$G_\eps(x,\omega)$, next on the process
$\dps \frac{u^\eps-u_\star}{\sqrt{\eps}}$. 
This will show Condition (ii) of Theorem~\ref{theo:cv_process}.
Following~\eqref{eq:decompo_Geps}, we write 
$$
G_\eps(x,\omega)
=
\frac{1}{\sqrt{\eps}} \int_0^x
{\cal A}(t)
\psi\left(\phi^{-1}\left(\frac{t}{\eps},\omega\right)\right) dt
-
\frac{x}{\sqrt{\eps}} \int_0^1
{\cal A}(t)
\psi\left(\phi^{-1}\left(\frac{t}{\eps},\omega\right)\right) dt
$$
with $\dps {\cal A}(t) = \int_0^1 F(s) \, ds - F(t)$, where $\dps F(t) = \int_0^t
f(s) \, ds$.  
The assumptions of Lemma~\ref{lemme3} are satisfied, thus, for any $p
\in \NN^\star$, there exists $C_p$ such that
\begin{equation}
\label{elt1_kol_G}
\forall \eps \in (0,1), \quad
\forall (x,y) \in (0,1)^2, \quad 
\EE\left[ \left| G_\eps(x,\cdot) - G_\eps(y,\cdot)
  \right|^{2p} \right] \leq C_p \left( |x-y|^p + \eps^{(p-1)/2} \right).
\end{equation}
This directly implies that
\begin{equation}
\label{elt2_kol_G}
\text{when} \quad |x-y| \geq \eps, \quad
\EE\left[ \left| G_\eps(x,\cdot) - G_\eps(y,\cdot)
  \right|^{2p} \right] \leq C_p |x-y|^{(p-1)/2}.
\end{equation}
When $|x-y| \leq \eps$, using~\eqref{eq:kol3}, we see that there exists a
deterministic constant $C$ independent of $\eps$, $x$ and $y$ such that,
\begin{equation}
\label{elt3_kol_G}
\EE \left[ \left| G_\eps(x,\cdot) - G_\eps(y,\cdot) \right|^{2p} \right]
\leq 
C |x-y|^p \leq C |x-y|^{(p-1)/2}
\quad \text{when} \quad |x-y| \leq \eps.
\end{equation}
Collecting~\eqref{elt2_kol_G} and~\eqref{elt3_kol_G}, we obtain that 
\begin{equation}
\label{elt_kol_G}
\forall \eps \in (0,1), \quad
\forall (x,y) \in (0,1)^2, \quad 
\EE \left[ \left| G_\eps(x,\cdot) - G_\eps(y,\cdot) \right|^{2p} \right]
\leq 
C |x-y|^{(p-1)/2}.
\end{equation}

We now turn to the process $R_\eps(x,\omega)$.
In view of~\eqref{def:Reps} and~\eqref{tight_reps}, there exists 
$C_p$ such that
\begin{equation}
\label{elt1_kol_R}
\forall \eps \in (0,1), \quad
\forall (x,y) \in (0,1)^2, \quad 
\EE \left[ \left| R_\eps(x,\cdot) - R_\eps(y,\cdot) \right|^{2p} \right]
\leq
C_p \eps^p \sqrt{(x-y)^{2p} + \eps^{p-1/2}}.
\end{equation}
Hence, we deduce that
\begin{equation}
\label{elt2_kol_R}
\EE \left[ \left| R_\eps(x,\cdot) - R_\eps(y,\cdot) \right|^{2p} \right]
\leq
C_p \eps^p |x-y|^{(2p-1)/4} 
\quad \text{when} \quad |x-y| \geq \eps.
\end{equation}
When $|x-y| \leq \eps$, using~\eqref{tight_reps2}, we see that
\begin{equation}
\label{elt3_kol_R}
\EE \left[ \left| R_\eps(x,\cdot) - R_\eps(y,\cdot) \right|^{2p} \right]
\leq
C_p (x-y)^{2p} 
\leq
C_p \eps^p |x-y|^p 
\leq
C_p \eps^p |x-y|^{(2p-1)/4}
\quad \text{when} \quad |x-y| \leq \eps.
\end{equation}
Collecting~\eqref{elt2_kol_R} and~\eqref{elt3_kol_R}, we obtain that 
\begin{equation}
\label{elt_kol_R}
\forall \eps \in (0,1), \quad
\forall (x,y) \in (0,1)^2, \quad 
\EE \left[ \left| R_\eps(x,\cdot) - R_\eps(y,\cdot) \right|^{2p} \right]
\leq 
C \eps^p |x-y|^{(2p-1)/4}.
\end{equation}

We next write, using~\eqref{eq:decompo_residu},
\begin{equation}
\label{elt2_kol}
\left| 
\frac{u^\eps(x,\omega) - u_\star(x)}{\sqrt{\eps}}
-
\frac{u^\eps(y,\omega) - u_\star(y)}{\sqrt{\eps}}
\right|^{2p}
\leq
C_p \left| G_\eps(x,\omega) - G_\eps(y,\omega) \right|^{2p}
+
C_p \left| R_\eps(x,\omega) - R_\eps(y,\omega) \right|^{2p}.
\end{equation}
Collecting~\eqref{elt_kol_G} and~\eqref{elt_kol_R}, we obtain that
\begin{equation}
\label{cond2_pre}
\forall \eps \in (0,1), \quad
\forall (x,y) \in (0,1)^2, \quad 
\EE \left[ \left| 
\frac{u^\eps(x,\cdot) - u_\star(x)}{\sqrt{\eps}}
-
\frac{u^\eps(y,\cdot) - u_\star(y)}{\sqrt{\eps}}
\right|^{2p} \right]
\leq
C |x-y|^{(p-1)/2} (1 + \eps^p).
\end{equation}
We thus obtain that
\begin{equation}
\label{cond2}
\begin{array}{c}
\text{
the residual process 
$\dps \frac{u^\eps(x,\omega) - u_\star(x)}{\sqrt{\eps}}$
satisfies Condition (ii) of Theorem~\ref{theo:cv_process}
}
\\
\text{with the exponents $\beta=2p$ and $\delta=p/2-3/2$.}
\end{array}
\end{equation}
Choosing $p$ such that $\beta > 0$ and $\delta > 0$ (it suffices to
choose $p > 3$), and collecting~\eqref{cond1} and~\eqref{cond2}, we see
that the random process 
$\dps \frac{u^\eps(x,\omega) - u_\star(x)}{\sqrt{\eps}}$
satisfies the assumptions of Theorem~\ref{theo:cv_process}. It thus
converges in law to the Gaussian process $G_0(x,\omega)$ defined
by~\eqref{eq:def_G0}. This concludes the proof of Theorem~\ref{theo4}.
\end{proof}

\section{Technical proofs}
\label{sec:tech}

We collect here the proofs of Lemmas~\ref{lemme1},~\ref{lemme2}
and~\ref{lemme3}. 

\subsection{Proof of Lemma~\ref{lemme1}}
\label{sec:tech1}

Lemma~\ref{lemme1} is a consequence of the following result: 
\begin{lemme} 
\label{lemme1-s}
Assume that $a_{\rm per}$ and $\phi$
satisfy~\eqref{eq:hyp_a},~\eqref{phi:cond1},~\eqref{phi:cond2}
and~\eqref{phi:cond3}.
Assume furthermore the independence
conditions~\eqref{eq:assump_Y_iid} and~\eqref{eq:assump_D_iid}.
For any $0 \leq \alpha \leq
\beta \leq 1$, define the random variable
\begin{equation}
\label{eq:def_Z}
Z_\eps(\alpha,\beta,\omega) = 
\frac{1}{\sqrt{\eps}} \int_\alpha^\beta 
\psi\left(\phi^{-1}\left(\frac{t}{\eps},\omega\right)\right) dt,
\end{equation}
where the function $\psi$ is defined by~\eqref{def:psi}. For any $p \in
\NN^\star$, there exists a
deterministic constant $C_p$ independent of $\eps$, $\alpha$ and
$\beta$ such that
\begin{equation*}
\forall \eps>0, \quad \EE \left[ Z_\eps(\alpha,\beta,\cdot)^{2p}
\right] \leq C_p \left[ (\beta-\alpha)^p + \eps^{(p-1)/2} \right].
\end{equation*}
\end{lemme}

We first prove Lemma~\ref{lemme1-s}, and next Lemma~\ref{lemme1}.

\begin{proof}[Proof of Lemma~\ref{lemme1-s}]
Using the variable $\dps s =
\phi^{-1}\left(\frac{t}{\eps},\omega\right)$, we write
\begin{equation}
\label{eq:toto1}
Z_\eps(\alpha,\beta,\omega)
=
\frac{1}{\sqrt{\eps}} \int_\alpha^\beta
\psi\left(\phi^{-1}\left(\frac{t}{\eps},\omega\right)\right) dt 
=
\sqrt{\eps}
\int_{\phi^{-1}(\alpha/\eps,\omega)}^{\phi^{-1}(\beta/\eps,\omega)}
\psi(s)\phi'(s,\omega) ds.
\end{equation}
For future use, we introduce, for any $x \in (0,1)$, the notation
$$
\overline{K}_x(\omega) = \lfloor \phi^{-1}(x/\eps,\omega)
\rfloor.
$$
In view of~\eqref{phi:cond1} and~\eqref{phi:cond2}, we have 
$$
M^{-1} \left|
\frac{\beta - \alpha}{\eps}
\right|
\leq 
\left|
\phi^{-1}\left(\frac{\alpha}{\eps},\omega\right) -
\phi^{-1}\left(\frac{\beta}{\eps},\omega\right)
\right|
\leq 
\nu^{-1} \left|
\frac{\beta - \alpha}{\eps}
\right|.
$$
Hence, up to some boundary terms (due to the fact that
$\phi^{-1}(\alpha/\eps,\omega)$ and
$\phi^{-1}(\beta/\eps,\omega)$ are not integer numbers), $Z_\eps/\sqrt{\eps}$
is a sum of the variables $Y_k$ defined by~\eqref{eq:def_Y}, with a
number of terms of the order of $\eps^{-1}$. Note however that this
number of terms, equal to 
$\overline{K}_\beta(\omega) - \overline{K}_\alpha(\omega)$, is
{\em random}. To proceed, we write $Z_\eps$ as the sum of two
contributions: (i) a sum of the variables $Y_k$
with a {\em deterministic} number of terms, and (ii) a remainder, that
will be successively estimated.

Following~\eqref{eq:toto1}, we have 
\begin{eqnarray}
Z_\eps(\alpha,\beta,\omega)
&=& 
\sqrt{\eps}
\left(\int_{\frac{\alpha}{\eps\EE(\int_0^1\phi')}}^{\frac{\beta}{\eps\EE(\int_0^1\phi')}}
  \psi(t) \phi'(t,\omega) dt +
  \int_{\frac{\beta}{\eps\EE(\int_0^1\phi')}}^{\phi^{-1}(\beta/\eps,\omega)} \psi(t) \phi'(t,\omega) dt + \int^{\frac{\alpha}{\eps\EE(\int_0^1\phi')}}_{\phi^{-1}(\alpha/\eps,\omega)} \psi(t) \phi'(t,\omega) dt \right)
\nonumber
\\
\label{eq:decompo}
&=& B_\eps(\alpha,\beta,\omega) + 
A_\eps(\beta,\omega) - A_\eps(\alpha,\omega)
\end{eqnarray}
with
\begin{eqnarray}
A_\eps(x,\omega) &=& \sqrt{\eps}
\int_{\frac{x}{\eps\EE(\int_0^1\phi')}}^{\phi^{-1}(x/\eps,\omega)}
\psi(t) \phi'(t,\omega) dt,
\label{eq:def_AA}
\\
B_\eps(\alpha,\beta,\omega) &=& \sqrt{\eps} \int_{\frac{\alpha}{\eps\EE(\int_0^1\phi')}}^{\frac{\beta}{\eps\EE(\int_0^1\phi')}}
  \psi(t) \phi'(t,\omega) dt.
\label{eq:def_BB}
\end{eqnarray}
Note that, up to boundary terms,
$B_\eps(\alpha,\beta,\omega)/\sqrt{\eps}$ is a sum 
of the variables $Y_k$, with a {\em deterministic} number of terms. 
We infer from~\eqref{eq:decompo} that, for any $p \in \NN^\star$, 
\begin{equation}
\label{eq:bound_Z}
\EE \left[ Z_\eps(\alpha,\beta,\cdot)^{2p} \right] \leq 
C_p \EE \left[ B_\eps(\alpha,\beta,\cdot)^{2p} \right]
+ C_p \EE \left[ A_\eps(\alpha,\cdot)^{2p} \right] 
+ C_p \EE \left[ A_\eps(\beta,\cdot)^{2p} \right]
\end{equation}
where the constant $C_p$ only depends on $p$. We now estimate $B_\eps$,
and next $A_\eps$.

\bigskip

\noindent
\textbf{Step 1: Estimation of $B_\eps$}

Denoting by
$\dps K_\alpha= \left\lfloor \frac{\alpha}{\eps\EE(\int_0^1\phi')}
\right\rfloor$ and $\dps K_\beta= \left\lfloor
  \frac{\beta}{\eps\EE(\int_0^1\phi')} \right\rfloor$, we have
\begin{eqnarray}
B_\eps(\alpha,\beta,\omega)
&=& 
\sqrt{\eps} \left(\sum\limits_{k=1+K_\alpha}^{K_\beta-1} \int_k^{k+1}
  \psi(t) \phi'(t,\omega) dt +
  \int_{K_\beta}^{\frac{\beta}{\eps\EE(\int_0^1\phi')}} \psi(t)
  \phi'(t,\omega) dt +
  \int^{1+K_\alpha}_{\frac{\alpha}{\eps\EE(\int_0^1\phi')}} \psi(t)
  \phi'(t,\omega) dt \right)
\nonumber
\\
&=&
\sqrt{\eps} \sum\limits_{k=1+K_\alpha}^{K_\beta-1} Y_k(\omega)
+ 
\sqrt{\eps} \int_{K_\beta}^{\frac{\beta}{\eps\EE(\int_0^1\phi')}} \psi(t)
  \phi'(t,\omega) dt 
+
\sqrt{\eps} \int^{1+K_\alpha}_{\frac{\alpha}{\eps\EE(\int_0^1\phi')}} \psi(t)
  \phi'(t,\omega) dt,
\label{eq:2}
\end{eqnarray}
where we recall that $Y_k$ is defined by~\eqref{eq:def_Y}. We thus
obtain, for a deterministic constant $C_p$ that only depends on $p$,
\begin{equation}
\label{eq:bb4}
\left|
B_\eps(\alpha,\beta,\omega)
\right|^{2p} 
\leq 
C_p \eps^p \left| \sum\limits_{k=1+K_\alpha}^{K_\beta-1} Y_k(\omega) \right|^{2p}
+ 
C_p \eps^p \| \psi \|^{2p}_{L^\infty(\RR)} \| \phi' \|^{2p}_{L^\infty(\RR
  \times \Omega)}. 
\end{equation}
Recall that $(Y_k)_{k \in \ZZ}$ is a sequence of independent
identically distributed variables, with $\EE(Y_k) = 0$. We now use the
fact that any such variables satisfy the following bounds:
\begin{equation}
\label{eq:b3}
\forall p\in \NN^\star, \quad \exists C_p>0, \quad \forall N \in \NN^\star, \quad
\left| \EE\left[ \left( \frac1N \sum_{k=1}^N Y_k\right)^{2p} \right]
\right| 
\leq
\frac{C_p}{N^p} 
\end{equation}
for a constant $C_p$ that depends on $p$ and the moments of $Y_k$, up to
order $2p$. This is proved by developing the power $2p$ of the sum, and
then using the fact that the variables are i.i.d and have mean value
zero. In our case, the variables $Y_k$ are bounded almost surely, and
thus all their moments are finite. 
We thus deduce from~\eqref{eq:bb4} and~\eqref{eq:b3} that
\begin{eqnarray}
\EE \left[ B_\eps(\alpha,\beta,\cdot)^{2p} \right]
& \leq &
C_p \eps^p \EE \left[ 
\left| \sum\limits_{k=1+K_\alpha}^{K_\beta-1} Y_k \right|^{2p} \right]
+ C_p \eps^p \| \psi \|^{2p}_{L^\infty(\RR)} 
\| \phi' \|^{2p}_{L^\infty(\RR \times \Omega)}
\nonumber
\\
& \leq &
C_p \eps^p (K_\beta - K_\alpha - 1)^p 
+ C_p \eps^p \| \psi \|^{2p}_{L^\infty(\RR)} 
\| \phi' \|^{2p}_{L^\infty(\RR \times \Omega)}
\nonumber
\\
& \leq &
C_p (\beta - \alpha)^p 
+ C_p \eps^p \| \psi \|^{2p}_{L^\infty(\RR)} 
\| \phi' \|^{2p}_{L^\infty(\RR \times \Omega)}.
\label{eq:bound_B}
\end{eqnarray}

\bigskip

\noindent
\textbf{Step 2: Estimation of $A_\eps$}

We now bound $A_\eps(x,\omega)$, for any $x \in [0,1]$. Denoting 
$\dps K_x = \left\lfloor \frac{x}{\eps \EE(\int_0^1\phi')}
\right\rfloor$ and 
$\overline{K}_x(\omega) = \lfloor \phi^{-1}(x/\eps,\omega)
\rfloor$, we have 
$$
A_\eps(x,\omega) 
=
\sqrt{\eps} \left(
\int_{K_x}^{\overline{K}_x(\omega)} \psi(t) \phi'(t,\omega) dt
+ 
\int_{\frac{x}{\eps\EE(\int_0^1\phi')}}^{K_x} \psi(t) \phi'(t,\omega) dt
+
\int_{\overline{K}_x(\omega)}^{\phi^{-1}(x/\eps,\omega)} \psi(t) \phi'(t,\omega) dt
\right),
$$
hence
$$
\left| A_\eps(x,\omega) \right|
\leq
\sqrt{\eps} \left|
\int_{K_x}^{\overline{K}_x(\omega)} \psi(t) \phi'(t,\omega) dt
\right|
+ 
2 \sqrt{\eps} \| \psi \|_{L^\infty(\RR)} \| \phi' \|_{L^\infty(\RR
  \times \Omega)},
$$
thus
\begin{equation}
\label{eq:bbb3}
\EE \left[ A_\eps(x,\cdot)^{2p} \right]
\leq
C_p \eps^p \EE \left[ 
\left| \int_{K_x}^{\overline{K}_x} \psi(t) \phi'(t,\cdot) dt 
\right|^{2p} \right]
+ C_p \eps^p \| \psi \|^{2p}_{L^\infty(\RR)} 
\| \phi' \|^{2p}_{L^\infty(\RR \times \Omega)}.
\end{equation}
Let us now bound the first term of the above right-hand side. The
difficulty stems from the fact that the random variable
$\overline{K}_x(\omega)$ is {\em not} independent from the random
process $\phi'(t,\omega)$. We write,
using the bound~\eqref{eq:b3} and Young's inequality with parameter
$\dps \frac{\gamma}{j^{2p+2}} > 0$ (where $\gamma >0$ is arbitrary), that 
\begin{eqnarray}
\EE \left[ 
\left| \int_{K_x}^{\overline{K}_x} \psi(t) \phi'(t,\cdot) dt 
\right|^{2p} 
\right]
& = & \sum_{j \in \ZZ^\star}
\EE \left[ 
\left| \int_{K_x}^{\overline{K}_x} \psi(t) \phi'(t,\cdot) dt 
\right|^{2p} 
\mathbf{1}_{\overline{K}_x(\omega) = K_x + j}
\right]
\nonumber
\\
& \leq & \sum_{j \in \ZZ^\star} \frac{\gamma}{2 j^{2p+2}}
\EE \left[ 
\left| \int_{K_x}^{K_x + j} \psi(t) \phi'(t,\cdot) dt 
\right|^{4p} 
\right]
+ \frac{j^{2p+2}}{2 \gamma}
\PP \left[ 
\overline{K}_x(\omega) = K_x + j 
\right]
\nonumber
\\
& \leq & \sum_{j \in \ZZ^\star}
\frac{\gamma}{2 j^{2p+2}} \EE \left[ 
\left| \sum_{k=K_x}^{K_x + j-1} Y_k
\right|^{4p} 
\right]
+ \frac{j^{2p+2}}{2 \gamma}
\PP \left[ 
\overline{K}_x(\omega) = K_x + j
\right]
\nonumber
\\
& \leq & \sum_{j \in \ZZ^\star}
C_{2p} \, \frac{\gamma}{2 j^{2}} 
+ \frac{j^{2p+2}}{2 \gamma}
\PP \left[ 
\overline{K}_x(\omega) = K_x + j
\right]
\nonumber
\\
& \leq & C_{2p} \frac{\gamma}{2} + \frac{1}{2\gamma}\EE \left[ 
\left| \overline{K}_x - K_x \right|^{2p+2} \right].
\label{eq:bbb1}
\end{eqnarray}
We are now left with bounding from above $\dps \EE \left( \left|
    \overline{K}_x - K_x \right|^{2p+2} \right)$. To this aim, we first
bound from above $\left| \overline{K}_x(\omega) - K_x \right|^{2p+2}$: 
\begin{eqnarray*}
\left| \overline{K}_x(\omega) - K_x \right|^{2p+2} 
&\leq& 
C_p \left( \left| K_x - \phi^{-1}(x/\eps,\omega) \right|^{2p+2} + 
\left| \phi^{-1}(x/\eps,\omega) - \overline{K}_x(\omega) \right|^{2p+2}
\right) 
\\
&\leq& C_p \left( \left| K_x - \phi^{-1}(x/\eps,\omega) \right|^{2p+2} + 1 \right).
\end{eqnarray*}
Recall now that, in view of~\eqref{phi:cond1}, we have
$|a - b| \leq \nu^{-1} | \phi(a,\omega) - \phi(b,\omega)|$ for any $a$
and $b$, almost surely. We get
\begin{equation}
\left| \overline{K}_x(\omega) - K_x \right|^{2p+2} \leq 
\frac{C_p}{\nu^{2p+2}} \left( \left| \phi(K_x,\omega) - \frac{x}{\eps}
  \right|^{2p+2} + \nu^{2p+2} \right).
\label{b:eq:1}
\end{equation}
We now recall that the random variables 
$\dps D_k(\omega)=\int_k^{k+1} \phi'(t,\omega) dt$, introduced
in~\eqref{eq:def_D}, are assumed to be i.i.d. random variables. 
Writing 
$\dps \phi(x,\omega)=\phi(0,\omega)+\int_0^x \phi'(t,\omega) dt$, we
obtain that 
\begin{eqnarray*}
\left| \phi(K_x,\omega) - \frac{x}{\eps} \right|^{2p+2}
&\leq &
C_p 
\left(\left|\int_0^{K_x}\phi'(t,\omega) dt - K_x \EE(D_0) \right|^{2p+2} + \left| \phi(0,\omega) \right|^{2p+2} + \left|K_x \EE(D_0)-\frac{x}{\eps}\right|^{2p+2}\right)
\nonumber
\\
& \leq &
C_p \left(\left| \sum_{k=0}^{K_x-1} \left( D_k(\omega) - \EE(D_0) \right) \right|^{2p+2} +
\left| \phi(0,\omega) \right|^{2p+2} + \left|K_x \EE(D_0)-\frac{x}{\eps}\right|^{2p+2}\right),
\end{eqnarray*}
where, we recall, $\dps K_x = \left\lfloor \frac{x}{\eps \EE(\int_0^1\phi')}
\right\rfloor$. Observing that $\dps \EE(D_0) =
\EE\left(\int_0^1\phi'\right)$, we have 
$\dps \left| \frac{x}{\eps} - K_x \EE(D_0) \right| \leq \EE(D_0)$, thus
\begin{equation}
\left| \phi(K_x,\omega) - \frac{x}{\eps} \right|^{2p+2}
\leq 
C_p \left(\left| \sum_{k=0}^{K_x-1} \left( D_k(\omega) - \EE(D_0) \right) \right|^{2p+2} +
\left| \phi(0,\omega) \right|^{2p+2} + |\EE(D_0)|^{2p+2}  \right).
\label{b:eq:2}
\end{equation}
Collecting~\eqref{b:eq:1} and~\eqref{b:eq:2}, we obtain
\begin{eqnarray*}
\left| \overline{K}_x(\omega) - K_x \right|^{2p+2} \leq
\frac{C_p}{\nu^{2p+2}} 
\left(\left| \sum_{k=0}^{K_x-1} \left( D_k(\omega) - \EE(D_0) \right) \right|^{2p+2} +
\left| \phi(0,\omega) \right|^{2p+2} + |\EE(D_0)|^{2p+2} + \nu^{2p+2} \right).
\end{eqnarray*}
Next, we take the expectation of the above inequality and
use~\eqref{eq:b3} to get 
$$
\EE\left(\left| \overline{K}_x - K_x \right|^{2p+2} \right) \leq
\frac{C_p}{\nu^{2p+2}} 
\left( K_x^{p+1} + \EE\left(\left| \phi(0,\cdot) \right|^{2p+2}\right) + |\EE(D_0)|^{2p+2} + \nu^{2p+2} \right).
$$
Since $\dps K_x = \left\lfloor \frac{x}{\eps \EE(\int_0^1\phi')}
\right\rfloor$, we know that $K_x$ is of the order of $1/\eps$, and thus
\begin{equation}
\label{eq:bbb2}
\forall x \in (0,1), \quad
\EE\left(\left| \overline{K}_x(\omega) - K_x \right|^{2p+2} \right) \leq C_p \frac{1}{\eps^{p+1}}
\end{equation}
for a constant $C_p$ independent of $\eps$ and
$x$. We infer from~\eqref{eq:bbb1} and~\eqref{eq:bbb2} that
$$
 \forall \gamma > 0, \quad \EE \left[ 
\left| \int_{K_x}^{\overline{K}_x} \psi(t) \phi'(t,\cdot) dt 
\right|^{2p} 
\right] \leq C_p \left(\frac{\gamma}{2} + \frac{1}{2\gamma \eps^{p+1}}\right).
$$
Taking $\gamma^{-1}=\eps^{(p+1)/2}$ leads to
\begin{equation}\label{eq:bbb4}
 \EE \left[ 
\left| \int_{K_x}^{\overline{K}_x(\omega)} \psi(t) \phi'(t,\omega) dt 
\right|^{2p} 
\right] \leq C_p \frac{1}{\eps^{(p+1)/2}}.
\end{equation}
Collecting~\eqref{eq:bbb3} and~\eqref{eq:bbb4}, we obtain 
\begin{equation}
\label{eq:bound_AA}
\forall x \in (0,1), \quad
\EE \left[ A_\eps(x,\cdot)^{2p} \right]
\leq
C_p \eps^{(p-1)/2}
\end{equation}
for a constant $C_p$ independent of $\eps$ and $x$.

\bigskip

\noindent
\textbf{Step 3: Conclusion}

Collecting~\eqref{eq:bound_Z},~\eqref{eq:bound_B}
and~\eqref{eq:bound_AA} (which is legitimate since $0 \leq \alpha \leq \beta
\leq 1$), we obtain
$$
\EE \left[ Z_\eps(\alpha,\beta,\cdot)^{2p} \right] 
\leq
C_p \left[ (\beta - \alpha)^p 
+ \eps^p 
+ \eps^{(p-1)/2} \right]
\leq
C_p \left[ (\beta - \alpha)^p + \eps^{(p-1)/2} \right]
$$
where $C_p$ is a deterministic constant independent from $\alpha$, $\beta$
and $\eps$. This concludes the proof of Lemma~\ref{lemme1-s}.
\end{proof}

\begin{proof}[Proof of Lemma~\ref{lemme1}]
The result directly follows from Lemma~\ref{lemme1-s} and an integration
by part argument.
We consider the random variable $Z_\eps(\alpha,\beta,\omega)$
defined by 
\begin{equation*}
Z_\eps(\alpha,\beta,\omega)=\frac{1}{\sqrt{\eps}}
\int_\alpha^\beta
\psi\left(\phi^{-1}\left(\frac{t}{\eps},\omega\right)\right)dt. 
\end{equation*}
Integrating by part, we see that
$$
\overline{Z}_\eps(\alpha,\beta,\omega)
=
\left[ {\cal A}(t) Z_\eps(\alpha,t,\omega) \right]_\alpha^\beta
- 
\int_\alpha^\beta {\cal A}'(t) Z_\eps(\alpha,t,\omega) \, dt
=
{\cal A}(\beta) Z_\eps(\alpha,\beta,\omega) 
- 
\int_\alpha^\beta {\cal A}'(t) Z_\eps(\alpha,t,\omega) \, dt.
$$
Using the Cauchy-Schwartz inequality, we obtain
$$
\overline{Z}_\eps(\alpha,\beta,\omega)^2
\leq
2 {\cal A}(\beta)^2 Z_\eps(\alpha,\beta,\omega)^2
+
2 \int_\alpha^\beta \left( {\cal A}'(t) \right)^2 dt
\int_\alpha^\beta Z_\eps(\alpha,t,\omega)^2 \, dt.
$$
We now take the power $p$ of this estimate:
\begin{eqnarray*}
\overline{Z}_\eps(\alpha,\beta,\omega)^{2p}
& \leq &
C_p \| {\cal A} \|_{L^\infty(\alpha,\beta)}^{2p} 
Z_\eps(\alpha,\beta,\omega)^{2p}
+
C_p \| {\cal A}' \|_{L^2(\alpha,\beta)}^{2p}
\left( \int_\alpha^\beta Z_\eps(\alpha,t,\omega)^2 \, dt
\right)^p
\\
& \leq &
C_p \| {\cal A} \|_{L^\infty(\alpha,\beta)}^{2p} 
Z_\eps(\alpha,\beta,\omega)^{2p}
+
C_p \| {\cal A}' \|_{L^2(\alpha,\beta)}^{2p}
\int_\alpha^\beta Z_\eps(\alpha,t,\omega)^{2p} \, dt
\left( \int_\alpha^\beta dt \right)^{p/q}
\\
& \leq &
C_p \| {\cal A} \|_{L^\infty(\alpha,\beta)}^{2p} 
Z_\eps(\alpha,\beta,\omega)^{2p}
+
C_p (\beta-\alpha)^{p-1} \, \| {\cal A}' \|_{L^2(\alpha,\beta)}^{2p}
\int_\alpha^\beta Z_\eps(\alpha,t,\omega)^{2p} \, dt,
\end{eqnarray*}
where we have used H\"older inequality with $1=1/p + 1/q$. Using
Lemma~\ref{lemme1-s}, we thus obtain 
\begin{eqnarray*}
\EE \left[ \overline{Z}_\eps(\alpha,\beta,\cdot)^{2p} \right]
&\leq&
C_p \| {\cal A} \|_{L^\infty(\alpha,\beta)}^{2p} 
\EE \left[ Z_\eps(\alpha,\beta,\cdot)^{2p} \right] 
+
C_p (\beta-\alpha)^{p-1} \, \| {\cal A}' \|_{L^2(\alpha,\beta)}^{2p}
\int_\alpha^\beta \EE \left[ Z_\eps(\alpha,t,\cdot)^{2p} \right] \,
dt
\\
& \leq & 
C_p \| {\cal A} \|_{L^\infty(\alpha,\beta)}^{2p} 
\left[ (\beta-\alpha)^p + \eps^{(p-1)/2} \right]
+
C_p (\beta-\alpha)^{p-1} \, \| {\cal A}' \|_{L^2(\alpha,\beta)}^{2p}
\int_\alpha^\beta \left[ (t-\alpha)^p + \eps^{(p-1)/2} \right]
dt
\\
& \leq & 
C_p \| {\cal A} \|_{L^\infty(\alpha,\beta)}^{2p} 
\left[ (\beta-\alpha)^p + \eps^{(p-1)/2} \right]
+
C_p (\beta-\alpha)^{p-1} \, \| {\cal A}' \|_{L^2(\alpha,\beta)}^{2p}
\left[ (\beta-\alpha)^{p+1} + (\beta-\alpha) \eps^{(p-1)/2} \right]
\\
& \leq & 
C_p \left[ (\beta-\alpha)^p + \eps^{(p-1)/2} \right]
\left[ 
\| {\cal A} \|_{L^\infty(\alpha,\beta)}^{2p} 
+
(\beta-\alpha)^p \, \| {\cal A}' \|_{L^2(\alpha,\beta)}^{2p}
\right].
\end{eqnarray*}
This concludes the proof of Lemma~\ref{lemme1}. 
\end{proof}

\subsection{Proof of Lemma~\ref{lemme2}}
\label{sec:tech2}

By definition, 
\begin{equation*}
\overline{Z}_\eps(\alpha,\beta,\omega)
=
\frac{1}{\sqrt{\eps}} \int_\alpha^\beta {\cal A}(t) 
\psi\left(\phi^{-1}\left(\frac{t}{\eps},\omega\right)\right)dt.
\end{equation*}
We start by replacing the function ${\cal A}$ by a piecewise constant
function $\widetilde{\cal A}$, that we will choose later as an accurate
approximation of ${\cal A}$, in a sense to be made precise. 
We thus introduce the function $\widetilde{\cal A}$ defined by 
\begin{equation}
\label{eq:def_A_tilde}
\widetilde{\cal A}(t)
=
\sum\limits_{p=1}^{N} {\cal A}_p \mathbf{1}_{(t_p,t_{p+1})}(t),
\end{equation}
with 
$\alpha = t_1 < t_2 < \dots < t_{N+1} = \beta$. Hence the sets $(t_p,t_{p+1})$ are disjoint
one from another, and $\dps \cup_{1 \leq p \leq N} [t_p,t_{p+1}] =
[\alpha,\beta]$. We associate to this function $\widetilde{\cal A}$ the
random variable  
\begin{equation}
\label{eq:def_Z_tilde}
\widetilde{Z}_\eps(\alpha,\beta,\omega)
=
\frac{1}{\sqrt{\eps}} \int_\alpha^\beta \widetilde{\cal A}(t)
\psi\left(\phi^{-1}\left(\frac{t}{\eps},\omega\right)\right)dt 
= 
\frac{1}{\sqrt{\eps}} \sum\limits_{p=1}^{N} {\cal A}_p
\int_{t_p}^{t_{p+1}} 
\psi\left(\phi^{-1}\left(\frac{t}{\eps},\omega\right)\right)dt.
\end{equation}

\noindent
\textbf{Step 1: $\widetilde{Z}_\eps(\alpha,\beta,\omega)$
converges in law to a Gaussian random variable}

In view of~\eqref{eq:def_Z} and~\eqref{eq:decompo},
we have, for each $p$, 
$$
\frac{1}{\sqrt{\eps}} \int_{t_p}^{t_{p+1}}
\psi\left(\phi^{-1}\left(\frac{t}{\eps},\omega\right)\right)dt 
= 
Z_\eps(t_p,t_{p+1},\omega)
=
B_\eps(t_p,t_{p+1},\omega) + 
A_\eps(t_{p+1},\omega) - A_\eps(t_p,\omega).
$$
We can write $B_\eps$ (see~\eqref{eq:2}) as
$$
B_\eps(t_p,t_{p+1},\omega) = 
\widetilde{B}_\eps(t_p,t_{p+1},\omega) + \widetilde{R}_{\eps,p}(\omega),
$$
where
\begin{equation}
\label{eq:def_tildeB}
\widetilde{B}_\eps(t_p,t_{p+1},\omega)
= 
\sqrt{\eps} \sum\limits_{k=1+K_p}^{K_{p+1}-1} Y_k(\omega)
\end{equation}
with $K_p= \left\lfloor \frac{t_p}{\eps\EE(\int_0^1\phi')}
\right\rfloor$, and
$$
\widetilde{R}_{\eps,p}(\omega) = 
\sqrt{\eps}
\int_{K_{p+1}}^{\frac{t_{p+1}}{\eps\EE(\int_0^1\phi')}} 
\psi(t) \phi'(t,\omega) dt 
+
\sqrt{\eps} 
\int^{1+K_{p}}_{\frac{t_p}{\eps\EE(\int_0^1\phi')}} 
\psi(t) \phi'(t,\omega) dt.
$$
We hence write
\begin{equation}
\label{eq:decompo_Z_tilde}
\widetilde{Z}_\eps(\alpha,\beta,\omega)
=
\sum\limits_{p=1}^{N} {\cal A}_p \left( 
\widetilde{B}_\eps(t_p,t_{p+1},\omega) + \widetilde{R}_{\eps,p}(\omega) 
+ 
A_\eps(t_{p+1},\omega) - A_\eps(t_p,\omega)
\right).
\end{equation}
Observe that $\widetilde{R}_{\eps,p}$ satisfies 
$$
| \widetilde{R}_{\eps,p}(\omega) | \leq 2 \sqrt{\eps} 
\| \psi \|_{L^\infty(\RR)} \| \phi' \|_{L^\infty(\RR \times \Omega)},
$$
and hence goes to $0$ as $\eps \rightarrow 0$ almost surely. 

In the sequel, we first show that $A_\eps$ converges in probability and thus in
law to $0$ as $\eps$ goes to $0$, and next that $\widetilde{B}_\eps$
converges in law to a Gaussian random variable as $\eps$ goes to
$0$. 

\medskip

\noindent
\textbf{Step 1a: $A_\eps(x,\omega)$ converges in probability to 0}

For any $x \in [0,1]$, we have
\begin{eqnarray}
A_\eps(x,\omega) &=& 
\sqrt{\eps} \left(
\int_0^{\phi^{-1}(x/\eps,\omega)}\psi(t)\phi'(t,\omega) dt 
- 
\int_0^{\frac{x}{\eps\EE(\int_0^1\phi')}} \psi(t)\phi'(t,\omega)
dt 
\right) \nonumber \\
&=& \sqrt{\eps} \left( 
\sum\limits_{k=0}^{\overline{K}_x(\omega)-1}
Y_k(\omega) 
- 
\sum\limits_{k=0}^{K_x-1} Y_k(\omega)
\right) + 
R_\eps(\omega)
\label{pc_x:4}
\end{eqnarray}
where 
$K_x = \left\lfloor \frac{x}{\eps \EE(\int_0^1\phi')}
\right\rfloor$ and 
$\overline{K}_x(\omega) = \lfloor \phi^{-1}(x/\eps,\omega)
\rfloor$, and
$$
R_\eps(\omega) = 
\sqrt{\eps} 
\int_{\overline{K}_x(\omega)}^{\phi^{-1}(x/\eps,\omega)}\psi(t)\phi'(t,\omega) dt 
- \sqrt{\eps} 
\int_{K_x}^{\frac{x}{\eps\EE(\int_0^1\phi')}} \psi(t)\phi'(t,\omega) dt.
$$
We have 
\begin{equation}
\label{pc_x:5}
| R_\eps(\omega) | \leq 2 \sqrt{\eps} 
\| \psi \|_{L^\infty(\RR)} \| \phi' \|_{L^\infty(\RR \times \Omega)},
\end{equation}
hence $R_\eps$ goes to $0$ as $\eps
\rightarrow 0$ almost surely. 

For any $\lambda >0$ and $\delta >0$, we write
\begin{eqnarray}
 && \PP\left(\left| \sum\limits_{k=0}^{\overline{K}_x(\omega)-1}
     Y_k(\omega) - \sum\limits_{k=0}^{K_x-1} Y_k(\omega) \right| 
> \frac{\lambda}{\sqrt{\eps}} \right) 
\nonumber \\ 
&=& \PP\left(\left| \sum\limits_{k=0}^{\overline{K}_x(\omega)-1} 
Y_k(\omega) - \sum\limits_{k=0}^{K_x-1} Y_k(\omega) \right| 
> \frac{\lambda}{\sqrt{\eps}} \ \text{and} \
\left| \overline{K}_x(\omega) - K_x \right| < \left\lfloor
  \frac{\delta}{\eps} \right\rfloor \right) 
\nonumber \\
&& \quad + \PP\left(\left| \sum\limits_{k=0}^{\overline{K}_x(\omega)-1} 
Y_k(\omega) - \sum\limits_{k=0}^{K_x-1} Y_k(\omega) \right| 
> \frac{\lambda}{\sqrt{\eps}} \ \text{and} \
\left| \overline{K}_x(\omega) - K_x \right| \geq \left\lfloor
  \frac{\delta}{\eps} \right\rfloor \right).
\label{pc_x:1}
\end{eqnarray}
Remark that
\begin{equation}
\label{pc_x:2}
\PP\left(\left| \sum\limits_{k=0}^{\overline{K}_x(\omega)-1} 
Y_k(\omega) - \sum\limits_{k=0}^{K_x-1} Y_k(\omega) \right| >
\frac{\lambda}{\sqrt{\eps}} \ \text{and} \ 
\left| \overline{K}_x(\omega) - K_x \right| \geq \left\lfloor
  \frac{\delta}{\eps} \right\rfloor \right) 
\leq 
\PP\left(\left| \overline{K}_x(\omega) - K_x \right| \geq \left\lfloor \frac{\delta}{\eps} \right\rfloor \right).
\end{equation}
We next write, using that $Y_k$ is a sequence of independent identically
distributed variables, that
\begin{equation*}
\PP\left(\left| \sum\limits_{k=0}^{\overline{K}_x(\omega)-1} 
Y_k(\omega) - \sum\limits_{k=0}^{K_x-1} Y_k(\omega) \right| >
\frac{\lambda}{\sqrt{\eps}} \ \text{and} \
\left| \overline{K}_x(\omega) - K_x \right| < \left\lfloor
  \frac{\delta}{\eps} \right\rfloor \right) 
\leq 
\PP\left(\sup\limits_{1\leq k \leq \lfloor \delta/\eps \rfloor } \left| \sum\limits_{i=1}^{k} Y_i(\omega) \right| \geq \frac{\lambda}{\sqrt{\eps}}\right).
\end{equation*}
We now recall the Kolmogorov inequality~\cite[p~175]{Billingsley1968}:
as $Y_i$ is a sequence of i.i.d. random variables with mean zero, we have
\begin{equation*}
 \PP\left(\sup\limits_{1\leq k \leq n } \left| \sum\limits_{i=1}^{k}
     Y_i(\omega)  \right| \geq x\right) 
\leq 
x^{-2}\Var\left(\sum\limits_{i=1}^{n} Y_i \right)
=
n x^{-2} \Var\left( Y_0 \right).
\end{equation*}
We thus deduce that
\begin{equation}
\label{pc_x:3}
\PP\left(\left| \sum\limits_{k=0}^{\overline{K}_x(\omega)-1} 
Y_k(\omega) - \sum\limits_{k=0}^{K_x-1} Y_k(\omega) \right| >
\frac{\lambda}{\sqrt{\eps}} \ \text{and} \
\left| \overline{K}_x(\omega) - K_x \right| < \left\lfloor
  \frac{\delta}{\eps} \right\rfloor \right)
\leq \frac{\delta}{\lambda^2} \Var(Y_0).
\end{equation}
Collecting~\eqref{pc_x:1},~\eqref{pc_x:2} and~\eqref{pc_x:3}, we obtain
$$
\PP\left(\left| \sum\limits_{k=0}^{\overline{K}_x(\omega)-1}
     Y_k(\omega) - \sum\limits_{k=0}^{K_x-1} Y_k(\omega) \right| 
> \frac{\lambda}{\sqrt{\eps}} \right) 
\leq
\PP\left(\left| \overline{K}_x(\omega) - K_x \right| \geq \left\lfloor
    \frac{\delta}{\eps} \right\rfloor \right)
+
\frac{\delta}{\lambda^2} \Var(Y_0).
$$  
For any fixed $\lambda>0$ and any $\eta$, we choose $\delta>0$ such that
$\delta \Var(Y_0)/\lambda^2 < \eta/2$. Recall now that 
$\eps \phi^{-1}(x/\eps,\omega)$ converges to
$\dps \frac{x}{\EE(\int_0^1 \phi')}$ as $\eps \rightarrow 0$
a.s. (see~\cite{Blanc2006}), which implies that 
$\PP\left(\left| \overline{K}_x(\omega) - K_x \right| \geq \left\lfloor
    \frac{\delta}{\eps} \right\rfloor \right)$ goes to $0$ when
$\eps \rightarrow 0$. There thus exists $\eps_0$ such that, for
any $\eps \leq \eps_0$, we have 
$\PP\left(\left| \overline{K}_x(\omega) - K_x \right| \geq \left\lfloor
    \frac{\delta}{\eps} \right\rfloor \right) \leq \eta/2$, and
thus 
$$
\PP\left(\left| \sum\limits_{k=0}^{\overline{K}_x(\omega)-1}
     Y_k(\omega) - \sum\limits_{k=0}^{K_x-1} Y_k(\omega) \right| 
> \frac{\lambda}{\sqrt{\eps}} \right) 
\leq \eta.
$$
We thus have proved that, for any $\lambda>0$, we have 
$$
\lim_{\eps \to 0} 
\PP\left(\sqrt{\eps} \left| \sum\limits_{k=0}^{\overline{K}_x(\omega)-1}
     Y_k(\omega) - \sum\limits_{k=0}^{K_x-1} Y_k(\omega) \right| 
> \lambda \right) = 0.
$$
Collecting this limit with~\eqref{pc_x:4} and~\eqref{pc_x:5}, we obtain
that $A_\eps(x,\omega)$ converges in probability to $0$ as $\eps
\rightarrow 0$, for any $x$. 

\medskip

\noindent
\textbf{Step 1b: Convergence of $\widetilde{B}_\eps$ and of $\widetilde{Z}_\eps$}

Recall that $(Y_k(\omega))_{k\in\ZZ}$ is a sequence of i.i.d. variables
of mean zero (see assumption~\eqref{eq:assump_Y_iid}). Using the Central
Limit Theorem, we obtain that 
$\widetilde{B}_\eps(t_p,t_{p+1},\omega)$ defined
by~\eqref{eq:def_tildeB} converges in law to a Gaussian variable,
$$
\widetilde{B}_\eps(t_p,t_{p+1},\omega)
\overset{\mathcal{L}}{\underset{\eps\to 0}{\longrightarrow}} 
{\cal N}(0,\sigma_p),
$$
the variance of which is
$$
\sigma_p = 
\frac{t_{p+1}-t_p}{\EE(\int_0^1\phi')} \ \Var(Y_0).
$$
In addition, the random variables
$\widetilde{B}_\eps(t_p,t_{p+1},\omega)$ are independent one from
another. 

\medskip

As $\widetilde{R}_\eps(\omega)$ and $A_\eps(x,\omega)$ converge to zero
in probability for any $x$, and 
$\widetilde{B}_\eps(t_p,t_{p+1},\omega)$ converges in law for any $p$, we
deduce from~\eqref{eq:decompo_Z_tilde} 
that $\widetilde{Z}_\eps(\alpha,\beta,\omega)$ converges in
law to a Gaussian variable,
$$
\widetilde{Z}_\eps(\alpha,\beta,\omega)
\overset{\mathcal{L}}{\underset{\eps\to 0}{\longrightarrow}} 
\widetilde{Z}_0(\alpha,\beta,\omega) \sim
{\cal N}(0,\widetilde{\sigma}(\alpha,\beta)),
$$
the variance of which is
$$
\widetilde{\sigma}(\alpha,\beta)
=
\sum_{p=1}^{N} {\cal A}_p^2 \ \sigma_p
=
\sum_{p=1}^{N} {\cal A}_p^2 \ 
\frac{t_{p+1}-t_p}{\EE(\int_0^1\phi')} \ \Var(Y_0)
=
\frac{\Var(Y_0)}{\EE(\int_0^1\phi')} \
\left\| \widetilde{\cal A} \right\|^2_{L^2(\alpha,\beta)}.
$$


\noindent 
\textbf{Step 2: Convergence of the random variable 
$\overline{Z}_\eps(\alpha,\beta,\omega)$}
 
Recall that ${\cal A}$ is piecewise continuous with a finite number of
discontinuities located at $\left\{t_k\right\}_{1\leq k \leq m}$ and
that, for each $1\leq k\leq m-1$, ${\cal A}'\in
L^1(t_k,t_{k+1})$. Introduce the broken $L^1$-norm of ${\cal A}'$: 
$$
|{\cal A}'|_{L^1(\alpha,\beta)} :=
\|{\cal A}'\|_{L^1(\alpha,t_1)}
+
\sum\limits_{k=1}^{m-1} \|{\cal A}'\|_{L^1(t_k,t_{k+1})}
+
\|{\cal A}'\|_{L^1(t_m,\beta)}.
$$
Let us fix some $\eta > 0$, and let us complement the previous set of points
$(t_p)_{1 \leq p \leq N+1}$ such that 
\begin{equation}
\label{eq:choix1}
\alpha = t_1, \ 
t_{N+1} = \beta \ \text{and} \
 \ 0 < t_{p+1} - t_p \leq \eta \text{ for any }p.
\end{equation}
We set
\begin{equation}
\label{eq:choix2}
{\cal A}_p = {\cal A}(t_{p+1}^-)
\end{equation}
and consider the function $\widetilde{\cal A}$ and the
random variable $\widetilde{Z}_\eps(\alpha,\beta,\omega)$ defined
by~\eqref{eq:def_A_tilde} and~\eqref{eq:def_Z_tilde}. 

We write, for any $\xi \in \RR$,
\begin{multline}
\EE \left( e^{i \xi \overline{Z}_\eps(\alpha,\beta,\cdot)} \right)
-
\EE \left( e^{i \xi \overline{Z}_0(\alpha,\beta,\cdot)} \right)
=
\EE \left( 
e^{i \xi \overline{Z}_\eps(\alpha,\beta,\cdot)} - 
e^{i \xi \widetilde{Z}_\eps(\alpha,\beta,\cdot)}
\right)
+
\EE \left( e^{i \xi \widetilde{Z}_\eps(\alpha,\beta,\cdot)} \right)
-
\EE \left( e^{i \xi \widetilde{Z}_0(\alpha,\beta,\cdot)} \right)
\\ 
+ 
\EE \left( e^{i \xi \widetilde{Z}_0(\alpha,\beta,\cdot)} 
-
e^{i \xi \overline{Z}_0(\alpha,\beta,\cdot)} 
\right),
\label{eq:decompo_gene}
\end{multline}
where $\overline{Z}_0(\alpha,\beta,\omega)$ is a Gaussian random variable 
distributed according to 
${\cal N}(0,\overline{\sigma}(\alpha,\beta))$, with the variance
$$
\overline{\sigma}(\alpha,\beta)
=
\frac{\Var(Y_0)}{\EE(\int_0^1\phi')} \
\| {\cal A} \|^2_{L^2(\alpha,\beta)}.
$$
We successively estimate the three terms of the right-hand side
of~\eqref{eq:decompo_gene}.

\medskip

For the first term, we first see that 
\begin{equation}
\label{eq:terme_1a}
\left| \EE \left( 
e^{i \xi \overline{Z}_\eps(\alpha,\beta,\cdot)} - 
e^{i \xi \widetilde{Z}_\eps(\alpha,\beta,\cdot)}
\right) \right|
\leq 
\EE \left| 
e^{i \xi \overline{Z}_\eps(\alpha,\beta,\cdot)} - 
e^{i \xi \widetilde{Z}_\eps(\alpha,\beta,\cdot)}
\right|
\leq
|\xi| \ \EE \left| 
\overline{Z}_\eps(\alpha,\beta,\cdot) - 
\widetilde{Z}_\eps(\alpha,\beta,\cdot)
\right|.
\end{equation}
We next compute
\begin{eqnarray*}
\overline{Z}_\eps(\alpha,\beta,\omega) - 
\widetilde{Z}_\eps(\alpha,\beta,\omega)
& = &
\frac{1}{\sqrt{\eps}} \int_\alpha^\beta 
\left( {\cal A}(t) - \widetilde{\cal A}(t) \right) 
\psi\left(\phi^{-1}\left(\frac{t}{\eps},\omega\right)\right)dt
\\
&=& \sum_{p=1}^{N}
\frac{1}{\sqrt{\eps}} \int_{t_p}^{t_{p+1}}
\left( {\cal A}(t) - {\cal A}_p \right) 
\psi\left(\phi^{-1}\left(\frac{t}{\eps},\omega\right)\right)dt.
\end{eqnarray*}
Using the random variable
$$
Z_\eps(t_p,x,\omega) = 
\frac{1}{\sqrt{\eps}} \int_{t_p}^x
\psi\left(\phi^{-1}\left(\frac{t}{\eps},\omega\right)\right)dt,
$$
we write
\begin{eqnarray*}
\overline{Z}_\eps(\alpha,\beta,\omega) - 
\widetilde{Z}_\eps(\alpha,\beta,\omega)
&=& 
\sum_{p=1}^{N}
\int_{t_p}^{t_{p+1}}
\left( {\cal A}(t) - {\cal A}_p \right) 
\frac{dZ_\eps(t_p,t,\omega)}{dt}
\, dt
\\
&=&
\sum_{p=1}^{N}
\left[ ({\cal A}(t) - {\cal A}_p) Z_\eps(t_p,t,\omega) 
\right]_{t_p^+}^{t_{p+1}^-}
-
\sum\limits_{p=1}^{N}
\int_{t_p}^{t_{p+1}}
{\cal A}'(t) 
Z_\eps(t_p,t,\omega)
\, dt
\\
&=&
-
\sum\limits_{p=1}^{N}
\int_{t_p}^{t_{p+1}}
{\cal A}'(t) 
Z_\eps(t_p,t,\omega)
\, dt
\end{eqnarray*}
where we have used~\eqref{eq:choix2}. We thus have, using
Lemma~\ref{lemme1-s}, that
\begin{eqnarray*}
\EE \left| 
\overline{Z}_\eps(\alpha,\beta,\cdot) - 
\widetilde{Z}_\eps(\alpha,\beta,\cdot)
\right|
&\leq &
\sum\limits_{p=1}^{N}
\int_{t_p}^{t_{p+1}}
| {\cal A}'(t) |
\ \EE \left|  
Z_\eps(t_p,t,\cdot)
\right|
\, dt
\\
&\leq &
C \sum\limits_{p=1}^{N}
\int_{t_p}^{t_{p+1}}
| {\cal A}'(t) |
\ \left((t - t_p)^2 + \sqrt{\eps}\right)^{1/4} \ dt
\\
&\leq &
C \sum\limits_{p=1}^{N}
\left((t_{p+1} - t_p)^2 + \sqrt{\eps}\right)^{1/4}
\int_{t_p}^{t_{p+1}}
| {\cal A}'(t) |
\ dt
\end{eqnarray*}
where $C$ is a constant independent of $\eps$ and
$(t_p)_{1 \leq p \leq N+1}$. In
view of~\eqref{eq:choix1}, we have
\begin{equation}
\label{eq:terme_1b}
\EE \left| 
\overline{Z}_\eps(\alpha,\beta,\cdot) - 
\widetilde{Z}_\eps(\alpha,\beta,\cdot)
\right|\leq
C \left(\eta^2 + \sqrt{\eps}\right)^{1/4}
|{\cal A}'|_{L^1(\alpha,\beta)}.
\end{equation}
Inserting~\eqref{eq:terme_1b} in~\eqref{eq:terme_1a}, we deduce that,
for any $\eps$ and $\eta$,
\begin{equation}
\label{eq:terme_1}
\left| \EE \left( 
e^{i \xi \overline{Z}_\eps(\alpha,\beta,\cdot)} - 
e^{i \xi \widetilde{Z}_\eps(\alpha,\beta,\cdot)}
\right) \right|
\leq 
C |\xi| \,
\left(\eta^2 + \sqrt{\eps}\right)^{1/4} \, |{\cal A}'|_{L^1(\alpha,\beta)}.
\end{equation}

\medskip

We next turn to the second term of the right-hand side
of~\eqref{eq:decompo_gene}. We recall that
$\widetilde{Z}_\eps(\alpha,\beta,\omega)$ and
$\widetilde{\sigma}(\alpha,\beta)$ depend on $\eta$, through the choice
of the function $\widetilde{\cal A}$. 
For the parameter $\eta$ that we have chosen, 
$\widetilde{Z}_\eps(\alpha,\beta,\omega)$ converges in law to 
$\widetilde{Z}_0(\alpha,\beta,\omega)$ when $\eps \to 0$. Thus, there
exists $\eps_0(\eta)$, 
that depends on $\eta$ and can be chosen such that $\eps_0(\eta) \leq
\eta^4$, such that, for all $\eps < \eps_0(\eta)$,
\begin{equation}
\label{eq:terme_2}
\left|
\EE \left( e^{i \xi \widetilde{Z}_\eps(\alpha,\beta,\cdot)} \right)
-
\EE \left( e^{i \xi \widetilde{Z}_0(\alpha,\beta,\cdot)} \right)
\right|
\leq \eta.
\end{equation}

\medskip

We finally turn to the third term of the right-hand side
of~\eqref{eq:decompo_gene}. Since $\overline{Z}_0(\alpha,\beta,\omega)$
and $\widetilde{Z}_0(\alpha,\beta,\omega)$ are Gaussian random
variables, we see that
$$
\EE \left( e^{i \xi \widetilde{Z}_0(\alpha,\beta,\cdot)} 
-
e^{i \xi \overline{Z}_0(\alpha,\beta,\cdot)} 
\right)
=
\exp(-\xi^2 \widetilde{\sigma}(\alpha,\beta)/2) 
-
\exp(-\xi^2 \overline{\sigma}(\alpha,\beta)/2).
$$
Denoting by $L$ the Lipschitz constant of the function 
$\sigma \mapsto \exp(-\xi^2 \sigma/2)$ on $[0,\infty)$, we thus have
\begin{equation}
\left|
\EE \left( e^{i \xi \widetilde{Z}_0(\alpha,\beta,\cdot)} 
-
e^{i \xi \overline{Z}_0(\alpha,\beta,\cdot)} 
\right)
\right|
\leq
L \left| \widetilde{\sigma}(\alpha,\beta)
-
\overline{\sigma}(\alpha,\beta)
\right|
=
L \frac{\Var(Y_0)}{\EE(\int_0^1\phi')}
\left| \left\| \widetilde{\cal A} \right\|^2_{L^2(\alpha,\beta)}
-
\| {\cal A} \|^2_{L^2(\alpha,\beta)}
\right|.
\label{eq:terme_3a}
\end{equation}
We next write
\begin{equation}
\label{eq:terme_3b}
\left\| \widetilde{\cal A} \right\|^2_{L^2(\alpha,\beta)}
-
\| {\cal A} \|^2_{L^2(\alpha,\beta)}
=
\sum\limits_{p=1}^{N} \int_{t_p}^{t_{p+1}}
( {\cal A}_p^2 - {\cal A}(t)^2 ) \, dt
=
\sum\limits_{p=1}^{N} \int_{t_p}^{t_{p+1}}
( {\cal A}_p + {\cal A}(t) ) ( {\cal A}_p - {\cal A}(t) ) \, dt.
\end{equation}
In view of~\eqref{eq:choix2}, we have
$$
\forall t \in [t_p,t_{p+1}], \quad
{\cal A}(t) 
= 
{\cal A}(t_{p+1}^-) - \int_t^{t_{p+1}} {\cal A}'(s) \, ds
= 
{\cal A}_p - \int_t^{t_{p+1}} {\cal A}'(s) \, ds.
$$
Inserting this relation in~\eqref{eq:terme_3b}, we obtain
$$
\left\| \widetilde{\cal A} \right\|^2_{L^2(\alpha,\beta)}
-
\| {\cal A} \|^2_{L^2(\alpha,\beta)}
=
\sum\limits_{p=1}^{N} \int_{t_p}^{t_{p+1}}
( {\cal A}_p + {\cal A}(t) ) \int_t^{t_{p+1}} {\cal A}'(s) \, ds \, dt.
$$
Thus, in view of the choice~\eqref{eq:choix1}, we have
\begin{eqnarray}
\left|
\left\| \widetilde{\cal A} \right\|^2_{L^2(\alpha,\beta)}
-
\| {\cal A} \|^2_{L^2(\alpha,\beta)}
\right|
& \leq &
2 \| {\cal A} \|_{L^\infty(\alpha,\beta)}
\sum\limits_{p=1}^{N} \int_{t_p}^{t_{p+1}}
\int_t^{t_{p+1}} \left| {\cal A}'(s) \right| \, ds \, dt
\nonumber
\\
& \leq &
2 \| {\cal A} \|_{L^\infty(\alpha,\beta)}
\sum\limits_{p=1}^{N} \int_{t_p}^{t_{p+1}}
\left| {\cal A}'(s) \right| \, (s-t_p) \, ds 
\nonumber
\\
& \leq &
2 \eta \| {\cal A} \|_{L^\infty(\alpha,\beta)}
\sum\limits_{p=1}^{N} \int_{t_p}^{t_{p+1}}
\left| {\cal A}'(s) \right| \, ds 
\nonumber
\\
& \leq &
2 \eta \| {\cal A} \|_{L^\infty(\alpha,\beta)}
|{\cal A}'|_{L^1(\alpha,\beta)}.
\label{eq:terme_3c}
\end{eqnarray}
Inserting~\eqref{eq:terme_3c} in~\eqref{eq:terme_3a}, we deduce that
\begin{equation}
\left|
\EE \left( e^{i \xi \widetilde{Z}_0(\alpha,\beta,\cdot)} 
-
e^{i \xi \overline{Z}_0(\alpha,\beta,\cdot)} 
\right)
\right|
\leq
2 \eta L \frac{\Var(Y_0)}{\EE(\int_0^1\phi')} 
\| {\cal A} \|_{L^\infty(\alpha,\beta)}
|{\cal A}'|_{L^1(\alpha,\beta)}.
\label{eq:terme_3}
\end{equation}

\medskip

Collecting~\eqref{eq:decompo_gene}, \eqref{eq:terme_1},
\eqref{eq:terme_2} and~\eqref{eq:terme_3}, we have, for any $\eta$ and
any $\eps<\eps_0(\eta) \leq \eta^4$, that
$$
\left| 
\EE \left( e^{i \xi \overline{Z}_\eps(\alpha,\beta,\cdot)} \right)
-
\EE \left( e^{i \xi \overline{Z}_0(\alpha,\beta,\cdot)} \right)
\right|
\leq
C |\xi| \left( 2\eta^2 \right)^{1/4} \ |{\cal A}'|_{L^1(\alpha,\beta)}
+
\eta
+ 
2 \eta L \frac{\Var(Y_0)}{\EE(\int_0^1\phi')}
\| {\cal A} \|_{L^\infty(\alpha,\beta)}
|{\cal A}'|_{L^1(\alpha,\beta)}.
$$
The above bound holds for any $\eps < \eps_0(\eta)$, and $\eta$ is
arbitrary small. In addition, $|{\cal A}'|_{L^1(\alpha,\beta)}$ is independent
from $\eta$, even though the set of points $(t_p)_{1 \leq p \leq N+1}$
depends on $\eta$. This means that 
$$
\lim_{\eps \to 0} 
\left| 
\EE \left( e^{i \xi \overline{Z}_\eps(\alpha,\beta,\omega)} \right)
-
\EE \left( e^{i \xi \overline{Z}_0(\alpha,\beta,\omega)} \right)
\right|
= 0,
$$
hence $\overline{Z}_\eps(\alpha,\beta,\omega)$ converges in law to 
$\overline{Z}_0(\alpha,\beta,\omega)$. This concludes the proof of
Lemma~\ref{lemme2}.  

\subsection{Proof of Lemma~\ref{lemme3}}
\label{sec:tech3}

According to the definition~\eqref{def:Heps}, we have:
\begin{equation}
\label{eq:diff_H}
\left| H_\eps(x,\omega) - H_\eps(y,\omega) \right| 
\leq 
\left| \frac{1}{\sqrt{\eps}} \int_y^x {\cal A}_1(t) 
\psi\left(\phi^{-1}\left(\frac{t}{\eps},\omega\right)\right) dt 
\right| 
+ 
|x-y| \left| \frac{1}{\sqrt{\eps}} \int_0^1 {\cal A}_2(t) 
\psi\left(\phi^{-1}\left(\frac{t}{\eps},\omega\right)\right) dt\right|.
\end{equation}
Thus, for any $p \in \NN^\star$, using Lemma~\ref{lemme1} and the fact
that $|y-x| \leq 1$, we have
\begin{eqnarray*}
\EE \left[ \left| H_\eps(x,\cdot) - H_\eps(y,\cdot) \right|^{2p} \right]
& \leq & 
C_p \EE \left[ \left| \frac{1}{\sqrt{\eps}} \int_y^x {\cal A}_1(t) 
\psi\left(\phi^{-1}\left(\frac{t}{\eps},\cdot\right)\right) dt 
\right|^{2p} \right] 
\\
&& \quad + C_p |x-y|^{2p} \EE \left[ 
\left| \frac{1}{\sqrt{\eps}} \int_0^1 {\cal A}_2(t) 
\psi\left(\phi^{-1}\left(\frac{t}{\eps},\cdot\right)\right)
dt\right|^{2p} \right] 
\\
& \leq &
C_p \left( |x-y|^p + \eps^{(p-1)/2} \right)
+ 
C_p |x-y|^{2p} 
\\
& \leq &
C_p \left( |x-y|^p + \eps^{(p-1)/2} \right).
\end{eqnarray*}
This concludes the proof of~\eqref{eq:kol2}.

Assume now that $|x-y| \leq \eps$. We infer from~\eqref{eq:diff_H} that 
$$
\left| H_\eps(x,\omega) - H_\eps(y,\omega) \right| 
\leq 
\left| \frac{x-y}{\sqrt{\eps}} \right| \ \| \psi \|_{L^\infty(\RR)} 
\left( \| {\cal A}_1 \|_{L^\infty(0,1)} + 
\| {\cal A}_2 \|_{L^\infty(0,1)} \right)
\leq
C \sqrt{|x-y|},
$$
where $C$ is a deterministic constant independent of $\eps$, $x$ and
$y$. This concludes the proof of~\eqref{eq:kol3}, and hence the proof of
Lemma~\ref{lemme3}.

\section{Approximation of the homogenized matrix}
\label{sec:proofAN}

The aim of this section is to prove our second main result,
Theorem~\ref{conv:AN}. Since the approach described in
Section~\ref{sec:pres_approx} mimicks
the approach proposed in~\cite{bourgeat_ihp}, our proof essentially
follows the arguments used in~\cite{bourgeat_ihp}. Because our proof is
involved, we feel that it is useful to first recall the arguments
of~\cite{bourgeat_ihp} in Section~\ref{sec:bourgeat}. 
We then collect some technical results in Section~\ref{sec:AN_tech},
before turning to the actual proof of Theorem~\ref{conv:AN} in
Sections~\ref{sec:actual_proof} and~\ref{sec:proof_j_ball}.

\subsection{Convergence proof in the classical random homogenization
  setting}
\label{sec:bourgeat}

Consider the classical random homogenization problem
$$
-\mbox{div}\left[A\left(\frac{x}{\eps},\omega\right)
   \nabla u^\eps(x,\omega) \right] =f(x) \mbox{ in $\mathcal{D}$}, 
\qquad
u^\eps(\cdot,\omega) =0 \mbox{ on $\partial \mathcal{D}$},
$$
where $\mathcal{D}$ is a bounded open set of $\RR^d$, $f \in
L^2(\mathcal{D})$, and $A$ is a stationary matrix in the sense
of~\eqref{eq:stationnarite-disc}, satisfying classical coercivity and
boundedness properties. The associated homogenized problem
is~\eqref{PB:homo}, where the homogenized matrix is given by
$$
\forall 1 \leq i,j \leq d, \quad
A^\star_{ij} = \EE \left[ \int_Q e_i^T 
A\left(y,\cdot\right) 
\left( e_j + \nabla w_{e_j}(y,\cdot) \right) dy \right],
$$
where $Q=(0,1)^d$ and where, for all $p\in\RR^d$, $w_p$ solves the
corrector problem
$$
 \left\{ 
\begin{array}{l}
\dps
 -\mbox{div}\left[A(y,\omega)
   \left(p+\nabla w_p(y,\omega)\right) \right] = 0 \mbox{ in } \RR^d, 
\\
\nabla w_p \mbox{ is stationary in the sense
  of~\eqref{eq:stationnarite-disc},}
\quad
\dps \EE\left(\int_Q \nabla w_p(y,\cdot) dy \right)=0.
\end{array}
\right.
$$
In~\cite{bourgeat_ihp}, the following approximation strategy is
proposed: introduce the approximate corrector
$w_p^N(\cdot,\omega)$ as the $Q_N$-periodic function satisfying:
\begin{equation}
\label{PB:corrN-per_class}
\text{for all $\psi$ $Q_N$-periodic,}
\quad
\int_{Q_N} \left( \nabla \psi \right)^T A(\cdot,\omega) \left(p + \nabla w^N_p(\cdot,\omega)
\right)=0
\qquad
\text{with } \int_{Q_N} w^N_p(\cdot,\omega) = 0
\end{equation}
and the approximate homogenized matrix $A^\star_N(\omega)$
defined by, for any $1 \leq i,j \leq d$,
\begin{equation}
\label{eq:mat_homog_class}
\left[ A^\star_N \right]_{ij}(\omega)
=
\frac{1}{|Q_N|} \int_{Q_N} e_i^T A(\cdot,\omega)
\left(e_j + \nabla w^N_{e_j}(\cdot,\omega) \right).
\end{equation}
Then (see~\cite[Theorem 1]{bourgeat_ihp}), we have that 
\begin{equation}
\label{eq:conv_class}
\lim_{N \to \infty} A^\star_N(\omega) = A^\star \quad \text{almost
  surely}.
\end{equation} 
A key ingredient of the proof is the following classical homogenization result
(see~\cite[Theorem 5.2 p.~151]{jikov}):
\begin{theo}
\label{Jik-conv}
Let $A^\eps$ be a sequence of matrices that $G$-converges to $A^\star$
in a domain $V$, and let $V_1$ be an arbitrary subdomain of $V$. Let $p
\in \RR^d$, and assume that the functions $w_p^\eps \in H^1(V_1)$
satisfy the conditions  
$$
w_p^\eps \rightharpoonup w_p^\infty \quad \text{weakly in $H^1(V_1)$}, 
\quad \text{and} \quad 
-\hbox{div}\left[A^\eps (p + \nabla w_p^\eps)\right] = 0 \quad 
\text{in $\mathcal{D}'(V_1)$}.
$$
Then we have that
$$
A^\eps (p + \nabla w_p^\eps) \rightharpoonup A^\star (p + \nabla
w_p^\infty) 
\quad \text{weakly in $(L^2(V_1))^d$},
$$
where $w_p^\infty$ satisfies
$$
-\hbox{div}\left[A^\star (p + \nabla w_p^\infty)\right]=0 \quad 
\text{in $\mathcal{D}'(V_1)$}.
$$
\end{theo}
The proof of~\eqref{eq:conv_class} goes as follows
(see~\cite{bourgeat_ihp} for details). The rescaled corrector
$$
w_{0,p}^N(x,\omega) := \frac{1}{N} w_p^N(Nx,\omega)
$$
is shown to satisfy the a priori bound 
$ \dps
\| w_{0,p}^N(\cdot,\omega) \|_{H^1(Q)} \leq C,
$
where $C$ is a deterministic constant independent from $N$. We thus
deduce that, almost surely, there exists a $Q$-periodic function 
$w^\infty_{0,p}(\cdot,\omega) \in H^1(Q)$ such that
$$
w_{0,p}^N(\cdot,\omega) \rightharpoonup w^\infty_{0,p}(\cdot,\omega) 
\quad \text{weakly in $H^1(Q)$}.
$$
Consider a $Q$-periodic function $\psi \in H^1(Q)$. Choosing $\psi_N(y)
= \psi(y/N)$ as test function in~\eqref{PB:corrN-per_class}, we obtain
\begin{equation}
\label{eq:rescale_class}
\int_Q \left( \nabla \psi \right)^T A(N \cdot,\omega) \left(p + \nabla w^N_{0,p}(\cdot,\omega)
\right)=0.
\end{equation}
We are then in position to use Theorem~\ref{Jik-conv} on the domain
$V_1 = Q$. We thus get that $\dps A(N \cdot,\omega) (p + \nabla
w_{0,p}^N)$ weakly converges to $A^\star (p + \nabla w_{0,p}^\infty)$ in
$(L^2(Q))^d$. We then infer from~\eqref{eq:rescale_class} that, for any
$Q$-periodic function $\psi$, we have
\begin{equation}
\label{eq:rescale_lim}
\int_Q \left( \nabla \psi \right)^T A^\star \left(p + \nabla w^\infty_{0,p}(\cdot,\omega)
\right)=0.
\end{equation}
This implies that $\nabla w^\infty_{0,p}(\cdot,\omega) =
0$. Using the same weak $L^2$
convergence as above, we deduce from~\eqref{eq:mat_homog_class} that
$$
\left[ A^\star_N \right]_{ij}(\omega)
=
\int_Q e_i^T A(N \cdot,\omega) \left(e_j + \nabla w^N_{0,e_j}(\cdot,\omega) \right)
\to
\int_Q e_i^T A^\star \left(e_j + \nabla w^\infty_{0,e_j}(\cdot,\omega) \right)
=
\left[ A^\star \right]_{ij}.
$$
This concludes the proof of~\eqref{eq:conv_class}.

\subsection{Some technical ingredients for our analysis}
\label{sec:AN_tech}

A key ingredient to prove Theorem~\ref{conv:AN} is to find an
appropriate domain on which to apply Theorem~\ref{Jik-conv}. The
following lemmas are useful for that purpose. 

We first recall (see~\cite[Lemme~2.1]{Blanc2006}) that 
$\dps \frac{1}{N}\phi(N\cdot,\omega)$
converges to $\dps \EE\left( \int_Q \nabla \phi \right) \, \cdot$ in
$L^\infty_{\rm loc}(\RR^d)$ almost surely. Likewise, in view of the
proof of~\cite[Lemme~2.2]{Blanc2006}, we have that 
$\dps \frac{1}{N}\phi^{-1}(N\cdot,\omega)$
converges to $\dps \left[ \EE\left( \int_Q \nabla \phi \right) \right]^{-1} \cdot$ in
$L^\infty_{\rm loc}(\RR^d)$ almost surely.
The functions being smooth, we thus have that, for any compact $K$,
\begin{equation}
\label{eq:reecriture_cras}
\lim_{N \to \infty} \left\|
\frac{1}{N}\phi(N\cdot,\omega) - \EE\left( \int_Q \nabla \phi \right) \, \cdot
\right\|_{C^0(K)}
=
\lim_{N \to \infty} \left\|
\frac{1}{N}\phi^{-1}(N\cdot,\omega) - \left[ \EE\left( \int_Q \nabla \phi \right) \right]^{-1} \cdot
\right\|_{C^0(K)}
= 
0 \ \ \text{a.s.}
\end{equation}
As pointed out in the proof of~\cite[Lemme~2.2]{Blanc2006}, a
consequence of the above fact is that
\begin{equation}
\label{eq:reecriture_cras2}
\lim_{N \to \infty} \left\| 
\mathbf{1}_{\frac{1}{N}\phi(N Q,\omega)} 
- 
\mathbf{1}_{\EE\left( \int_Q \nabla \phi \right) Q} \right\|_{L^1(\RR^d)} 
= 
0 \ \ \text{a.s.}
\end{equation}
This can be shown by first assuming that $\phi(0,\omega) = 0$, and using
a regularization of the indicator functions. The general case
$\phi(0,\omega) \neq 0$ next follows as an easy consequence.

\medskip

The first ingredient we need to prove Theorem~\ref{conv:AN} is the
following lemma, which is somewhat related with the above results: 
\begin{lemme}
\label{theo:phiQN}
Let $\phi$ be a diffeomorphism that
satisfies~\eqref{phi:cond1},~\eqref{phi:cond2}
and~\eqref{phi:cond3}. For any compact set $K$ that is a proper subset
of the open set $\dps \EE \left(\int_Q \nabla \phi \right) Q$, and
almost all $\omega$, there exists $N_0(\omega) \in \NN$ such that
$$
\forall N \geq N_0(\omega), \quad 
\overset{\circ}{K} \subset \frac{1}{N}\phi(Q_N,\omega),
$$
where $\overset{\circ}{K}$ denotes the interior of the set $K$ and, we
recall, $Q_N = N \, Q$.
\end{lemme}

The following easy result is useful to prove Lemma~\ref{theo:phiQN}:
\begin{lemme}
\label{lem:lip}
Let $\phi$ be a diffeomorphism that
satisfies~\eqref{phi:cond1} and~\eqref{phi:cond2}. Then there exists a
deterministic constant $L_{\rm Lip}$ such that the diffeomorphism
$\phi^{-1}(\cdot,\omega)$ is Lipschitz with that constant.
\end{lemme}

\begin{proof}[Proof of Lemma~\ref{lem:lip}]
We infer from~\eqref{phi:cond2} that $\nabla \phi^T \nabla \phi$, which
is a symmetric matrix and therefore diagonalizable, has a bounded
spectrum. The assumption~\eqref{phi:cond1} then implies that the
eigenvalues of $\nabla \phi^T \nabla \phi$ are bounded away from
$0$. Hence, there exists a deterministic constant $c>0$ such
that for all $\xi \in \RR^d$ we have
$$
\xi^T (\nabla \phi(x,\omega)^T \nabla \phi(x,\omega) \xi \geq 
c |\xi|^2 \quad \text{a.s., a.e. on $\RR^d$}.
$$
For any $\overline{\xi} \in \RR^d$, we set 
$\xi=(\nabla \phi(x,\omega))^{-1} \overline{\xi}$ and obtain that 
\begin{equation}
\label{eq:utile}
\left| (\nabla \phi(x,\omega))^{-1} \overline{\xi} \right|
\leq c^{-1/2} \, \left| \overline{\xi} \right|.
\end{equation} 
The diffeomorphism $\phi^{-1}(\cdot,\omega)$ is thus Lipshitz with the
deterministic constant $c^{-1/2}$.
\end{proof}

\begin{proof}[Proof of Lemma~\ref{theo:phiQN}]
Let $K$ be a proper subset of the
open set $\dps \EE \left( \int_Q \nabla \phi \right) Q$, and let us fix
$\omega$ such that
\begin{equation}
\label{eq:choix_omega}
\frac{1}{N}\phi^{-1}(N \cdot,\omega)
\text{ converges to } 
\left[ \EE \left(\int_Q \nabla \phi\right) \right]^{-1} \cdot 
\text{ in $C^0(K)$}.
\end{equation} 
In view of~\eqref{eq:reecriture_cras}, we know
that~\eqref{eq:choix_omega} holds for almost all $\omega$. 

We prove Lemma~\ref{theo:phiQN} by contradiction. 
Suppose that, for all $N_0 \in \NN$, there exists
$N(N_0,\omega) \geq N_0$ such that $\overset{\circ}{K}$ is not included
in $\dps
\frac{1}{N(N_0,\omega)}\phi(Q_{N(N_0,\omega)},\omega)$. Otherwise
stated, there exist $N(N_0,\omega)$ and $z(N_0,\omega)$ such that 
$$
z(N_0,\omega) \in \overset{\circ}{K}  
\quad \text{and} \quad 
z(N_0,\omega) \notin \frac{1}{N(N_0,\omega)}\phi(N(N_0,\omega)Q,\omega).
$$
Introduce $\dps y(N_0,\omega) =
\frac{1}{N(N_0,\omega)}\phi^{-1}(N(N_0,\omega)z(N_0,\omega),\omega)$. We
thus have that 
\begin{equation}
\label{eq:flor1}
y(N_0,\omega) \notin Q.
\end{equation}
We now pass to the limit $N_0 \to \infty$. 
Observing that $z(N_0,\omega)$ belongs to the compact set $K$, we deduce that
$\left\{ z(N_0,\omega) \right\}_{N_0 \in \NN}$ is a bounded sequence and thus
converges, up to the extraction of a subsequence, toward some
$\overline{z}(\omega) \in K$.

Let us now show that $\left\{y(N_0,\omega)\right\}_{N_0 \in \NN}$ is also a
bounded sequence. Using the fact that the diffeomorphism
$\phi^{-1}(\cdot,\omega)$ is 
a Lipschitz mapping with a deterministic constant $L_{\rm Lip}$
(see Lemma~\ref{lem:lip}), we write
$$
|y(N_0,\omega)|
=
\frac{1}{N(N_0,\omega)} 
\left| \phi^{-1}(N(N_0,\omega)z(N_0,\omega),\omega) \right| 
\leq 
L_{\rm Lip} |z(N_0,\omega)| + \frac{1}{N(N_0,\omega)} |\phi^{-1}(0,\omega)|.
$$
We deduce that, almost surely, $\left\{ y(N_0,\omega) \right\}_{N_0 \in
  \NN}$ is a bounded 
sequence and thus converges, up to the extraction of a subsequence,
toward some $\overline{y}(\omega)$. In view of~\eqref{eq:flor1}, and since $Q$
is an open set, we have that $\overline{y}(\omega) \notin Q$. 

We now claim that 
\begin{equation}
\label{eq:claim}
\overline{z}(\omega) = \EE \left(\int_Q \nabla \phi\right) 
\overline{y}(\omega).
\end{equation} 
Indeed, we write that
\begin{eqnarray*}
&& \left| y(N_0,\omega) 
- 
\left[ \EE \left(\int_Q \nabla \phi\right) \right]^{-1} \overline{z}(\omega)
\right|
\\
&\leq&
\hspace{-2mm}
\left| 
\frac{1}{N(N_0,\omega)} \phi^{-1}(N(N_0,\omega)z(N_0,\omega),\omega) 
- 
\left[ \EE \left(\int_Q \nabla \phi\right) \right]^{-1} z(N_0,\omega)
\right| 
+
\left|
\left[ \EE \left(\int_Q \nabla \phi\right) \right]^{-1} z(N_0,\omega)
-
\left[ \EE \left(\int_Q \nabla \phi\right) \right]^{-1} \overline{z}(\omega)
\right|
\\
&\leq&
\left\| 
\frac{1}{N(N_0,\omega)} \phi^{-1}(N(N_0,\omega) \cdot,\omega) 
- 
\left[ \EE \left(\int_Q \nabla \phi\right) \right]^{-1} \cdot
\right\|_{C^0(K)} 
+
C \left| z(N_0,\omega) - \overline{z}(\omega) \right|.
\end{eqnarray*}
Both terms converge to 0 when $N_0 \to \infty$, respectively in view of~\eqref{eq:choix_omega}
and of the definition of $\overline{z}(\omega)$. By definition of
$\overline{y}(\omega)$, we deduce~\eqref{eq:claim}.  

We now reach a contradiction since $\dps \overline{z}(\omega) \in K
\subset \EE \left(\int_Q \nabla \phi \right) Q$ whereas
$\overline{y}(\omega) \notin Q$. This concludes the proof of
Lemma~\ref{theo:phiQN}.
\end{proof}

The second ingredient we need to prove Theorem~\ref{conv:AN} is the
following lemma: 
\begin{lemme}
\label{lem:inclusion}
Let $\phi$ be a diffeomorphism that
satisfies~\eqref{phi:cond1} and~\eqref{phi:cond2}.
There exists an open set $\widetilde{Q}(\omega)$ and
some $k(\omega) \in \NN$ such that
\begin{align}
\label{prop:i}
\forall N \in \NN^\star, \quad
&
\frac{1}{N} \phi(Q_N,\omega) \subset \widetilde{Q}(\omega),
\\
\label{prop:ii}
&
\EE \left(\int_Q \nabla \phi \right) Q \subset \widetilde{Q}(\omega), 
\\
\label{prop:iii}
\forall N \in \NN^\star, \quad
&
\widetilde{Q}(\omega) \subset \frac{1}{N}\phi(Q_{k(\omega)N},\omega).
\end{align}
\end{lemme}

\begin{proof}
The first assertion relies on the fact that, in view
of~\eqref{phi:cond2}, we have 
$$
\forall N \in\NN^\star, \quad
\frac{1}{N} |\phi(Nx,\omega)|
\leq M |x| + \frac{1}{N} |\phi^{-1}(0,\omega)| 
\leq M + |\phi^{-1}(0,\omega)| 
\quad 
\text{a.s., a.e. on $Q$}.
$$
It is thus sufficient to choose $\widetilde{Q}(\omega)$ such that 
$\Big[ -M - |\phi^{-1}(0,\omega)|, M + |\phi^{-1}(0,\omega)| \Big]^d 
\subset \widetilde{Q}(\omega)$.
Upon choosing a larger $\widetilde{Q}(\omega)$, the second assertion is
also satisfied. Now that $\widetilde{Q}(\omega)$ is chosen, we show
that we can 
choose $k(\omega)$ such that the third assertion is satisfied. 
Using Lemma~\ref{lem:lip}, we see that, almost surely,
$$
\forall N \in\NN^\star, \quad
\frac{1}{N} |\phi^{-1}(Nx,\omega)| 
\leq 
L_{\rm Lip} |x| + \frac{1}{N} |\phi(0,\omega)|
\leq 
L_{\rm Lip} |x| + |\phi(0,\omega)|
\quad 
\text{a.e. on $\RR^d$}.
$$
There thus exists $k(\omega)$ such that, for any $N \in\NN^\star$, we have
$\dps \frac{1}{N} \phi^{-1}\left(N \widetilde{Q}(\omega),\omega\right)
\subset Q_{k(\omega)}$. This implies the third assertion and concludes
the proof.
\end{proof}

\subsection{Proof of Theorem~\ref{conv:AN}}
\label{sec:actual_proof}

To simplify the notation, we introduce the matrix
\begin{equation}
\label{eq:def_alpha}
\alpha := \EE\left(\int_Q \nabla \phi\right) \in \RR^{d \times d}.
\end{equation}
As pointed out in~\cite[Remark 1.9]{Blanc2007}, we have that
\begin{equation}
\label{eq:signe}
\det \alpha = \EE \left( \int_Q \det \nabla \phi \right).
\end{equation} 
We hence deduce from~\eqref{phi:cond1} that
$$
\det \alpha \geq \nu > 0.
$$
We also introduce the matrix $\beta \in \RR^{d \times d}$ defined by
\begin{equation}
\label{eq:def_beta}
\beta 
= 
\EE\left[ \int_{\phi(Q,\cdot)} \left( 
\nabla \phi\left(\phi^{-1}(x,\cdot),\cdot\right) \right)^{-1} \, dx \right]
=
\EE \left[ \int_Q \det(\nabla \phi) \, (\nabla \phi)^{-1} \right].
\end{equation}
The proof of the following lemma, useful for proving
Theorem~\ref{conv:AN}, is postponed until
Section~\ref{sec:proof_j_ball}. 

\begin{lemme}
\label{lem:j_ball}
The constant matrix $\beta A^\star \alpha^{-T}$ is
coercive. 
\end{lemme}

The proof of Theorem~\ref{conv:AN} is composed of four steps. In 
Step 1, we introduce a rescaled corrector, denoted $w_{0,p}^N(\cdot,\omega)$
(see~\eqref{def:w0pN} below), and show that it converges toward some
function $w_{0,p}^\infty(\cdot,\omega)$ weakly in $H^1$. Then, in Step
2, we prove that $w_{0,p}^\infty(\cdot,\omega)$ is $\alpha
Q$-periodic. Next, in Step 3, we show 
that $w_{0,p}^\infty(\cdot,\omega)$ solves the equation 
$-\hbox{div}\left[B \nabla w_{0,p}^\infty\right]=0$ in
$\RR^d$ for a constant deterministic matrix $B$
(see~\eqref{eq:conc_step3} below for a precise statement). Combining these
results and using Lemma~\ref{lem:j_ball}, we conclude 
that $\nabla w_{0,p}^\infty \equiv 0$. This is a key ingredient to
prove, in Step 4, that the random approximation $A^\star_N(\omega)$
indeed converges to the homogenized matrix $A^\star$ almost surely. 

\medskip

\noindent
\textbf{Step 1: Introduction of a rescaled corrector $w_{0,p}^N$, and
  convergence of $w_{0,p}^N$ to some $w_{0,p}^\infty$}

We first establish some a priori bounds.
Taking $\widetilde{\psi} = \widetilde{w}_p^N$ as test function
in~\eqref{PB:corrN-per}, and using~\eqref{phi:cond1}
and~\eqref{phi:cond2}, we see that 
$$
\| (\nabla \phi(\cdot,\omega))^{-T} 
\nabla \widetilde{w}_p^N(\cdot,\omega) \|_{L^2(Q_N)} \leq C \sqrt{|Q_N|},
$$
where $C$ is a deterministic constant independent from $N$.
Using again~\eqref{phi:cond2}, we deduce that
$$
\| \nabla \widetilde{w}_p^N(\cdot,\omega) \|_{L^2(Q_N)} \leq C \sqrt{|Q_N|}.
$$
Let $k \in \NN$. Since $\widetilde{w}_p^N$ is $Q_N$-periodic, we
infer from the above bound that 
\begin{equation}
\label{eq:flor3}
\| \nabla \widetilde{w}_p^N(\cdot,\omega) \|_{L^2(Q_{kN})} \leq 
C \sqrt{|Q_{kN}|},
\end{equation}
where $C$ is a deterministic constant independent from $N$ and $k$.

We now introduce the following rescaled corrector:
\begin{equation}
\label{def:w0pN}
w_{0,p}^N(x,\omega) = \frac{1}{N} w_p^N(Nx,\omega),
\end{equation}
where, we recall
$\dps
w_p^N(y,\omega) = \widetilde{w}_p^N(\phi^{-1}(y,\omega),\omega)$.
Using~\eqref{phi:cond2} and~\eqref{eq:utile}, we infer
from~\eqref{eq:flor3} that
$$
\| \nabla w_{0,p}^N \|_{L^2\left(\frac{1}{N}\phi\left(Q_{kN},\omega\right)\right)} 
\leq C k^{d/2}
$$
where $C$ is a deterministic constant independent from $N$ and $k$. We
now choose $k$ in the above bound equal to the integer $k(\omega)$
defined in Lemma~\ref{lem:inclusion}. We infer from the above bound
and~\eqref{prop:iii} that
\begin{equation}
\label{bornw0pN_bis}
\forall N \in \NN^\star, \quad
\| \nabla w_{0,p}^N \|_{L^2\left(\widetilde{Q}(\omega)\right)} 
\leq C(\omega).
\end{equation}

Recall that the solution $\widetilde{w}_p^N$ to~\eqref{PB:corrN-per} is
unique up to an additive constant. We now fix this constant by choosing
$\widetilde{w}_p^N$ such that $\dps \int_{N \widetilde{Q}(\omega)}
w^N_p(\cdot,\omega) = 0$, where the set $\widetilde{Q}(\omega)$ is
defined in Lemma~\ref{lem:inclusion}. In view of~\eqref{def:w0pN}, this
means that $\dps \int_{\widetilde{Q}(\omega)} w^N_{0,p}(\cdot,\omega) =
0$. Using~\eqref{bornw0pN_bis} and the Poincar\'e-Wirtinger inequality,
we deduce that there exists $C(\omega)$ such that
$$
\forall N \in \NN^\star, \quad 
\| w^N_{0,p}(\cdot,\omega) \|_{H^1\left(\widetilde{Q}(\omega)\right)} 
\leq C(\omega). 
$$
This implies that, almost surely, there exists 
$w^\infty_{0,p}(\cdot,\omega) \in H^1\left(\widetilde{Q}(\omega)\right)$
such that
\begin{equation}
\label{Jikov:cond1}
w_{0,p}^N(\cdot,\omega) \rightharpoonup w^\infty_{0,p}(\cdot,\omega) 
\quad \text{weakly in $H^1\left(\widetilde{Q}(\omega)\right)$},
\end{equation}
and, using the Rellich Theorem, that
\begin{equation}
\label{L2fort}
w_{0,p}^N(\cdot,\omega) \rightarrow w^\infty_{0,p}(\cdot,\omega) 
\quad \text{strongly in $L^2\left(\widetilde{Q}(\omega)\right)$}.
\end{equation}

\bigskip

\noindent
\textbf{Step 2: $w_{0,p}^\infty$ is $\alpha Q$-periodic}

We infer from~\eqref{def:w0pN} that
$$
w_{0,p}^N\left(\frac{\phi(Ny,\omega)}{N},\omega\right) 
= 
\frac{1}{N} w_p^N(\phi(Ny,\omega),\omega) 
= 
\frac{1}{N}\widetilde{w}_p^N(Ny,\omega).
$$
Since the function $\widetilde{w}_p^N$ is $Q_N$-periodic, we see that
the function 
$\dps y \mapsto w_{0,p}^N\left(\frac{\phi(Ny,\omega)}{N},\omega\right)$
is $Q$-periodic. Hence, for any $k \in \ZZ^d$, we have, almost surely,
\begin{multline}
\label{eq:bibi}
\int_Q \left[ w_{0,p}^\infty \left(\alpha y,\omega \right) - 
w_{0,p}^\infty\left(\alpha (y + k),\omega \right) \right]^2 dy 
\leq
2\int_Q \left[ w_{0,p}^\infty\left(\alpha y,\omega \right) - 
w_{0,p}^N\left( \frac{\phi(Ny,\omega)}{N},\omega \right)
\right]^2 dy 
\\
+2 \int_Q \left[ 
w_{0,p}^N\left(\frac{\phi(N(y+k),\omega)}{N},\omega \right) - 
w_{0,p}^\infty\left(\alpha (y + k),\omega \right)
\right]^2 dy. 
\end{multline}
We now show that both terms in the above right-hand side converge to 0
when $N \to \infty$. It is sufficient to consider the first term. Let us
fix $\eta > 0$. 

We observe that the first term in the above right-hand side satisfies
\begin{equation}
\label{eq:c0c1}
\int_Q \left[ 
w_{0,p}^\infty\left(\alpha y,\omega \right) 
- 
w_{0,p}^N\left( \frac{\phi(Ny,\omega)}{N},\omega \right) \right]^2 dy 
\leq 
2 \left[ C_0^N(\omega) + C_1^N(\omega) \right],
\end{equation}
where
\begin{eqnarray*}
C_0^N(\omega) 
&=& 
\int_Q \left[ w_{0,p}^\infty\left(\alpha y,\omega \right) - 
w_{0,p}^\infty\left( \frac{\phi(Ny,\omega)}{N},\omega \right) \right]^2 dy,
\\
C_1^N(\omega) 
&=& 
\int_Q \left[ w_{0,p}^\infty\left(\frac{\phi(Ny,\omega)}{N},\omega \right) -
w_{0,p}^N \left( \frac{\phi(Ny,\omega)}{N},\omega \right) \right]^2
dy.
\end{eqnarray*}
To show that $C_0^N(\omega)$ converges to 0, we use the fact that the
function $\dps \frac{\phi(Ny,\omega)}{N}$ converges to the function
$\alpha y$ in $L^\infty_{\rm loc}(\RR^d)$ almost surely (see~\cite[Lemme
2.1]{Blanc2006}), and a regularization argument. Since 
$w_{0,p}^\infty(\cdot,\omega) \in
H^1\left(\widetilde{Q}(\omega)\right)$, there exists 
$w^\infty_\eta(\cdot,\omega) \in C^\infty\left(\widetilde{Q}(\omega)\right)$
such that 
\begin{equation}
\label{eqw0eta}
\| w^\infty_\eta(\cdot,\omega) - w_{0,p}^\infty(\cdot,\omega)  
\|_{L^2\left(\widetilde{Q}(\omega)\right)} \leq \eta.
\end{equation}
We then write that
\begin{equation}
\label{eq:decompo_C0}
C_0^N(\omega) \leq C_{00}^N(\omega) + C_{01}^N(\omega) + C_{02}^N(\omega), 
\end{equation}
where
\begin{eqnarray*}
C_{00}^N(\omega) 
&=& 
\int_Q \left[ w_{0,p}^\infty \left(\alpha y,\omega \right) - 
w_\eta^\infty \left(\alpha y,\omega \right) \right]^2 dy, 
\\
C_{01}^N(\omega) 
&=& 
\int_Q \left[ w_\eta^\infty\left(\alpha y,\omega \right) - 
w_\eta^\infty \left( \frac{\phi(Ny,\omega)}{N},\omega \right) 
\right]^2 dy, 
\\
C_{02}^N(\omega) 
&=& 
\int_Q \left[ 
w_\eta^\infty\left( \frac{\phi(Ny,\omega)}{N},\omega \right)  
-
w_{0,p}^\infty\left( \frac{\phi(Ny,\omega)}{N},\omega \right) 
\right]^2 dy.
\end{eqnarray*}
We infer from~\eqref{eqw0eta} and~\eqref{prop:ii} that 
\begin{equation}
\label{eqC00}
\forall N \in \NN^\star, \quad
C_{00}^N(\omega) 
=
\frac{1}{\sqrt{\det \alpha}}
\| w_{0,p}^\infty - w_\eta^\infty \|_{L^2(\alpha Q)}
\leq 
\frac{1}{\sqrt{\det \alpha}} \eta. 
\end{equation}
Likewise, we infer from~\eqref{eqw0eta},~\eqref{prop:i}
and~\eqref{phi:cond1} that
\begin{equation}
\label{eqC02}
\forall N \in \NN^\star, \quad
C_{02}^N(\omega) 
\leq
\frac{1}{\sqrt{\nu}}
\| w_\eta^\infty - w_{0,p}^\infty 
\|_{L^2\left(\frac{1}{N}\phi(Q_N,\omega)\right)}
\leq 
\frac{1}{\sqrt{\nu}} \eta.
\end{equation}
We now turn to $C_{01}^N(\omega)$. 
Using the fact that $w_\eta^\infty(\cdot,\omega) \in 
C^\infty\left(\widetilde{Q}(\omega)\right)$ and that the function $\dps
\frac{\phi(Ny,\omega)}{N}$ converges to the function $\alpha y$ in
$L^\infty_{\rm loc}(\RR^d)$ almost surely, we obtain that
$C_{01}^N(\omega)$ converges to zero as $N$ goes to infinity, almost
surely. We thus can choose $N(\eta,\omega) \in \NN$ such that
\begin{equation}
\label{eqC01}
\forall N \geq N(\eta,\omega), \quad
C_{01}^N(\omega) \leq \eta.
\end{equation}
Collecting~\eqref{eq:decompo_C0},~\eqref{eqC00},~\eqref{eqC01}
and~\eqref{eqC02}, we conclude that
\begin{equation}
\label{eq:c0N}
C_0^N(\omega) \to 0 \text{ as $N$ goes to infinity, almost surely.}
\end{equation}
We next turn to $C_1^N(\omega)$, which
is non-negative by definition, and satisfies,
using~\eqref{prop:i} and~\eqref{L2fort},
\begin{equation}
\label{eq:c1N}
C_1^N(\omega) 
= 
\int_{\widetilde{Q}(\omega)}
\mathbf{1}_{\frac{\phi(Q_N,\omega)}{N}} \frac{1}{\text{det}
\left(\nabla \phi\right)} \left[ 
w_{0,p}^\infty\left(y,\omega \right) 
-
w_{0,p}^N\left( y,\omega \right) \right]^2 dy 
\leq 
\frac{1}{\nu} 
\| w_{0,p}^\infty(\cdot,\omega)
-
w_{0,p}^N(\cdot,\omega) \|^2_{L^2\left(\widetilde{Q}(\omega)\right)} 
\rightarrow 0
\quad \text{as $N\to \infty$}. 
\end{equation}
Collecting~\eqref{eq:bibi},~\eqref{eq:c0c1},~\eqref{eq:c0N}
and~\eqref{eq:c1N}, we deduce that, almost surely,
$$
\forall k \in \ZZ^d, \quad
\int_Q \left[ w_{0,p}^\infty\left(\alpha y,\omega \right) - 
w_{0,p}^\infty\left(\alpha (y + k),\omega \right) \right]^2 dy 
= 0.
$$
The function $w_{0,p}^\infty$ is thus $\alpha Q$-periodic.

\bigskip

\noindent
\textbf{Step 3: $w_{0,p}^\infty$ solves 
$-\hbox{div}\left[B \nabla w_{0,p}^\infty \right] = 0$ in ${\cal
  D}'(\RR^d)$ where $B$ is a constant deterministic matrix}

In the two above steps, we closely followed the proof strategy
of~\cite{bourgeat_ihp} recalled in Section~\ref{sec:bourgeat}. This Step
3 follows a slightly different pattern, and is more involved than the
corresponding argument in~\cite{bourgeat_ihp}, which consists in showing
the weak formulation~\eqref{eq:rescale_lim}. As pointed out above, the
difficulty comes from identifying an appropriate domain, independent of $N$,
on which to apply Theorem~\ref{Jik-conv}. To circumvent this difficulty,
we work on the entire space $\RR^d$, with test functions of compact
support. 

Introduce a test function $\psi \in \mathcal{D}(\RR^d)$, and define the
$Q_N$-periodic function 
$$
\psi_N(y) := \sum\limits_{k \in \ZZ^d} \psi\left(\frac{1}{N}y - k\right).
$$
We note that, for any $y \in Q_N$, only a finite number of terms in the
above sum do not vanish, and that this number of terms only depends on
the support of $\psi$ and thus is independent of $N$. 

Choosing $\psi_N$ as test function in~\eqref{PB:corrN-per}, we write
$$
\int_{Q_N} \det (\nabla \phi(y,\omega)) 
\left(\sum\limits_{k \in \ZZ^d} 
\nabla \psi\left(\frac{1}{N}y - k\right) \right)^T 
\left(\nabla \phi(y,\omega)\right)^{-1} A_{\rm per}(y)
\left(p + \left(\nabla \phi\right)^{-T}(y,\omega) 
\nabla \widetilde{w}^N_p(y,\omega)\right) \, dy=0.
$$
After the change of variable $z=\phi(y,\omega)$, we obtain
$$
\sum\limits_{k \in \ZZ^d} 
\int_{\phi(Q_N,\omega)} 
\left( \nabla \psi\left(\frac{1}{N}\phi^{-1}(z,\omega) - k\right)
\right)^T 
\left(\nabla \phi(\phi^{-1}(z,\omega),\omega)\right)^{-1}
A_{\rm per}(\phi^{-1}(z,\omega))
\left(p + \nabla w^N_p(z,\omega)\right) \, dz=0,
$$
that we recast, using the definition~\eqref{def:w0pN} of $w^N_{0,p}$, as
\begin{equation}
\label{tentative:0}
\forall N \in \NN^\star, \quad
\sum\limits_{k \in \ZZ^d} I^N_k(\omega) =0
\quad \text{a.s.},
\end{equation}
where
$$
I^N_k(\omega) = \int_{\frac{1}{N}\phi(Q_N,\omega)}
\left( \nabla \psi\left(\frac{1}{N}\phi^{-1}(N z,\omega) - k\right)
\right)^T 
\left(\nabla \phi(\phi^{-1}(N z,\omega),\omega)\right)^{-1}
A_{\rm per}(\phi^{-1}(N z,\omega))
\left(p + \nabla w^N_{0,p}(z,\omega)\right) \, dz.
$$
We claim that
\begin{equation}
\label{eq:conv_Ik}
\forall k \in \ZZ^d, \quad
\lim_{N \to \infty} I^N_k(\omega) = I^\infty_k(\omega)
\quad \text{a.s.},
\end{equation}
where 
$$
I^\infty_k(\omega) := \int_{\alpha Q} 
\left( \nabla \psi\left(\alpha^{-1} z - k\right)\right)^T \beta \, A^\star
\left(p + \nabla w^\infty_{0,p}(z,\omega)\right) \, (\det \alpha)^{-1} \, dz,
$$
where the constant matrices $\alpha$ and $\beta$ are defined
by~\eqref{eq:def_alpha} and~\eqref{eq:def_beta}.

\medskip

Assume momentarily that~\eqref{eq:conv_Ik} indeed holds. Then, as the
sum in~\eqref{tentative:0} has a finite number of terms, independently
of $N$, we can pass to the limit $N \to \infty$ and obtain that
$$
\sum\limits_{k \in \ZZ^d} I^\infty_k(\omega) =0
\quad \text{a.s.},
$$
which also reads
$$
\sum\limits_{k \in \ZZ^d} \int_{\alpha Q} 
\left( \nabla \psi\left(\alpha^{-1} z - k\right) \right)^T \beta \, A^\star
\left(p + \nabla w^\infty_{0,p}(z,\omega)\right) (\det \alpha)^{-1} dz = 0.
$$
Using the $\alpha Q$-periodicity of the function $w_{0,p}^\infty$ (shown
in the above Step 2), we deduce that
$$
\int_{\RR^d} \left( \nabla \psi(z) \right)^T \beta \, A^\star
\left(p + \nabla w^\infty_{0,p}(\alpha z,\omega) \right) dz = 0
$$
for all test functions $\psi \in \mathcal{D}(\RR^d)$. We indeed have
shown that
\begin{equation}
\label{eq:conc_step3}
-\mbox{div}\left[ \beta \, A^\star 
\nabla w^\infty_{0,p}(\alpha \cdot,\omega) \right] = 0 
\mbox{ in $\mathcal{D}'(\RR^d)$}. 
\end{equation}

\medskip

To conclude this Step, we are hence left with showing~\eqref{eq:conv_Ik}.
Formally, this comes from the strong $L^1(\RR^d)$ convergence of the
indicator function $\mathbf{1}_{\frac{1}{N}\phi(Q_N,\omega)}$ towards
$\mathbf{1}_{\alpha Q}$ 
and from the div-curl lemma. We indeed observe
that the integrand in $I_k^N(\omega)$ is the product of
$\dps \left(\nabla \phi(\phi^{-1}(N z,\omega),\omega)\right)^{-T}
\nabla \psi\left(\frac{1}{N}\phi^{-1}(N z,\omega) - k\right)$ 
with
$\dps A_{\rm per}(\phi^{-1}(N z,\omega)) 
\left(p + \nabla w^N_{0,p}(z,\omega)\right)$. We will show in the sequel
that the first factor is curl-free, whereas the second factor is
divergence free.
Using the div-curl lemma, this product
converges (at least in the sense of distributions) towards the product
of the weak limits of the two factors, which can be identified. 
One difficulty to make this argument rigorous is to find a fixed domain
(independent of $N$) on which 
to apply the div-curl lemma. For that purpose, Lemma~\ref{theo:phiQN}
is useful. 

We now proceed in details.
Let $\eta > 0$, and let ${\cal O}_\eta \subset \widetilde{\cal O}_\eta$
be two deterministic open
sets such that $\overline{\widetilde{\cal O}_\eta}$ is proper subset of
$\alpha Q$, $\overline{{\cal O}_\eta}$ is a proper subset of
$\widetilde{\cal O}_\eta$, and 
\begin{equation}
\label{eq:approx_set}
\left| \alpha Q \setminus \widetilde{\cal O}_\eta \right| \leq \eta,
\qquad
\left| \widetilde{\cal O}_\eta \setminus {\cal O}_\eta \right| \leq \eta.
\end{equation}
We then decompose $I^N_k(\omega)$ and $I^\infty_k(\omega)$ as follows:
using~\eqref{prop:i} and~\eqref{prop:ii}, we write
\begin{equation}
\label{eq:decompo1}
I^N_k(\omega) = I^N_{k,\eta}(\omega) + \mathcal{R}_{k,\eta}^N(\omega),
\qquad
I^\infty_k(\omega) = 
I^\infty_{k,\eta}(\omega) + \mathcal{R}_{k,\eta}^\infty(\omega),
\end{equation}
with
\begin{eqnarray*}
I^N_{k,\eta}(\omega) 
&=& 
\int_{{\cal O}_\eta}\left( \nabla 
\psi\left(\frac{1}{N}\phi^{-1}(N z,\omega) - k\right) \right)^T 
\left(\nabla \phi(\phi^{-1}(N z,\omega),\omega)\right)^{-1}
A_{\rm per}\left(\phi^{-1}(N z,\omega)\right)
\left(p + \nabla w^N_{0,p}(z,\omega)\right) \, dz, 
\\
I^\infty_{k,\eta}(\omega)  
&=& 
\int_{{\cal O}_\eta} \left( \nabla \psi\left(\alpha^{-1} z - k\right)\right)^T 
\beta \,
A^\star \left(p + \nabla w^\infty_{0,p}(z,\omega) \right) \, 
(\det \alpha)^{-1} \, dz, 
\\
\mathcal{R}_{k,\eta}^N(\omega) 
&=& 
\int_{\widetilde{Q}(\omega)} 
\left(\mathbf{1}_{\frac{1}{N}\phi(Q_N,\omega)}(z) - 
\mathbf{1}_{{\cal O}_\eta}(z)\right)
\left( \nabla \psi\left(\frac{1}{N}\phi^{-1}(N z,\omega) - k\right)
\right)^T 
\\
&& \quad\quad\quad\quad \times
\left(\nabla \phi(\phi^{-1}(N z,\omega),\omega)\right)^{-1}
A_{\rm per}\left(\phi^{-1}(N z,\omega)\right)
\left(p + \nabla w^N_{0,p}(z,\omega)\right) \, dz, 
\\
\mathcal{R}^\infty_{k,\eta}(\omega) 
&=& 
\int_{\alpha Q \setminus {\cal O}_\eta} \left(\nabla \psi(\alpha^{-1} z -
  k) \right)^T \beta \, A^\star (p + \nabla w_{0,p}^\infty(z,\omega))
\ (\det \alpha)^{-1} \, dz.
\end{eqnarray*}
To use the div-curl lemma, we need to further decompose $I^N_{k,\eta}(\omega)$
and $I^\infty_{k,\eta}(\omega)$. Introducing a smooth truncation
function $\xi \in \mathcal{D}\left(\widetilde{\cal O}_\eta\right)$ such that $0 \leq
\xi(x) \leq 1$ a.e. and $\xi \equiv 1$ on ${\cal O}_\eta$, we write that
\begin{equation}
\label{def:tildrest}
I^N_{k,\eta}(\omega) = \widetilde{I}^N_{k,\eta}(\omega) - C^N_{\eta}(\omega), 
\qquad
I^\infty_{k,\eta}(\omega) = 
\widetilde{I}^\infty_{k,\eta}(\omega) - C^\infty_{\eta}(\omega),
\end{equation}
where 
\begin{eqnarray*} 
C^N_{\eta}(\omega) 
&=& 
\int_{\widetilde{\cal O}_\eta \setminus {\cal O}_\eta} \xi(z) 
\left( \nabla \psi\left(\frac{1}{N}\phi^{-1}(N z,\omega) - k\right) \right)^T
\left(\nabla \phi(\phi^{-1}(N z,\omega),\omega)\right)^{-1} 
A_{\rm per}\left(\phi^{-1}(N z,\omega)\right)
\left(p + \nabla w^N_{0,p}(z,\omega) \right) dz, 
\\
C^\infty_{\eta}(\omega) 
&=& 
\int_{\widetilde{\cal O}_\eta \setminus {\cal O}_\eta} \xi(z) 
\left( \nabla \psi\left(\alpha^{-1} z - k\right) \right)^T \beta \, A^\star
\left(p + \nabla w^\infty_{0,p}(z,\omega)\right) \, (\det \alpha)^{-1} \, dz,
\\
\widetilde{I}^N_{k,\eta}(\omega) 
&=& 
\int_{\widetilde{\cal O}_\eta} \xi(z) 
\left( \nabla \psi\left(\frac{1}{N}\phi^{-1}(N z,\omega) - k\right) 
\right)^T \left(\nabla \phi(\phi^{-1}(N z,\omega),\omega)\right)^{-1}
A_{\rm per}\left(\phi^{-1}(N z,\omega)\right)
\left(p + \nabla w^N_{0,p}(z,\omega) \right) \, dz,
\\
\widetilde{I}^\infty_{k,\eta}(\omega) 
&=& 
\int_{\widetilde{\cal O}_\eta} \xi(z) 
\left( \nabla \psi\left(\alpha^{-1} z - k\right) \right)^T \beta \, A^\star 
\left(p + \nabla w^\infty_{0,p}(z,\omega) \right) \, (\det \alpha)^{-1} \, dz.
\end{eqnarray*}

We first bound from above $C^N_{\eta}(\omega)$, $C^\infty_{\eta}(\omega)$, 
$\mathcal{R}^N_{k,\eta}(\omega)$ and 
$\mathcal{R}^\infty_{k,\eta}(\omega)$.
As $|\xi| \leq 1$, we see that
\begin{eqnarray*}
\left| C^N_{\eta}(\omega) \right|
&\leq& 
\| \nabla \psi \|_{L^\infty} \, \| \nabla \phi^{-1}(\cdot,\omega) \|_{L^\infty}
\, \| A_{\rm per} \|_{L^\infty} \ 
\left| \widetilde{\cal O}_\eta \setminus {\cal O}_\eta \right|^{1/2} 
\ \| p + \nabla w^N_{0,p}(\cdot,\omega) 
\|_{L^2\left(\widetilde{\cal O}_\eta \setminus {\cal O}_\eta\right)}
\\
&\leq& 
C \sqrt{\eta}
\ \| p + \nabla w^N_{0,p}(\cdot,\omega) 
\|_{L^2\left(\widetilde{Q}(\omega)\right)}
\end{eqnarray*}
where, in the second line, we have
used~\eqref{eq:utile},~\eqref{eq:approx_set} and the fact
that $\widetilde{\cal O}_\eta \setminus {\cal O}_\eta \subset
\widetilde{\cal O}_\eta 
\subset \alpha Q \subset \widetilde{Q}(\omega)$
(see~\eqref{prop:ii}). Now
using~\eqref{bornw0pN_bis}, we deduce that there exists $C(\omega)$,
independent of $\eta$ and $N$, such that
\begin{equation}
\label{eq:flor4}
\forall N \in \NN^\star, \quad
\left| C^N_{\eta}(\omega) \right|
\leq C(\omega) \sqrt{\eta}.
\end{equation}
We likewise obtain that
\begin{equation}
\label{tentative:3}
\left| C^\infty_{\eta}(\omega) \right|
\leq C(\omega) \sqrt{\eta}
\quad \text{and} \quad
\left| \mathcal{R}^\infty_{k,\eta}(\omega) \right| 
\leq C(\omega) \sqrt{\eta}.
\end{equation}
Now turning to $\mathcal{R}^N_{k,\eta}(\omega)$, we obtain, using
similar arguments, that
$$
\left| \mathcal{R}^N_{k,\eta}(\omega) \right|
\leq
C \left\| \mathbf{1}_{\frac{1}{N}\phi(Q_N,\omega)} - 
\mathbf{1}_{{\cal O}_\eta} \right\|_{L^2\left(\widetilde{Q}(\omega)\right)}
\left\| p + \nabla w^N_{0,p}(\cdot,\omega) 
\right\|_{L^2\left(\widetilde{Q}(\omega)\right)}.
$$
Using~\eqref{bornw0pN_bis}, a triangle inequality
and~\eqref{eq:approx_set}, we deduce that
\begin{equation}
\label{tentative:2}
\left| \mathcal{R}^N_{k,\eta}(\omega) \right|
\leq
C(\omega) \left( 
\left\| \mathbf{1}_{\frac{1}{N}\phi(Q_N,\omega)} - 
\mathbf{1}_{\alpha Q} \right\|_{L^2\left(\widetilde{Q}(\omega)\right)}
+
\sqrt{\eta}
\right).
\end{equation}
Recall now that, in view of~\eqref{eq:reecriture_cras2}, we have
\begin{equation}
\label{eq:2bis}
\lim_{N \to \infty}
\left\| \mathbf{1}_{\frac{1}{N}\phi(Q_N,\omega)} - 
\mathbf{1}_{\alpha Q} \right\|_{L^2\left(\widetilde{Q}(\omega)\right)}
= 0 \ \ \text{a.s.}
\end{equation}
We eventually estimate 
$\widetilde{I}^N_{k,\eta}(\omega) -
\widetilde{I}^\infty_{k,\eta}(\omega)$ using the div-curl lemma.
The compact $\overline{\widetilde{\cal O}_\eta}$ being a proper subset of
$\alpha Q$, we infer from Lemma~\ref{theo:phiQN} that there exists
$N_0(\omega)$ such that for all $N \geq N_0(\omega)$, 
$\dps \widetilde{\cal O}_\eta \subset \frac{1}{N}\phi(Q_N,\omega)$. We then
deduce from~\eqref{PB:corrN-per} that, for any $N \geq N_0(\omega)$,
\begin{equation}
\label{div-curl:1}
-\hbox{div}\left[A_{\rm per}\left(\phi^{-1}(N z,\omega)\right)
\left(p + \nabla w^N_{0,p}(z,\omega)\right)\right]=0 \quad 
\text{ in $\mathcal{D}'\left(\widetilde{\cal O}_\eta\right)$.}
\end{equation}
Using~\eqref{Jikov:cond1}, we can thus apply Theorem~\ref{Jik-conv} on
the domain $\widetilde{\cal O}_\eta$, and obtain that
\begin{equation}
\label{div-curl:2}
A_{\rm per}\left(\phi^{-1}(N z,\omega)\right)
\left(p + \nabla w^N_{0,p}(z,\omega)\right) 
\rightharpoonup 
A^\star \left(p + \nabla w^\infty_{0,p}(z,\omega)\right) 
\quad \text{weakly in $L^2\left(\widetilde{\cal O}_\eta\right)$}.
\end{equation}
From the proof of~\cite[Lemme 2.2]{Blanc2006}, we know that 
$\dps \frac{1}{N}\phi^{-1}(N z,\omega)$ strongly converges in
$L^\infty_{\rm loc}(\RR^d)$ toward $\alpha^{-1} z$. 
As, by definition,
$\nabla \psi \in \mathcal{C}^\infty(\RR^d)$, we obtain that 
\begin{equation}
\label{eq:flor5}
\nabla \psi\left(\frac{1}{N}\phi^{-1}(N z,\omega) - k\right) 
\to 
\nabla \psi\left(\alpha^{-1}z - k\right) 
\quad \text{strongly in $L_{\rm loc}^\infty(\RR^d)$}.
\end{equation}
Since $(\nabla \phi)^{-1}$ is stationary, we infer from \cite[Lemme
2.2]{Blanc2006} that
\begin{equation}
\label{eq:flor6}
\left(\nabla \phi(\phi^{-1}(N z,\omega),\omega)\right)^{-1} 
\rightharpoonup 
(\det \alpha)^{-1} \, \beta \quad \text{weakly-$\star$ in $L^\infty(\RR^d)$},
\end{equation}
where the matrix $\beta$ is defined by~\eqref{eq:def_beta}.
As $\widetilde{\cal O}_\eta$ is a bounded open set of $\RR^d$, we deduce
from~\eqref{eq:flor5} and~\eqref{eq:flor6} that 
\begin{equation}
\label{div-curl:3}
\left(\nabla \psi\left(\frac{1}{N}\phi^{-1}(N z,\omega) - k\right) \right)^T 
\left(\nabla \phi\left(\phi^{-1}(N z,\omega),\omega\right)\right)^{-1} 
\rightharpoonup 
(\det \alpha)^{-1} \left( \nabla \psi\left(\alpha^{-1}z - k\right)
\right)^T \beta
\quad \text{weakly in $L^2\left(\widetilde{\cal O}_\eta\right)$.}
\end{equation}
We eventually note that
\begin{equation}
\label{eq:flor7}
\left( \nabla \phi(\phi^{-1}(N z,\omega),\omega) \right)^{-T} \
\nabla \psi\left(\frac{1}{N}\phi^{-1}(N z,\omega) - k\right)
\text{ is curl-free},
\end{equation}
as this vector is the gradient of $\dps
\psi\left(\frac{1}{N}\phi^{-1}(N z,\omega) - k\right)$.
Collecting~\eqref{div-curl:1},~\eqref{div-curl:2},~\eqref{div-curl:3}
and~\eqref{eq:flor7}, we are in position to apply the div-curl lemma
(see for instance~\cite[Lemma 1.1 p. 4]{jikov}). We thus obtain that
\begin{equation}
\label{conv0}
\forall \eta, \quad
\lim_{N \to \infty} \widetilde{I}^N_{k,\eta}(\omega) 
= 
\widetilde{I}^\infty_{k,\eta}(\omega) 
\quad \text{a.s.}
\end{equation}
Collecting~\eqref{eq:decompo1},~\eqref{def:tildrest},~\eqref{eq:flor4},~\eqref{tentative:3},~\eqref{tentative:2},~\eqref{eq:2bis}
and~\eqref{conv0}, we deduce the claim~\eqref{eq:conv_Ik}. This
concludes this Step. 

\bigskip

\noindent
\textbf{Step 4: Conclusion}

Collecting the conclusion of Step 2 and~\eqref{eq:conc_step3}, we have
shown that the function $w^\infty_{0,p}(\cdot,\omega)$ solves the problem
$$
-\mbox{div}\left[ \beta \, A^\star 
\nabla w^\infty_{0,p}(\alpha \cdot,\omega) \right] = 0 
\mbox{ in $\mathcal{D}'(\RR^d)$}, 
\qquad
w^\infty_{0,p}\left(\alpha \cdot,\omega \right) \mbox{ is $Q$-periodic}.
$$
The function $g(x,\omega) = w^\infty_{0,p}\left(\alpha x,\omega \right)$
is thus $Q$-periodic and satisfies 
$$
-\mbox{div}\left[ \beta \, A^\star \alpha^{-T} 
\nabla g(\cdot,\omega) \right] = 0 
\mbox{ in $\mathcal{D}'(\RR^d)$}.
$$
We know from Lemma~\ref{lem:j_ball} that the matrix $\beta \, A^\star
\alpha^{-T}$ is coercive. The above equation has thus a unique solution
(up to the addition of a random constant), hence $\nabla g \equiv 0$,
which implies that $\nabla w^\infty_{0,p} \equiv 0$. 
We thus deduce from~\eqref{div-curl:2} that
\begin{equation}
\label{eq:1bis}
A_{\rm per}\left(\phi^{-1}(N\cdot,\omega)\right) 
\left(p + \nabla w^N_{0,p}(\cdot,\omega) \right) 
\rightharpoonup 
A^\star p 
\quad \text{as $N \to \infty$, weakly in $L^2\left({\cal O}_\eta\right)$.}
\end{equation}
We are now in position to prove the convergence of the approximation
described in Section~\ref{sec:pres_approx}. We infer from~\eqref{eq:-1}
that 
\begin{equation}
\label{eq:0bis}
\left[ B^\star_N \right]_{ij}(\omega) 
= 
R^N_{1,\eta}(\omega) + R^N_{2,\eta}(\omega),
\end{equation}
with
\begin{eqnarray*}
R^N_{1,\eta}(\omega) 
&=& 
\int_{\widetilde{Q}(\omega)} 
\left(\mathbf{1}_{\frac{1}{N}\phi(Q_N,\omega)}(x) -
  \mathbf{1}_{{\cal O}_\eta}(x)\right) 
e_i^T A_{\rm per}\left(\phi^{-1}(Nx,\omega)\right) 
\left(e_j + \nabla w^N_{0,e_j}(x,\omega)\right) \, dx, 
\\ 
R^N_{2,\eta}(\omega) 
&=& 
\int_{{\cal O}_\eta} e_i^T A_{\rm per}\left(\phi^{-1}(Nx,\omega)\right) 
\left(e_j + \nabla w^N_{0,e_j}(x,\omega)\right) \, dx,
\end{eqnarray*}
where we have used that ${\cal O}_\eta \subset \alpha Q$,~\eqref{prop:i}
and~\eqref{prop:ii}. We deduce from~\eqref{eq:1bis} that
$$
\forall \eta, \quad
\lim_{N \to \infty} R^N_{2,\eta}(\omega)
=
\left| {\cal O}_\eta \right| \left[ A^\star \right]_{ij}
\quad \text{a.s.},
$$
hence, in view of~\eqref{eq:approx_set},
\begin{equation}
\label{eq:flor8}
\lim_{\eta \to 0}
\lim_{N \to \infty} R^N_{2,\eta}(\omega)
=
\left| \alpha Q \right| \left[ A^\star \right]_{ij}
\quad \text{a.s.}
\end{equation}
Turning to $R^N_{1,\eta}(\omega)$, we deduce from~\eqref{bornw0pN_bis}
and~\eqref{eq:approx_set} that
$$
\left| R^N_{1,\eta}(\omega) \right| 
\leq 
C(\omega) 
\left\| \mathbf{1}_{\frac{1}{N}\phi(Q_N,\omega)} - 
\mathbf{1}_{{\cal O}_\eta} \right\|_{L^2\left(\widetilde{Q}(\omega)\right)}
\leq
C(\omega) 
\left( 
\left\| \mathbf{1}_{\frac{1}{N}\phi(Q_N,\omega)} - 
\mathbf{1}_{\alpha Q} \right\|_{L^2\left(\widetilde{Q}(\omega)\right)}
+
\sqrt{\eta}
\right),
$$
hence, in view of~\eqref{eq:2bis},
\begin{equation}
\label{eq:flor9}
\lim_{\eta \to 0}
\lim_{N \to \infty} R^N_{1,\eta}(\omega)
=
0 \quad \text{a.s.}
\end{equation}
Collecting~\eqref{eq:0bis},~\eqref{eq:flor8} and~\eqref{eq:flor9}, we obtain
$$
\lim_{N \to \infty} \left[ B^\star_N \right]_{ij}(\omega)
=
\left| \alpha Q \right| \, A^\star_{ij}
\quad \text{a.s.}
$$
We then deduce from~\eqref{def:AstarN} the claimed convergence. This
concludes the proof of Theorem~\ref{conv:AN}.

\subsection{Proof of Lemma~\ref{lem:j_ball}}
\label{sec:proof_j_ball}

We first show that
\begin{equation}
\label{eq:ball1} 
\text{the homogenized matrix $A^\star$ defined by~\eqref{def:astar} is
  coercive.} 
\end{equation}
For any $p \in \RR^d$, we indeed have 
\begin{eqnarray*}
p^T A^\star p 
&=&
(\det \alpha)^{-1} \ \EE \left(\int_{\phi(Q,\cdot)} p^T 
A_{\rm per}\left(\phi^{-1}\left(y,\cdot\right)\right) 
\left( p + \nabla w_p(y,\cdot) \right) dy \right)
\\
&=&
(\det \alpha)^{-1} \ \EE \left(\int_{\phi(Q,\cdot)} 
\left(p + \nabla w_p(y,\cdot) \right)^T 
A_{\rm per}\left(\phi^{-1}\left(y,\cdot\right)\right) 
\left( p + \nabla w_p(y,\cdot) \right) dy \right)
\end{eqnarray*}
where the last line is obtained using the arguments presented in the
existence proof of~\cite[Théorème 1.2]{Blanc2006}. The matrix $A_{\rm
  per}$ being coercive (see~\eqref{eq:hyp_a}), we deduce that there
exists $C>0$ such that, for any $p \in \RR^d$,
\begin{eqnarray*}
p^T A^\star p 
& \geq &
a_- (\det \alpha)^{-1} \ \EE \left(\int_{\phi(Q,\cdot)} 
\left(p + \nabla w_p(y,\cdot) \right)^T 
\left( p + \nabla w_p(y,\cdot) \right) dy \right)
\\
& \geq &
C \ \left[ \EE \left(\int_{\phi(Q,\cdot)} p + \nabla w_p(y,\cdot)
  \right) \right]^2 
\quad \text{(Cauchy-Schwarz inequality)}
\\
& \geq &
C \ p^T p \quad \text{(using third line of~\eqref{PB:correc})}.
\end{eqnarray*}
This proves~\eqref{eq:ball1}.

\medskip

We now claim that the matrix $\beta$ defined by~\eqref{eq:def_beta} satisfies
\begin{equation}
\label{eq:ball2} 
\beta = \det \alpha \ \alpha^{-1}. 
\end{equation}
This is obvious in dimension $d=1$, and also in dimension $d=2$, using
the explicit formula of the inverse of a $2 \times 2$ matrix. In
dimension $d \geq 3$, we observe that $\dps \beta = \EE \left[ \int_Q
  {\rm adj} \, \nabla \phi \right]$, where ${\rm adj} \, \nabla \phi$ is
the adjugate 
matrix (i.e. the transpose of the matrix of cofactors) of $\nabla
\phi$. The matrix $\nabla \phi$ being stationary, we deduce
from~\cite[Corollary~1]{henao} and~\eqref{eq:signe} (see
also~\cite[Corollary 6.2.2]{ball1977} for the specific case $d=3$) that 
$\dps \EE \left[ \int_Q {\rm adj} \, \nabla \phi \right] 
= 
{\rm adj} \ \EE \left[ \int_Q \nabla \phi \right]
$, from which we readily infer~\eqref{eq:ball2}. 

\medskip

We are now in position to prove Lemma~\ref{lem:j_ball}.
Using~\eqref{eq:ball2} and~\eqref{eq:ball1}, we indeed see
that there exists $C>0$ such that, for any $p \in \RR^d$, we have 
$$
p^T \beta A^\star \alpha^{-T} p 
=
\det \alpha \ p^T \alpha^{-1} A^\star \alpha^{-T} p 
\geq
C p^T \alpha^{-1} \alpha^{-T} p.
$$
Since $\det \alpha > 0$, we see that the matrix $\alpha^{-1} \alpha^{-T}$
is symmetrix positive definite, which concludes the proof of Lemma~\ref{lem:j_ball}. 

\bigskip

\noindent
\textbf{Acknowledgements:} We thank Claude Le Bris and Xavier
Blanc for stimulating discussions and useful comments on a preliminary
version of this article. We are thankful to Xavier Blanc for pointing
out the reference~\cite{henao}.
This work is partially supported by ONR under
Grant N00014-12-1-0383 and by EOARD under Grant FA8655-10-C-4002.

\end{document}